\documentclass[12pt]{article}
\usepackage[utf8]{inputenc}
\usepackage{geometry}
\usepackage{amsthm}
\usepackage{mathtools}
\usepackage{hyperref}
\usepackage{graphicx}
\usepackage{newtxmath}
\usepackage{subcaption}
\usepackage{float}
\usepackage[dvipsnames]{xcolor}
\usepackage[linesnumbered, boxruled, vlined]{algorithm2e}
\usepackage{cleveref}
\usepackage{tikz,tkz-euclide}
\usepackage{siunitx}
\usepackage{pgfplots}

\crefname{equation}{}{}
\Crefname{equation}{}{}

\usepackage[
backend=biber,
style=alphabetic,
citestyle=numeric,
bibstyle=numeric,
sorting=nyt,
maxbibnames=99
]{biblatex}
\newtheorem{theorem}{Theorem}[section]
\newtheorem*{theorem*}{Theorem}
\newtheorem{proposition}[theorem]{Proposition}
\newtheorem*{proposition*}{Proposition}

\newtheorem*{corollary*}{Corollary}
\newtheorem{lemma}[theorem]{Lemma}
\newtheorem*{lemma*}{Lemma}

\newtheorem{definition}[theorem]{Definition}
\newtheorem*{definition*}{Definition}

\newtheorem{remark}[theorem]{Remark}
\newtheorem*{remark*}{Remark}

\newcommand{\comment}[1]{}
\newcommand{\R}{\mathbb{R}}

\newcommand{\Ker}[1]{\mathrm{Ker}\left(#1 \right)}

\newcommand{\Triu}{ \mathrm{ut} \left( ( \Omega U_m)^t \Omega \right)}

\newcommand{\fl}{ \mathbf{fl} }

\newcommand{\F}{\mathbf{F}}
\newcommand{\e}{\mathbf{e}}
\newcommand{\f}{\mathbf{f}}

\SetAlgoSkip{bigskip}
\SetCommentSty{mycommfont}
\SetKwInput{KwInput}{Input}                
\SetKwInput{KwOutput}{Output}

\DeclareMathOperator*{\argmin}{arg\,min}

\title{Randomized Householder QR}
\author{Laura Grigori\thanks{Institute of Mathematics, EPFL and Scientific Computing Division, Paul Scherrer Institute, Switzerland.}, Edouard Timsit\thanks{Sorbonne Université, Inria, CNRS, Université de Paris, Laboratoire Jacques-Louis Lions, Paris, France}}
\date{May 22nd, 2024}

\begin{document}

\maketitle

\textbf{Abstract: } This paper introduces a randomized Householder QR factorization (RHQR). This factorization can be used to obtain a well conditioned basis of a vector space and thus can be employed in a variety of applications. The RHQR factorization of the input matrix $W$ is equivalent to the standard Householder QR factorization of matrix $\Psi W$, where $\Psi$ is a sketching matrix that can be obtained from any subspace embedding technique. For this reason, the RHQR factorization can also be reconstructed from the Householder QR factorization of the sketched problem, yielding a single-synchronization randomized QR factorization (recRHQR). In most contexts, left-looking RHQR requires a single synchronization per iteration, with half the computational cost of Householder QR, and a similar cost to Randomized Gram-Schmidt (RGS) overall.  
We discuss the usage of RHQR factorization in the Arnoldi process and then in GMRES, showing thus how it can be used in Krylov subspace methods to solve systems of linear equations. Based on Charles Sheffield's connection between Householder QR and Modified Gram-Schmidt (MGS), a BLAS2-RGS is also derived. A finite precision analysis shows that, under mild probabilistic assumptions, the RHQR factorization of the input matrix $W$ inherits the stability of the Householder QR factorization, producing a well-conditioned basis and a columnwise backward stable factorization, all independently of the condition number of the input $W$, and with the accuracy of the sketching step. We study the subsampled randomized Hadamard transform (SRHT) as a very stable sketching technique.

Numerical experiments show that RHQR produces a well conditioned basis whose sketch is numerically orthogonal even for the most difficult inputs, and an accurate factorization. The same results were observed with the high-dimensional operations made in half-precision. The reconstructed RHQR from the HQR factorization of the sketch was stabler than the standard Randomized Cholesky QR. 

The first version of this work was made available on HAL on the 7th of July 2023 and can be found at : https://hal.science/hal-04156310/

\section{Introduction}
Computing the QR factorization of a matrix $W \in \R^{n \times m}$, $m \ll n$, is a process that lies at the heart of many linear algebra algorithms, for instance for solving least-squares problems or computing bases in Krylov subspace methods. Such process outputs the factorization
\begin{align*} 
W = QR, \quad Q \in \R^{n \times m}, \quad Q^tQ = I_m, \quad R \in \R^{m \times m}, \quad R \text{ upper triangular. } 
\end{align*}
The QR factorization of $W \in \R^{n \times m}$ is usually obtained using the Gram-Schmidt process or the Householder process. The Householder process relies on \textit{Householder vectors} $u_1, \cdots, u_m \in \R^n$ and orthogonal matrices called \textit{Householder reflectors} $P(u_1), \cdots, P(u_m) \in \R^{n \times n}$ such that
\begin{align*} W = P(u_1) \cdots P(u_m)\left[ \begin{matrix} R \\ 0_{(n-m) \times m} \end{matrix} \right] = QR, \quad Q = P(u_1) \cdots P(u_m) \cdot \begin{bmatrix} I_m \\ 0_{(n-m) \times m} \end{bmatrix}, \quad Q^t Q = I_m, \end{align*}
where $R \in \R^{m \times m}$ is an upper-triangular factor. It was shown in~\cite{schriebervanloan} that the composition of Householder reflectors admits the factorization
\begin{align*}P(u_1) \cdots P(u_m) = I_n - U T U^t, \quad P(u_m) \cdots P(u_1) = I_n - U T^t U^t\end{align*}
where $U = \left[ u_1 \; | \; \cdots \; | \; u_m \right] \in \R^{n \times m}$ is formed by the Householder vectors (and is lower-triangular by design, see~\Cref{section:preliminaries}), and $T \in \R^{m \times m}$ is an upper-triangular factor. Applying the Woodburry-Morrison formula to $I_n - U T U^t$, it was outlined in~\cite{puglisi} that
\begin{align*}  U^t U = T^{-t} + T^{-1}. \end{align*}
All equations above hold in exact arithmetics. J.H. Wilkinson showed in~\cite{wilkinson} that the Householder QR factorization process is normwise backward stable in finite precision. Later on, it has been shown in~\cite[Chapter 19.3, Thm 19.4 p360]{higham} that the Householder procedure is column-wise backward stable in finite precision. Overall, computations with Householder reflectors are well-known for their excellent numerical stability. 
An idea from Charles Sheffield and praised by Gene Golub led to the analysis in~\cite{paige} revealing that the $T$ factor produced by the Householder QR factorization of $\left[0_{m \times m}; \; W\right] \in \R^{(n+m) \times m}$ captures the loss of orthogonality of the Modified Gram-Schmidt procedure on $W$ in finite precision arithmetics. Accordingly, a BLAS2 implementation of MGS was derived in~\cite{barlow} (see a comparison of different implementations of Gram-Schmidt procedures in the recent review~\cite{carson}).

The present work focuses on the use of the $\epsilon$-embedding technique, which has proved to be an effective solution for reducing the communication and computational cost of several classic operations, while providing quasi-optimal results. When solving large scale linear algebra problems on a parallel computer, the communication is the limiting factor preventing scalability to a large number of processors. For an $m$ dimensional vector subspace of interest $\mathcal{W} \in \R^n$, some positive real number $\epsilon < 1$, and integer $m < \ell \ll n $, a \textit{sketching matrix} $\Omega \in \R^{\ell \times n}$ is said to be an $\epsilon$-embedding of $\mathcal{W}$ if
\begin{align} \label{eq:embeddingprop}\forall w \in \mathcal{W}, \quad \left| \| \Omega w \| - \|w \| \right| \leq \epsilon \|w\|.\end{align}
If $\ell$ is a modest multiple of $m$, there exist simple distributions on $\R^{\ell \times n}$ from which one can easily draw $\Omega \in \R^{\ell \times m}$ independently of $\mathcal{W}$ and verify~\cref{eq:embeddingprop} with high probability. Furthermore, some distributions allow the sketching operation $x \mapsto \Omega x$ to be done for a quasi linear cost $\mathcal{O}(n \log (n))$ and negligible storage cost of $\Omega$ (see~\cite{ailon2009,sparsejlt}). Several methods have been proposed in order to approximate a least squares problem, i.e finding the minimizer of $x \mapsto \| W x - b \|_2$ using the solution of the cheaper \textit{sketched least squares problem}, i.e finding the minimizer of $x \mapsto \| \Omega (W x - b) \|_2$. In the pioneering work~\cite{rokhlintygert}, the authors propose to first compute $\Omega W$, then compute its rank revealing QR factorization to solve the sketched least squares problem. The minimizer and the computed R factor are then used respectively as a starting point and a preconditioner for the conjugate gradient method approximating the initial least squares problem. Later, an alternative randomized Gram-Schmidt (RGS) process was introduced in~\cite{rgs}, where each vector of the sketched basis is obtained by sketching the corresponding basis vector right after its orthogonalization step. For a given basis $W$ of a vector subspace of interest and an $\epsilon$-embedding $\Omega \in \R^{\ell \times m}$ of $\mathrm{Range}(W)$, the RGS algorithm produces, in exact arithmetics, the factorization
$$W = QR, \quad Q \in \R^{n \times m}, \quad (\Omega Q)^t \Omega Q = I_m, \quad \mathrm{Cond}(Q) \leq \frac{1+\epsilon}{1-\epsilon}, \quad R \in \R^{m \times m}, \quad R \text{ upper-triangular },$$
with the same communication cost as Classical Gram-Schmidt (CGS), similar numerical stability to MGS (sometimes even better), and half their computational cost. This \textit{sketch orthogonal basis} $Q$ can then be used as a well conditioned basis of $\mathrm{Range}(W)$, also allowing to approximate the least squares problem. It was shown in~\cite{rgs}, among other things, that this basis could be used in the Arnoldi iteration and GMRES allowing to obtain a quasi-optimal solution. The authors however identified very difficult cases in which both RGS and MGS experience instabilities in finite precision, with the sketch basis $\Omega Q$ losing orthogonality and $\mathrm{Cond}(Q) \approx 10^2$. Authors in~\cite{yujijoel} also use randomization in Krylov subspaces methods, combined with partial orthogonalization and then whitening of the Krylov basis. Other randomized algorithms for computing a well conditioned basis are studied in~\cite{balabanovqr,multisketching,melnichenko} relying on techniques as column pivoting or multi-sketching. 
An overview of randomized algorithms and sketching techniques is given in the following reviews and books~\cite{findingstructure,foundations,woodruff,randlapack}.

In this work we introduce a randomized version of the Householder QR factorization (RHQR). By using a sketching matrix $\Psi \in \R^{(\ell+m)\times n}$, which can be obtained from any $\epsilon$-embedding matrix $\Omega \in \R^{\ell \times n}$, this RHQR factorization relies on \textit{randomized Householder vectors} $u_1, \hdots u_m \in \R^n$ and upper-triangular $T$ factor, and verifies in exact arithmetics
\begin{align} \label{intro:equivfacto} \Psi \Bigl( I_n - U T (\Psi U)^t \Psi \Bigr) \cdot \begin{bmatrix} R \\ 0_{(n-m) \times m} \end{bmatrix} = \Bigl( I_{\ell+m} - (\Psi U) T (\Psi U)^t \Bigr) \cdot \begin{bmatrix} R \\ 0_{\ell \times m} \end{bmatrix}, \quad (\Psi U)^t \Psi U = T^{-t} + T^{-1}, \end{align}
where the left-hand side is the sketch of the RHQR factorization of $W$, and the right-hand side is the standard Householder QR factorization of the sketch $\Psi W$. This equation shows that in exact precision, the sketch of the basis output by RHQR is orthogonal, i.e
\begin{align*} Q := \Bigl( I_n - U T (\Psi U)^t \Psi \Bigr) \begin{bmatrix} I_n \\ 0_{(n-m) \times m} \end{bmatrix}, \quad W = QR, \quad (\Psi Q)^t \Psi Q = I_{m}, \quad \mathrm{Cond}(Q) \leq \frac{1+\epsilon}{1-\epsilon}.\end{align*}
We also show that the randomized Householder vectors (and thus the whole RHQR factorization of $W$) can be reconstructed from the Householder QR factorization of $\Psi W$, which outputs the first $m$ rows of $U$ and matrices $T, \Psi U$ verifying in exact precision
\begin{align} \label{intro:reconstructrhqr}
 W_{m+1:m, 1:m} = U_{m+1:m, 1:m} \cdot \mathrm{ut}\left( T^t (\Psi U)^t (\Psi W) \right), 
\end{align}
where $\mathrm{ut}$ denotes the upper-triangular part of the input matrix. The key to the randomized Householder QR factorization is the use of the sketching matrix $\Psi$ that allows to preserve the decomposition into a lower triangular matrix and an upper triangular matrix of tall and skinny matrices, thus making it possible to exploit the connection between the sketches of the randomized Householder vectors, $\Psi U$, and the L-factor of the LU factorization of $Q-I$. This connection also allows to express any matrix $Q$ such that $(\Psi Q)^t \Psi Q = I_m$ as a product of randomized Householder reflectors, whose randomized Householder vectors can be retrieved. 
Equations~\cref{intro:equivfacto,intro:reconstructrhqr} yield three main procedures outlined in this work : \textit{right-looking} and \textit{left-looking} RHQR, and \textit{reconstructed} RHQR (recRHQR). The \textit{right-looking}  and \textit{left-looking} RHQR presented in~\Cref{algo:rhqr_rightlooking,algo:rhqr_leftlooking} are based on~\cref{intro:equivfacto}. We show that the core iteration of left-looking RHQR makes a single synchronization as long as the next vector of the input basis is available (e.g QR factorizations, $s$-step Krylov methods, block Krylov methods). We also show that the computational cost of its core iteration is dominated by that of two sketches and a matrix vector product, resulting in twice less flops than Householder QR. We observe in our experiments that RHQR is as stable as Householder QR. Furthermore, the core iteration of RHQR relies only on basic linear algebra operations and doesn't need the solving of a sketched least-squares problem as in RGS, making it potentially even cheaper than RGS. On the previously mentioned difficult cases from~\cite{rgs}, the left-looking RHQR outputs a basis $\Psi Q$ that is numerically orthogonal, with $\mathrm{Cond}(Q) < 2$, while maintaining an accurate factorization. The third process, \textit{reconstruct} RHQR shown in~\Cref{algo:reconstructrhqr}, is based on~\cref{intro:reconstructrhqr} and does a single synchronization, like CholeskyQR and its variants. On the previously mentioned difficult cases, recRHQR is more stable than Randomized CholeskyQR and outputs a basis whose condition number is less than $5$, while maintaining an accurate factorization.

We provide a worst-case rounding error analysis of the RHQR factorization. The use of this model is made possible by the dimension reduction enabled by the sketching technique, coupled with the existence of very accurate sketching techniques. We consider a mixed precision setting, where the sketching and the low-dimensional operations are done in high precision $\e$, and the high-dimensional operations are done in low precision $\f$.  Given that the accuracy of RHQR depends mainly on the forward accuracy of the sketching step,  we analyze and propose the subsampled randomized Hadamard transform (SRHT) as a very accurate sketching technique using worst-case rounding error analysis. We show that under mild probabilistic assumptions and using SRHT sketching, if $(4 \log_2(n) + 2 \ell + 7)\cdot \e \leq \f$, then in mixed precision arithmetics and independently of the condition number of the input $W$,
$$\mathrm{Cond}(\widehat{Q}) \lesssim \frac{1+\epsilon}{1-\epsilon} \cdot \frac{1+ 12\cdot m^{3/2}\cdot \f}{1-12 \cdot m^{3/2} \cdot \f}, \quad \quad \|(W - \widehat{Q} \widehat{R})(:,j)\| \lesssim \frac{1+\epsilon}{1-\epsilon} \cdot 12 m^{3/2} \f \cdot \|W(:,j)\|. \quad \quad \text{(w.h.p)}$$
thus inheriting every aspect of Householder QR's stability. To our knowledge, this makes RHQR the first unconditionally stable randomized algorithm for computing the factorization $W = QR$ in its simplest form. Probabilistic assumptions are only made in order to ensure the $\epsilon$-embedding properties, hence the mention of high probability in the above result. The accumulation of rounding errors is nonetheless analyzed in the worst-case. In contrast to our result, authors showed in~\cite{rgs}, by using a probabilistic rounding error model and mixed precisions, that the bound on the condition number of the basis computed by RGS depends on the condition number of the input matrix $W$. Thus RHQR, and possibly recRHQR, can be used to improve the stability of other algorithms that require computing a well conditioned basis of a vector space, and of algorithms that require a basis that is sketch-orthogonal to machine's precision. They can be used for example in randomized QRCP algorithms~\cite{guqrcp,martinssonqrcp,melnichenko}, in combination with a (strong) rank revealing factorization of the $R$ factor.  They can also be used to produce an $\ell_2$-orthonormal basis of a vector space through the Cholesky QR factorization of the $Q$ factor produced by RHQR or recRHQR, thus replacing randomized Cholesky QR with a more stable factorization (see Remark 2.1 in \cite{brgs} for more details).  Such algorithms can become unconditionally stable, given the stability of randomized Householder QR.  We finally note that the randomized Householder QR factorization can be used in place of the Householder QR factorization in TSQR~\cite{tsqr}, resulting thus in a communication avoiding randomized Householder QR factorization for tall and skinny matrices for half the flops of TSQR and same communication cost.

Following the methods in~\cite[Chapter 6.3.2]{saad}, this RHQR factorization is then embedded in the Arnoldi iteration, allowing to solve systems of linear equations or eigenvalue problems. In this case, RHQR-Arnoldi requires an additional sketch and an additional generalized matrix vector product when compared to RGS-Arnoldi (we recall that Householder-Arnoldi is also more expensive than MGS-Arnoldi). Finally, based on Charles Sheffield's connection between MGS and Householder QR and as it was used in~\cite{barlow} to derive a BLAS2 version of MGS, we derive a BLAS2-RGS in~\Cref{algo:blas2rgs} verifying the following equations:
 
$$\begin{bmatrix} 0_{m \times m} \\ W \end{bmatrix} = \widebar{Q} R, \quad \widebar{Q} = \begin{bmatrix} I_m - T \\ Q T \end{bmatrix}, \quad (\Psi \widebar{Q})^t \Psi \widebar{Q} = I_m, \quad I_m + (\Omega Q)^t \Omega Q = T^{-1} + T^{-t}.$$
On the previously mentioned difficult cases, BLAS2-RGS outputs a basis whose condition number is less than $20$, while maintaining an accurate factorization.

This paper is organized as follows. \Cref{section:preliminaries} discusses the classical Householder QR factorization and the subspace embedding technique. \Cref{section:rhqr} introduces the randomized Householder reflector, and proves that the sketched RHQR factorization of $W$ coincides with the Householder QR factorization of the sketch $\Psi W$ (i.e~\cref{intro:equivfacto}). We then introduce compact formulas that allow to handle this procedure almost in-place and in a left-looking context, as described in our main contribution~\Cref{algo:rhqr_leftlooking} (left-looking RHQR). We also describe a block version in~\Cref{algo:brhqr} (blockRHQR). In~\Cref{section:finiteprecision} we provide a finite precision analysis of the RHQR factorization. In~\Cref{section:arnoldigmres}, we adapt the randomized Arnoldi iteration to RHQR in~\Cref{algo:rhArnoldi} (RHQR-Arnoldi), as well as GMRES in~\Cref{algo:rhGMRES} (RHQR-GMRES), using the same principles as those found in~\cite{saad}. In~\Cref{section:rmgs}, based on Charles Sheffield's connection between Householder QR and Modified Gram-Schmidt, we introduce a Randomized Modified Gram-Schmidt process (BLAS2-RGS) in~\Cref{algo:blas2rgs}. In~\Cref{section:rhqrbis}, we describe an alternative randomized Householder reflector and the corresponding RHQR factorizations in~\Cref{algo:rhqrbis_rightlooking,algo:rhqrbis_leftlooking} (left-looking and right-looking trimRHQR). \Cref{section:experiments} describes a set of numerical experiments showing the performance of the algorithms introduced in this paper compared to state of the art algorithms.

\section{Preliminaries}\label{section:preliminaries}
In this section, we first introduce our notations. We then outline the original Householder procedure. Finally, we introduce the subspace embedding technique.

\subsection{Notations}\label{section:preliminaries:notations}
We denote, for all $n \in \mathbb{N}^*$, $0_n$ the nul vector of $\R^n$. We denote, for all $n, m \in \mathbb{N}$, $0_{n \times m}$ the nul matrix of $\R^{n \times m}$. The matrix for which we seek to produce a factorization is formed by the vectors $w_1 \hdots w_m \in \R^n$, is denoted $W \in \R^{n \times m}$, and is viewed as two blocks:
$$\left[w_1 \; | \; \cdots \; | \; w_m \right] = W = \left[ \begin{matrix} W_1 \\ W_2 \end{matrix} \right], \quad W_1 \in \R^{m \times m}, \quad W_2 \in \R^{(n-m) \times m}.$$
The symbol $\oslash$ used between two vectors denotes the concatenation of vectors. If $x \in \R^{n_1}$ and $y \in \R^{n_2}$, then the vector $x \oslash y \in \R^{n_1 + n_2}$ is defined as:
$$ x \oslash y = \left[ \begin{matrix} x \\ y \end{matrix} \right].$$
The vertical concatenation of matrices with same column dimension is denoted $C = \left[A; \; B \right]$, $A \in \R^{a \times b}$, $B \in \R^{c \times b}$, $C \in \R^{(a+c) \times b}$. 
If $x \in \R^n$, then $|x| \in \R^n$ denotes the vector whose entries are the absolute values of the entries of $x$. We write the inner product of two vectors $\langle x , y \rangle = x^t y$. The norm of a vector, denoted $\|x\|$, is always the $\ell_2$-norm (Euclidean norm) of $x$. If $A \in \R^{n\times m}$ is a matrix, we write $\|A\|_2$ its spectral norm, and $\|A\|_F$ its Frobenius norm.
The null space of $A$ is denoted
$$\Ker{A} = \{ x \in \R^m, \quad Ax = 0_n \}.$$
\color{black}
The symbol $u_j$ for $j \in \{1, \hdots , m \}$ is dedicated to Householder vectors. The symbol $T_j$ for all $j \in \{1, \hdots , m\}$ denotes $j \times j$ matrices that are related to randomized Householder vectors $u_1 \hdots u_j$. The symbol $U$ denotes the matrix formed by the vectors $u_1, \hdots, u_m$.
The symbols $\Omega$ and $\Psi$ denote sketching matrices. If the matrix to be factored is $W \in \mathbb{R}^{n \times m}$, then $\Omega$ and $\Psi$ are respectively in $\R^{\ell \times (n-m)}$ and $\R^{(\ell+m) \times n}$. 

The symbol $P$ denotes several versions of the Householder reflector. When $P$ has one non-zero argument $v \in \R^n$, then it denotes the standard Householder reflector of $\R^n$,
$$P(v) = I_n - \frac{2}{\|v\|^2} v v^t.$$
When $P$ has two arguments $v \in \R^n$ and $\Theta \in \R^{\ell \times n}$, $v \notin \Ker{\Theta}$, then it denotes the randomized Householder reflector,
$$P(v, \Theta) = I_n - \frac{2}{ \| \Theta v \|^2} \cdot v (\Theta v)^t \Theta,$$
which is described in~\Cref{section:rhqr}.

Given any matrix $A \in \R^{n \times m}$, we denote $\mathrm{ut}(A)$, $\mathrm{sut}(A)$, $\mathrm{lt}(A)$ and $\mathrm{slt}(A)$ the upper triangular, strictly upper triangular, lower triangular, strictly lower triangular parts of $A$, respectively. We denote $A_{i,j}$ the entry of $A$ located on the $i$-th row and the $j$-th column. We denote $A_{j_1 : j_2}$ the matrix formed by the $j_1, j_1 + 1, \hdots j_2$-th columns of $A$. We denote $a_1, \hdots a_m$ the column vectors of $A$. In the algorithms, we denote the entries $i$ to $j$ of a vector $w$ with the symbol $(w)_{i:j}$.

The canonical vectors of $\R^n$ are denoted $e_1, \hdots, e_m$. For some $k \in \{1, \hdots , n\}$ and $j \in \{1, \hdots , k\}$, we denote $e_j^k$ the $j$-th canonical vector of $\R^k$. 

The letter $\epsilon$ denotes a real number of $\left] 0 ,  1 \right[$, which is a parameter that determines the dimensions of $\Omega$ and thus $\Psi$. 

The letter $\Delta$ is never a standalone symbol, but is coupled with other standalone symbols. If $x$ is some mathematical object, then $\Delta x$ denotes an object of the same nature, related to $x$, with smaller magnitude (by magnitude we mean norms for vectors and matrices, absolute value for scalars). For example, if $\lambda \in \R$ is a scalar, then $\Delta \lambda$ is another scalar, that has some relation to $\lambda$, and whose absolute value is small compared to that of $\lambda$. 

We make use of two machine precisions in this work. The first is denoted $\e$, and is referred to as \textit{high precision}. The second is denoted $\f$ and is referred to as \textit{low precision}, $0 < \mathbf{e} \ll \mathbf{f} \ll 1$. The set of floating numbers in the lowest precision is denoted $\F_\f \subset \R^n$, while the set of floating numbers in the highest precision is denoted $\F_\e$. We use the worst-case rounding error analysis in this work and employ the same notations as in~\cite{higham} for the high precision $\e$:
$$\forall x, y \in \F_\e, \quad \forall \; * \in \{+, -, \times, \div\}, \quad \fl (x * y) = (1+\delta) (x * y), \quad | \delta | \leq \e \quad \text{(high precision)}$$
The letter $\delta$ is used exclusively to denote the round-off error described above. As in~\cite{higham}, we don't differentiate the symbols for all round-off errors. Each occurrence of $\delta$ denotes potentially different round-off errors, but of magnitude all bounded by $\e$. For instance, if $x \in \R^n$ is a vector, then the notation $\mathrm{Diag}(1+\delta) \cdot x$ is to be understood as a vector whose coordinates are $(1+\delta)x_1, (1+\delta)x_2, \hdots$, where all $\delta$ are potentially different.% (hence $\mathrm{Diag}(1+\delta)$ is not a homothety).  
Using notations from~\cite{higham},
\begin{align} \label{eq:notationgamma} 1+\theta_n = \prod_{i = 1}^n (1+\delta)^{\sigma_i}, \quad \sigma_i = \pm 1, \quad |\theta_n| \leq \gamma_n := \frac{n \e}{1-n\e}\end{align}
We also use the rules of rounding error computing described in~\cite[Lemma 3.3]{higham}. These rules allow to efficiently accumulate a number of rounding errors as the calculation progresses. We recall them here for completeness:
\begin{align}\label{eq:errorrules}
\begin{split}
        (1+\theta_k)(1+\theta_j) = 1+\theta_{k+j}, \quad \quad \frac{1+\theta_k}{1+\theta_j} = \begin{dcases}
            1+ \theta_{k + 2j} \text{ if } k < j \\
            1 + \theta_{k+j} \text{ if } k \geq j
        \end{dcases}\,, \quad \quad \gamma_{k} \gamma_j \leq \gamma_{\min \{k,j \}} \text{ if } \max \{k,j\} \cdot \e \leq \frac{1}{2} \\ \\
         k \gamma_j = \gamma_{kj}, \quad \quad \gamma_k + \e = \gamma_{k+1}, \quad \quad \gamma_k + \gamma_j + \gamma_{k}\gamma_j = \gamma_{k+j} \quad \quad \quad \quad \quad \quad \quad \quad
        \end{split}
\end{align}
Finally, we use the following notations for computations in low precision $\f$: 
$$\forall x, y \in \F_\f, \quad \forall \; * \in \{+, -, \times, \div\}, \quad \fl (x * y) = (1+\zeta) (x * y), \quad | \zeta | \leq \f \quad \text{(low precision)}$$
and
\begin{align} \label{eq:notationchi} 1+\eta_n = \prod_{i = 1}^n (1+\zeta)^{\sigma_i}, \quad \sigma_i = \pm 1, \quad |\eta_n| \leq \chi_n := \frac{n \f}{1-n\f}\end{align}
and make use of the same rules of computation for $\eta_n, \gamma_n$ as for $\theta_n, \gamma_n$.

\subsection{Householder orthonormalization}\label{section:preliminaries:housholder}
The Householder QR factorization is an orthogonalization procedure, alternative to the Gram Schmidt process. While the Gram-Schmidt process focuses on building explicitly the orthogonal factor $Q$, Householder's procedure focuses instead on building the factor $R$ using orthogonal transformations. $Q$ and $Q^t$ are stored in an implicit form, and can explicitly retrieved. The central operator of this procedure is the Householder reflector:
\begin{align}\label{eq:householderreflector}
    \forall z \in \R^n \setminus \{0\}, \quad P(z) = I_n - \frac{2}{\|z\|^2} z \quad z^t,
\end{align}
or equivalently
\begin{align*}
    \forall z \in \R^n \setminus \{0\}, \quad \forall x \in \R^n, \quad P(z) \cdot x = x - \frac{2}{\|z\|^2} \, \langle z, x \rangle \, z.
\end{align*}
$P(z)^2 = I_n$, hence its name, and also $P(z)^t = P(z)$. These two properties combined show that $P(z)$ is an orthogonal matrix. The vector $z$ is an eigenvector associated to eigenvalue $-1$, and the orthogonal of its span is an $n-1$ dimensional eigenspace associated to eigenvalue $1$, i.e the Householder reflector is an orthogonal reflector with respect to the latter hyperplane. Given a vector $w \in \R^n$ that is not a multiple of $e_1$, and setting $z = w - \|w\| e_1$, one can straightforwardly derive,
$$P(z) \cdot w = \| w \| e_1,$$
hence the Householder reflector can be used to annihilate all the coordinates of a vector, except its first entry (it has been shown that the formula for $z$ can be refined into stabler formulas in finite precision, see~\cite[Chapter 19.1]{higham}). It can be inferred by induction that, given a vector $w  = c \oslash d \in \R^n$, $c \in \R^{j-1}$, $d \in \R^{n-j+1}$, there exists $v_j \in R^{n-j+1}$, $u_j = 0_{j-1} \oslash v_j \in \R^n$ such that
\begin{align*}
P(u_j) \cdot w = \left[ \begin{matrix} I_{j-1} & \\ & P(v_j) \end{matrix} \right] \cdot w = \left[ \begin{matrix} c \\ \| d \| e_1^{n-j+1} \end{matrix} \right] = \left[ \begin{matrix} c \\ \| d \| \\ 0_{n-j} \end{matrix} \right],
 \end{align*}
that is the reflector $P(u_j)$ now annihilates the coordinates $j+1$ to $n$ of the vector $w$, without modifying its first $j-1$ entries. By induction (this classic argument is detailed below in the proof of~\Cref{thm:mainthm}), there exist \textit{Householder vectors} $u_1, \cdots u_m$ such that
\begin{align} \label{eq:dethousefact} P(u_m) \cdots P(u_1) W = \left[ \begin{matrix} R \\ 0_{(n-m) \times m} \end{matrix} \right] \iff W = P(u_1) \cdots P(u_m) \left[ \begin{matrix} R \\ 0_{(n-m) \times m} \end{matrix} \right], \quad R \text{ upper-triangular }\end{align}
It was outlined in~\cite{schriebervanloan} that the composition $P(u_1) \cdots P(u_m)$ can be factored as:
\begin{align} \label{eq:basecompactformulas}
    \begin{dcases}
        \; \; P(u_1) \cdots P(u_m) = I_n - U T U^t \\
        \; \; P(u_m) \cdots P(u_1) = I_n - U T^t U^t
    \end{dcases}
\end{align}
where $U \in \R^{n \times m}$ is formed by $u_1, \hdots, u_m$ and $T \in \R^{m \times m}$ is an upper-triangular matrix. If we denote $\beta_j = 2/\|u_j\|^2$ for all $j \in \{1, \hdots ,m\}$, then $T$ defined by induction as:
$$T_1 = \left[\beta_1 \right], \quad \forall j \in \{1, \hdots , m-1\}, \quad T_{j+1} = \left[ \begin{matrix} & T_j & & & -\beta_{j+1} T_j U_j^t u_{j+1} \\ & 0_{1 \times j} & & & \beta_{j+1} \end{matrix} \right], \quad T := T_m$$
As a triangular matrix with a diagonal of non-zeros, $T$ is non-singular. As outlined in~\cite{puglisi}, we can then apply the Woodburry Morrison formula to $I_n - U T U^t$ and derive
$$U^t U = T^{-1} + T^{-t}, \quad T = \Bigl( \mathrm{sut} (U^t U) + \frac{1}{2} \mathrm{Diag}(U^t U) \Bigr)^{-1}$$
where $\mathrm{sut}$ denotes the strictly upper-triangular part and $\mathrm{Diag}$ denotes the matrix formed by the diagonal entries. We insist that both the explicit compositions of Householder reflectors and formulas~\cref{eq:basecompactformulas} are operators of $\R^n$, hence implicitly represented by $\R^{n \times n}$ matrices. The right-hand sides of~\cref{eq:basecompactformulas} are referred to as \textit{compact forms} of the compositions of Householder reflectors.

Whereas the Gram-Schmidt type algorithms produces the factors $Q \in \R^{n \times m}$ and an upper-triangular $R \in \R^{m \times m}$ such that $W=QR$, the Householder process outputs instead a composition of isometric operators $P(u_1) \cdots P(u_m) \in \R^{n \times n}$, and an upper-triangular factor $R \in \R^{m \times m}$ such that $W = P(u_1) \cdots P(u_m) \cdot \left[R; \; 0_{(n-m)\times m} \right]$. A matrix $Q \in \R^{n \times m}$ can still be output by Householder QR using the compact formulas~\cref{eq:basecompactformulas},
$$W = P(u_1) \cdots P(u_m) \begin{bmatrix} R \\ 0_{(n-m) \times m} \end{bmatrix} \iff W = \Bigl(\begin{bmatrix} I_m \\ 0_{(n-m)\times m} \end{bmatrix} - U T U^t \begin{bmatrix} I_m \\ 0_{(n-m)\times m} \end{bmatrix} \Bigr) \cdot R =: QR,$$
We refer to $Q \in \R^{n \times m}$ as the \textit{thin Q factor} associated to $U,T$. We stress that $U,T$ and~\cref{eq:basecompactformulas} are sufficient for the computation of $Q^t x, x \in \R^n$ for solving the least-squares problem. If $W$ admits the factorization in~\cref{eq:dethousefact}, then 
$$ \argmin_{x \in \R^m} \| W x - b \| = R^{-1} \cdot  \left[I_m \; \; 0_{m \times (n-m)} \right] \cdot \Bigl( P(u_m) \cdots P(u_1) \Bigr) \cdot b = R^{-1} \cdot \left[I_m \; 0_{m \times (n-m)} \right] \cdot (b - U T^t U^t b)$$
that is the application of $\left(P(u_1) \cdots P(u_m)\right)^{-1} = P(u_m) \cdots P(u_1)$, followed by the sampling of the first $m$ coordinates, followed by the backward solve with $R$. In this work, we denote the pair $(U,T)$ in~\cref{eq:basecompactformulas} as the \textit{implicit} $Q$ factor.

As described in~\cite{paige,barlow}, Charles Sheffield pointed out to Gene Golub that given a matrix $W \in \R^{n \times m}$, its modified Gram-Schmidt (MGS) QR factorization was equivalent to the Householder QR factorization of $\left[ 0_{m\times m}; \; \; W \right]$ in both exact and finite precision arithmetics. The analysis derived in~\cite{paige} shows that, in this context of finite precision, the factor $T$ computed by Householder QR captures the loss of orthogonality in Modified Gram-Schmidt. Accordingly, a BLAS2-version of MGS, based on Householder QR formulas, was derived in~\cite{barlow}.

All these principles and formulas find a straightforward randomized equivalent in the randomized Householder QR factorization (RHQR) that we derive in~\Cref{section:rhqr}.

\subsection{Subspace embeddings} \label{section:preliminaries:randomization}

\begin{definition}
\cite[Definition 1]{woodruff} We say that a matrix $\Omega \in \R^{\ell \times n}$ is an $\epsilon$-embedding of a vector-subspace $\mathcal{W} \subset \R^n$ if and only if 
\begin{align}\label{eq:subspaceembedding}
    \forall x \in \mathcal{W}, \; \; (1-\epsilon) \|x\| \leq \| \Omega x \| \leq (1+\epsilon) \|x\|.
\end{align}
\end{definition}
The $\epsilon$-embedding property is sometimes written as in~\cref{eq:subspaceembedding}, only with squared norms. The use of $\epsilon$-embedding matrices is obviously justified by one or other of these properties.
For any $m$-dimensional vector subspace $\mathcal{W} \subset \R^n$, there exist simple distributions over $\R^{\ell \times n}$ whose realizations $\Omega \in \R^{\ell \times n}$ drawn independently of $\mathcal{W}$ are an $\epsilon$-embedding of $\mathcal{W}$ with high probability. These distributions are called oblivious subspace embeddings (OSE). 
\begin{definition}
\cite[Definition 2]{woodruff} Let $\epsilon, \delta \in \left]0,1\right[$. Let $m \leq \ell \ll n \in \mathbb{N}^*$. We say that a distribution $\mathcal{D}$ over $\R^{\ell \times n}$ is an oblivious subspace embedding with parameters $(\epsilon, \delta, m)$, denoted ($\epsilon, \delta, m$)-OSE, if and only if for any $m$-dimensional vector subspace $\mathcal{W}_m \subset \R^n$, $\Omega$ drawn from $\mathcal{D}$ independently of $\mathcal{W}_m$ is an $\epsilon$-embedding of $\mathcal{W}_m$ with probability at least $1 - \delta$.
\end{definition}
As an example, we mention three such distributions. The Gaussian OSE is the simplest, and consists in drawing a matrix $G \in \R^{\ell \times n}$ where the coefficients are i.i.d standard Gaussian variables, and set
\begin{align*} 
    \Omega = \frac{1}{\sqrt{\ell}} \, G.
\end{align*}
This matrix is dense, hence the sketching $x \mapsto \Omega x$ has the cost of a dense matrix-vector multiplication in general. However its implementation on a distributed computer is very straightforward and scales well. Furthermore, instead of storing the matrix $\Omega$ itself, one can simply choose a seed and draw from the seeded random number generator during the product $x \mapsto \Omega x$. On a distributed computer, that seed may be communicated between nodes. As a Gaussian matrix, it benefits from a large set of results from random matrix theory. As shown in~\cite[Theorem 4]{woodruff}, it is a $(\epsilon, \delta, m)$-OSE  if the sampling size $\ell$ (that is the number of rows of $\Omega$) is $\mathcal{O}(\epsilon^{-2} \cdot (m + \log(1/\delta))$, that is a modest multiple of $\log(m)$.

A technique with smaller storage cost and faster application $x \mapsto \Omega x$ is given by the subsampled randomized Hadamard transform (SRHT) OSE, which writes
\begin{align*}
    \Omega = \sqrt{ \frac{n}{\ell}} \, P H D,
\end{align*}
where $P$ is made of $\ell$ rows of the identity matrix $I_n$ drawn uniformly at random (which corresponds to a uniform sampling step), $H$ is the normalized Walsh-Hadamard transform, and $D$ is a diagonal of random signs. This matrix is almost that presented in~\cite{ailon2009}, only with a uniform sampling $P$ instead of a hashing step. In most applications, this matrix is not explicitly stored, and is only given as a matrix-vector routine. In this case, its storage cost is negligible. Its application to a vector costs $\mathcal{O}(n \log n)$ flops when implemented with standard Walsh-Hadamard Transform. A synthetic analysis of this technique was proposed in~\cite{troppsrht}. This distribution is OSE($\epsilon,\delta,m$) if $\ell \in \mathcal{O}(\epsilon^{-2}(m+\log(n/\delta)) \log(m/\delta))$, which is also a modest multiple of $m$. We show in the present work that this sketching technique is very accurate in finite precision arithmetics. Authors in~\cite{bsrht} proposed a block version of this technique that is suited for parallel computing. While the Hadamard matrix is well suited for uniform sampling, as all its entries are of equal magnitude, the question of sampling from the rows of general orthogonal matrices is discussed in~\cite{ilse}. 

Finally, we mention distributions based on hashing techniques. These matrices can be obtained by specifying a number $s \leq \ell$ and setting their entries as
$$(\Omega)_{i,j} = \begin{dcases} +1/\sqrt{s} \text{ with probability } s/(2\ell) \\ -1/\sqrt{s} \text{ with probability } s/(2\ell) \\ 0 \text{ with probability } 1 -s/\ell\end{dcases}$$
In this case, the obtained sketching matrix $\Omega$ is sparse, and the user specifies the number of non-zero per columns.  These matrices are easy to generate and straightforward to implement on a parallel computer. Using sparse matrix formats and having an expected $ns$ non-zero entries, their application can be fast and their storage cost low. The lower the parameter $s$, the sparser the matrices $\Omega$ drawn from this distribution, the greater the sampling size required in order to be $(\epsilon, \delta, m)$-OSE.

Let us consider a set of linearly independent vectors $w_1, \hdots w_m$ forming a full-rank matrix $W \in \R^{n \times m}$ and spanning $\mathcal{W} = \mathrm{Range}(W)$. The largest and smallest \textit{singular values} of $W$ are defined as
$$\sigma_{\text{max}}(W) = \|W\|_2 = \max_{ \|x\|_2 = 1 } \|Wx \|_2 = \|W v_{\text{max}}\|_2, \quad \sigma_{\text{min}}(W) = \min_{\|x\|_2 = 1} \|Wx\|_2 = \|W v_{\text{min}}\|_2, \quad v_{\text{max}}, v_{\text{min}} \in \R^m$$
i.e the greatest dilatation and the greatest contraction of an input vector $x \in \R^m$. Their ratio also defines the condition number of $W$, as shown by the min-max theorem.
$$\mathrm{Cond}(W) = \sqrt{\frac{ \lambda_{\text{max}}(W^tW) }{\lambda_{\text{min}}(W^tW)}} = \frac{\sigma_{\text{max}}(W)}{\sigma_{\text{min}}(W)} = \frac{\| W v_{\text{max}} \|}{ \| W v_{\text{min}} \|}, \quad v_{\text{max}}, v_{\text{min}} \in \R^m.$$

If $\Omega \in \R^{\ell \times n}$ is an $\epsilon$-embedding for $\mathrm{Range}(W)$, then (see \cite{woodruff,rgs})
\begin{align*}
    \mathrm{Cond}(W) \leq \frac{1+\epsilon}{1-\epsilon} \, \mathrm{Cond}(\Omega W).
\end{align*}
In particular, if $Q \in \R^{n \times m}$ is a basis of $\mathcal{W}$ such that $\Omega Q \in \R^{\ell \times m}$ has orthonormal columns, the condition number of $Q$ is simply bounded by the constant $(1+\epsilon)/ (1-\epsilon)$. With $\epsilon = 1 / 2$, we get $\mathrm{Cond}(Q) \leq 3$. Hence sketch-orthogonal bases are well-conditioned bases. Furthermore, computations of the coordinates of $x \in \mathrm{Range}(Q)$ in a sketch-orthogonal bases is cheap, as one only needs to compute the coordinates of $\Omega x$ in the basis $\Omega Q$.

\section{Randomized Householder process} \label{section:rhqr}
In this section we introduce a randomized version of the Householder reflector and show its main properties. We then show how several randomized Householder reflectors can be used to factor $W = QR$ such that $\Psi W = (\Psi Q) \cdot R$ is the Householder factorization of $\Psi W$ (\Cref{thm:mainthm}). These computations yield the right-looking version of the RHQR factorization in~\Cref{algo:rhqr_rightlooking}. We then derive a compact representation of the composition of multiple randomized Householder reflectors, yielding our main contribution, the left-looking version of the RHQR factorization in~\Cref{algo:rhqr_leftlooking}. We insist that, thanks to randomization, the update of the T factor in RHQR is actually cheap. We also show a block version of RHQR in~\Cref{algo:brhqr}. Exploiting the equivalence between the Householder factorizations, we derive an algorithm reconstructing the RHQR factorization of $W$ from the Householder QR factorization of $\Psi W$ in~\Cref{algo:reconstructrhqr}.
\subsection{Algebra of the randomized Householder QR}
Let us set a matrix $\Theta \in \R^{\ell \times n}$, a vector $z \in \R^n \setminus \Ker{\Theta}$ and define the \textit{randomized Householder reflector associated to $\Theta$ and $z$}:
\begin{align} \label{eq:basereflector}
\forall \Theta \in \R^{\ell \times n}, \quad \forall z \in \R^n \setminus \Ker{\Theta}, \quad P(z, \Theta) = I_n - \frac{2}{\| \Theta z \|^2} \cdot z (\Theta z)^t \Theta
\end{align}
\begin{proposition}\label{prop:baseproperties}
    Let $P(z, \Theta)$ be defined as in~\cref{eq:basereflector}. The following properties hold:
    \begin{enumerate}
        \item $\forall \lambda \in \R^*, \quad P(\lambda z, \Theta) = P(z, \Theta)$
        \item $P(z, \Theta)^2 = I_n$, i.e $P(z, \Theta)$ is a (non orthogonal) reflector
        \item $\forall x \in \R^n, \quad \Theta \cdot P(z, \Theta) \cdot x = P(\Theta z) \cdot \Theta x$ where $P(\Theta z)$ denotes the standard Householder reflector of $\R^\ell$ associated to $\Theta z \in \R^\ell$. 
        \item $\forall x \in \R^n, \quad \| \Theta \cdot P(z, \Theta) \cdot x \|_2 = \| \Theta x \|_2$, (consequence of 3) We say that $P(z, \Theta)$ is \textbf{sketch-isometric}  with respect to $\Theta$.
    \end{enumerate}
\end{proposition}
\begin{proof}
    Let us denote $P := P(z, \Theta)$. First, for $\lambda \in \R^*$, we obtain:
    $$I_n - \frac{2}{\lambda^2 \| \Theta z \|^2} \cdot (\lambda z) \cdot ( \lambda \Theta z)^t \Theta = I_n - \frac{2}{\| \Theta z \|^2} \cdot \frac{\lambda^2}{\lambda^2} \cdot z \cdot( \Theta z)^t \Theta = P. $$
    For the next two points, we suppose that $\| \Theta z \|^2 = 2$. For $x \in \R^n$,
    $$P x = x - \langle \Theta x, \Theta z \rangle z.$$
    Remark also that $Pz = -z$. Then, 
    $$ P P x = P x + \langle \Theta x, \Theta z \rangle z = x - \langle \Theta x, \Theta z \rangle z  + \langle \Theta x, \Theta z \rangle z = x.$$
    Lastly, developping the squared norm,
    $$\|\Theta P x \|^2 = \|\Theta x \|^2 - 2 \langle \Theta x, \Theta z \rangle^2 + \langle \Theta x, \Theta z \rangle^2 \| \Theta z \|^2 = \|\Theta x \|^2 - 2 \langle \Theta x, \Theta z \rangle^2 + 2 \langle \Theta x, \Theta z \rangle^2 = \| \Theta x \|^2.$$
    For the third property, by definition of~\cref{eq:basereflector}, for all $x \in \R^n$, we get
    $$ \Theta \cdot P \cdot x = \Theta x - \frac{2}{\|\Theta z\|^2} \Theta z (\Theta z)^t \Theta x = \Bigl(I_\ell - \frac{2}{\| \Theta z \|^2} \Theta z (\Theta z)^t \Bigr) \cdot \Theta x = P(\Theta z) \cdot \Theta x$$
    where $P(\Theta z)$ is the Householder reflector of $\R^\ell$ associated to $\Theta z \in \R^\ell$ defined in~\cref{eq:householderreflector}.
\end{proof}
\begin{remark} \label{remark:sketchisometric}
    The sketch-isometric property of $P = P(z, \Theta)$ is equivalent to
    $$ P^t \Theta^t \Theta P = \Theta^t \Theta,$$
    while $P^t \Theta^t \Theta P$ is at most of rank $\ell$. 
\end{remark}
\begin{remark}
    If $\Theta$ comes from an OSE distribution and is drawn independently of the vectors considered, then it is in practice impossible that $z \in \Ker{\Theta}$.
\end{remark}

Compositions of multiple randomized Householder reflectors can be represented with compact formulas.

\begin{proposition}\label{prop:generalcompactwoodburry}
    Let $u_1, \cdots u_m \in \R^n$. Define for all $j \in \{1, \hdots, m\}$ the coefficients $\beta_j = 2/\| \Theta u_j \|^2$. Define by induction the matrix $T_j$ for all $1 \leq j \leq m$ as
    \begin{align*}
    \begin{split}
        & T_1 \in \R^{1 \times 1}, \; \; T_1 = \left[ \beta_1 \right] \\
        & \forall 1 \leq j \leq m-1, \; \; T_{j+1} = \left[ \begin{matrix} T_j & - \beta_j \cdot T_j(\Theta U_j)^t \Theta u_{j+1} \\ 0_{1 \times j} & \beta_j \end{matrix} \right]
        \end{split}
    \end{align*}
    Then, denoting $U \in \R^{n \times m}$ the matrix formed by $u_1, \hdots, u_m$ and $T = T_m$, we get the following compact representation:
    \begin{align}\label{eq:generalcompact}
        \begin{split}
            P(u_1, \Theta) \cdots P(u_m, \Theta) = I_n - U T (\Theta U)^t \Theta \\
            P(u_m, \Theta) \cdots P(u_1, \Theta) = I_n - U T^t (\Theta U)^t \Theta
        \end{split}
    \end{align}
    Furthermore, the factor $T$ verifies
    \begin{align}\label{eq:generalwoodburry}
        (\Theta U)^t \Theta U = T^{-1} + T^{-t}
    \end{align}
\end{proposition}
\begin{proof}
    The compact formulas can be derived by straightforward induction as in the deterministic case shown in~\cite{schriebervanloan}. Then, apply the Woodburry-Morrison formula to the left-wise and right-wise compositions as in~\cite{puglisi}.
\end{proof}
\begin{remark}
    As in the deterministic case, remark that $\beta$ can be computed without computing the norm of $\Omega u$, as $\beta = - \left[(w_1 - \| \Omega w \|)\|\Omega w_1\| \right]^{-1}$.
\end{remark}

Let us set a matrix $\Omega \in \R^{\ell \times (n-m)}$, and define the matrix
\begin{align}\label{eq:psi} \Psi = \left[ \begin{array}{c|ccccc} I_m & & & & & \\ \hline & & & & & \\ & & & \Omega & & \\ & & & & &\end{array} \right] \in \R^{(\ell+m) \times n}.\end{align}
Then the randomized Householder reflector associated to $z$ and $\Psi$ is
\begin{align} \label{eq:rhousereflector}
  \forall z \in \R^n \setminus \Ker{\Psi}, \quad P(z, \Psi) = I_n - \frac{2}{\|\Psi z \|^2} \cdot z \cdot (\Psi z)^t \Psi \in \R^{n \times n}
\end{align}

All properties from~\Cref{prop:baseproperties,prop:generalcompactwoodburry} apply to $P(z, \Psi)$ and to multiple compositions of those reflectors, simply replacing $\Theta$ by $\Psi$. We now show that the design of $\Psi$ allows to annihilate entries of a vector below a given index.

\begin{proposition}\label{prop:elimination}
    Let $j \in \{1, \hdots ,m\}$. Denote $c \in \R^{j-1}$ and $d \in \R^{n-j+1}$ such that $w = c \oslash d$. Define $w' = 0_{j-1} \oslash d$, and suppose that $w'$ is neither in the kernel of $\Psi$ nor a multiple of $e_j^n$. Define the randomized Householder vector as:
    \begin{align} \label{eq:rhousevector} u_j = w' - \| \Psi w' \| e_j \in \R^n.\end{align} 
    Then 
    $$P(u_j, \Psi) \cdot w = c \oslash (\| \Psi w' \| e_1^{n-j+1}) = \begin{bmatrix} c \\ \| \Psi w' \| \\ 0_{n-j} \end{bmatrix}.$$
\end{proposition}
\begin{proof} 
    First, remark that $\langle \Psi u_j, \Psi w \rangle = \langle \Psi u_j, \Psi w' \rangle$ because of $u_j$'s first $j-1$ zero entries coupled with the design of $\Psi$. Also by design of $\Psi$, remark that $\langle \Psi w, \Psi e_j \rangle = \langle \Psi w', \Psi e_j \rangle$. Developing the expression of $u_j$, we get:
    $$\langle \Psi w, \Psi u_j \rangle = \| \Psi w' \|^2 - \| \Psi w' \| \langle \Psi w', \Psi e_j \rangle.$$
    On the other hand, recalling that $\| \Psi e_j \| = 1$ (see~\eqref{eq:psi}),
    $$\| \Psi u_j \|^2 = 2 \| \Psi w' \|^2 - 2 \| \Psi w' \| \langle \Psi w', \Psi e_j \rangle,$$
    hence
    $$\frac{2}{\| \Psi u_j \|^2} \cdot \langle \Psi w, \Psi u_j \rangle = 1,$$
    which finally yields
    $$ P(u_j, \Psi) \cdot w = w - u_j = c \oslash \Bigl( d - d + \| \Psi w' \| e_1 \Bigr) = c \oslash \left( \| \Psi w' \| e_1 \right). $$
\end{proof}
Let us quickly point out that the formula in~\Cref{prop:elimination} suffers the same cancellation problems as in the deterministic case in finite precision arithmetics. Indeed, let us suppose that $w$ is almost a positive multiple of $e_1$, that is $w_1 = \lambda e_1 + g \in \R^n$, where $g$ is a vector of small norm, and $\lambda \in \R$. The vector $u_1$ as defined in~\Cref{prop:elimination} might suffer cancellations. It is then useful to flip the sign of $u_1$, and choose in general 
$$u_j := w' + \mathrm{sign} \langle w_j, e_j \rangle \| \Psi w' \| e_j^n,$$
with $w'$ from~\Cref{prop:elimination}. To sum up, denoting $c \in \R^{j-1}$ and $d \in \R^{n-j+1}$ such that $w = c \oslash d$,
\begin{align}\label{eq:sumup}
\begin{dcases}  \; \sigma_j = \mathrm{sign}\langle w, e_j \rangle \\ \;  \rho_j = \| \Psi (0_{j-1} \oslash d) \| \\ \;  u_j = (0_{j-1} \oslash d) + \sigma_j \rho_j e_j^n \end{dcases} \quad \implies  \quad  \; P(u_j, \Psi) \cdot w = \left[ \begin{matrix} c \\ -\sigma_j \rho_j \\ 0_{n-j} \end{matrix} \right] \in \R^n  
\end{align}

We can now state our main result.

\begin{theorem}\label{thm:mainthm} (simultaneous factorizations)
    Let $W \in \R^{n \times m}$, $W = \left[W_1; \; W_2 \right]$, $W_1 \in \R^{m \times m}$, and assume that $W_2$ is full-rank. Let $\Omega \in \R^{\ell \times (n-m)}$ such that $\Omega W_2$ is also full-rank. Define $\Psi$ as in~\cref{eq:psi}. Then there exist randomized Householder vectors $u_1, \hdots, u_m \in \R^n$ such that
    \begin{align} \label{eq:maintheorem} 
    \begin{dcases}
            \; W = P(u_1, \Psi) \cdots P(u_m, \Psi) \cdot \begin{bmatrix} R \\ 0_{(n-m) \times m} \end{bmatrix} = \Bigl(I_n - U T (\Psi U)^t \Psi \Bigr) \begin{bmatrix} R \\ 0_{(n-m) \times m} \end{bmatrix} = QR  \text{ (RHQR) } \\
            \; \Psi W = P(\Psi u_1) \,\cdots \, P(\Psi u_m) \cdot\begin{bmatrix}  R \\ \quad 0_{\ell \times m} \end{bmatrix}  = \Bigl(I_{\ell+m} - \Psi U T (\Psi U)^t \Bigr) \begin{bmatrix} R \\ \quad 0_{\ell \times m} \end{bmatrix} = (\Psi Q)R \text{ (Householder QR)}
            \end{dcases}
    \end{align}
    i.e the sketch of the RHQR factorization of $W$ is the Householder QR factorization of the sketch $\Psi W$. In particular, $(\Psi Q)^t \Psi Q = I_{m}$. If $\Omega$ is an $\epsilon$-embedding of $\mathrm{Range}(W_2)$, then $\Psi$ is an $\epsilon$-embedding of $\mathrm{Range}(W)$, and we get
    $$ \mathrm{Cond}(Q) \leq \frac{1+\epsilon}{1-\epsilon}$$
\end{theorem}
\begin{proof}
Let us first compute $u_1$ as in~\Cref{prop:elimination} such that the first column of $W$ gets all its entries below the first one cancelled. Since $W_2$ and $\Omega W_2$ are full-rank, $u_1$ is well-defined and is not in the kernel of $\Psi$. We get
$$P(u_1, \Psi) \cdot W = \left[ \begin{array}{c|cccc} r_{1,1} & r_{1,2} & \cdots & & r_{1,m} \\ \hline 0 & & & &\\ \vdots & & & & \\ \vdots & & * & &  \\ \vdots & & & & \\ 0 & & & & \end{array} \right] \in \R^{n \times m}$$
Let us now denote $w$ the second column of $P(u_1, \Psi) \cdot W$, and compute $u_2$ such that all the entries of $w$ below the second one are cancelled. The span of the last $n-m$ rows of $P(u_1, \Psi)$ is still that of $W_2$. Since $W_2$ and $\Omega W_2$ are full rank, $u_2$ is well-defined and is not in the kernel of $\Psi$. Furthermore, as shown in~\Cref{prop:elimination}, the reflector $P(u_2, \Psi)$ does not modify the first row of the input matrix. We get
$$P(u_2, \Psi) P(u_1, \Psi) \cdot W = \left[ \begin{array}{c|c} \begin{matrix} r_{1,1} & r_{1,2} \\ 0 & r_{2,2} \end{matrix} & \begin{matrix} r_{1,3} & \cdots & r_{1,m} \\ r_{2,3} & \cdots & r_{2,m} \end{matrix} \\ \hline \begin{matrix} 0 \; \; \; \;  & 0 \\ \vdots \; \; \; \; & \vdots \\ \vdots \; \; \; \; & \vdots \\ \vdots \; \; \; \; & \vdots \\ 0 \; \; \; \; & 0 \end{matrix} &  \begin{matrix} & & \\ & & \\ & * & \\ & & \\ & & \end{matrix} \end{array} \right] \in \R^{n \times m}$$
Iterating the argument and denoting $P_j = P(u_j, \Psi)$ for all $j \in \{1, \hdots , m\}$, we get
\begin{align} \label{eq:triangularization} P_m P_{m-1} \cdots P_1 W = \left[ \begin{matrix} R \\ 0_{(n-m) \times m} \end{matrix} \right]\end{align}
where $R \in \R^{m \times m}$ is upper-triangular by construction. Recalling that $P_j^2 = I_n$ for all $j \in \{1, \hdots , m\}$, we deduce
\begin{align} \label{eq:rhqrfact} W = P_1 P_2 \cdots P_m \left[ \begin{matrix} R \\ 0_{(n-m) \times m} \end{matrix} \right]\end{align}
If we sketch the above equation, we get
$$\Psi W = \Psi P_1 P_2 \cdots P_m \begin{bmatrix} R \\ 0_{(n-m)\times m} \end{bmatrix}$$
Using the third point of~\Cref{prop:baseproperties}, we get
$$\Psi W = P(\Psi u_1) \cdots P(\Psi u_m) \cdot \Psi \begin{bmatrix} R \\ 0_{(n-m)\times m} \end{bmatrix} = P(\Psi u_1) \cdots P(\Psi u_m) \cdot \begin{bmatrix} R \\ 0_{\ell \times m} \end{bmatrix}$$
where $P(\Psi u_1), \hdots P(\Psi u_m)$ are the standard Householder reflectors of $\R^{_ell +m}$ associated with vectors $\Psi u_1 \hdots \Psi u_m \in \R^{\ell + m}$, which shows the equivalence of the two factorizations. The thin Q factor of the RHQR factorization is
$$Q = P_1 \cdots P_m \begin{bmatrix} I_m \\ 0_{(n-m) \times m} \end{bmatrix}$$
hence
$$\Psi Q = P(\Psi u_1) \cdots P(\Psi u_m) \cdot \begin{bmatrix} I_m \\ 0_{\ell \times m} \end{bmatrix}$$
which is the thin Q factor of the Householder QR factorization of $\Psi W$, hence it is orthogonal, hence $Q$ is sketch-orthogonal.

Finally, the $\epsilon$-embedding property of $\Psi$ follows from Pythagorea's theorem and the $\epsilon$-embedding property of $\Omega$.
\end{proof}

\subsection{Algorithms}

\begin{algorithm}
\caption{Sketching routine for randomized Householder QR}\label{algo:sketchroutine}    
\KwInput{$X \in \R^{n \times k}, \; \Omega \in \R^{\ell \times (n-m)}$}
\KwOutput{$\Psi X\in \R^{(\ell+m)\times k}$ with $\Psi$ defined in~\cref{eq:psi}}
\SetKwFunction{RhSk}{RHSketch}
\SetKwProg{Fn}{function}{:}{}
\Fn{\RhSk{$X, \; \Omega$}}{
$A \gets X_{1:m,\;1:k}$ \\
$B \gets \Omega X_{m+1:n,\; 1:k}$ \\
\KwRet $\left[A; \; B \right]$
}
\end{algorithm}

We detail the modified sketching routine $x \mapsto \Psi x$ in~\Cref{algo:sketchroutine}. We detail the randomized Householder vector computation in~\Cref{algo:rhousevector} (RHVector). We stress that in~\Cref{algo:rhousevector} and in a distributed environment, supposing that sketches are available to all processors after synchronization, then by design of $\Psi$ the scalings in~\cref{line:scale1} do not require further synchronization. The scaling in~\cref{line:scale1}, which is posterior to the sketch of $w$, is more precise than sketching again the obtained vector $u$.

\begin{algorithm}
\caption{Computation of randomized Householder vector}\label{algo:rhousevector}    
\KwInput{$w \in \R^n, \; y \in \R^{\ell+m}$ such that $y = $ \RhSk{$w,\Omega$},$ \; j \in \{1 \hdots m\}$ }
\KwOutput{$u, s, \sigma, \rho, \beta$ such that $s = $ \RhSk{$u, \Omega$} and $\beta = 2/\|s\|^2$ and such that \RhSk{$P(u) w, \; \Omega$}$ = w_{1:j-1} \oslash (-\sigma \rho) \oslash 0_{\ell+m-j}$}
\SetKwFunction{Rh}{RHVector}
\SetKwProg{Fn}{function}{:}{}
\Fn{\Rh{$w$, $y$, $j$}}{
$\sigma \gets $ sign of $j$-th entry of $y$ \\
$u \gets 0_{j-1} \oslash w_{j:n}$ \\
$s \gets 0_{j-1} \oslash y_{j:\ell+m}$\\
$\rho \gets \| s \|$ \\
Add $\sigma \rho$ to the $j$-th entry of $u$ \\
Add $\sigma \rho$ to the $j$-th entry of $s$ \\
$\gamma \gets $ new $j$-th entry of $s$ \\
$\beta \gets (\rho \gamma)^{-1}$ \\
Choose $\begin{dcases} u \gets \sqrt{\beta} \cdot u \\ s \gets \sqrt{\beta} \cdot s \\ \beta \gets 1 \end{dcases}\; \;$ (in turn $\|s\| = \sqrt{2}$) or $ \quad\begin{dcases} u \gets \gamma^{-1} u \\ s \gets \gamma^{-1} s  \\ \beta \gets \gamma/\rho\end{dcases} \; \; $ (in turn $\langle u, e_j \rangle = 1$) \label{line:scale1}\\
\KwRet $u, s, \sigma$, $\rho$, $\beta$
}
\end{algorithm}
\comment{
\begin{algorithm}
\caption{Computation of randomized Householder vector}\label{algo:rhousevector}    
\KwInput{$d \in \R^{n-j+1}$ where $j \in \{1, \hdots,m \}$, $\Omega \in \R^{(\ell-m) \times n}$}
\KwOutput{$v, v', \sigma, \rho, \beta$ such that $0_{j-1} \oslash v' = \Psi (0_{j-1} \oslash v)$ and $P(0_{j-1} \oslash v, \Psi) \cdot (c \oslash d) = c \oslash (-\sigma \rho e_j^n)$ for all $c \in \R^{j-1}$, with $\Psi$ defined in~\cref{eq:psi}}
\SetKwFunction{Rh}{RHVector}
\SetKwProg{Fn}{function}{:}{}
\Fn{\Rh{$y$, $\Omega$}}{
$v \gets d$ \\
$a \gets $ first $m-j+1$ coordinates of $v$ \\
$b \gets $ last $n-m$ coordinates of $v$ \\
Sketch $\Omega b$ \\
$v' \gets a \oslash \Omega b$ \\ 
$\sigma \gets$ sign of first entry of $v'$ \\
$\rho \gets \| v' \|$ \\
Add $\sigma \rho$ to the first entry of $v$ and to the first entry of $v'$ \\
Choose $\alpha = $ first entry of $v'$ or $\alpha = \|v'\|/\sqrt{2}$ \label{line:alpha} \\
$v \gets v/\alpha$  \\
$v' \gets v'/\alpha$  \\
$\beta \gets 2/\|v'\|^2$ \quad \textcolor{black}{\# for second choice of $\alpha$, just set $\beta = 1$} \\
\KwRet $v, v', \sigma$, $\rho$, $\beta$
}
\end{algorithm}
}

The RHQR process, as presented in its simplest form in the proof of~\Cref{thm:mainthm}, is described in~\Cref{algo:rhqr_rightlooking} (right-looking RHQR). One sketching is performed in~\cref{line:rhsketch}, which induces one synchronization. This algorithm can be performed almost in place, since $U$ is lower-triangular by construction, and $R$ is upper-triangular by construction. Both can almost fit in the original $W$, while the diagonal of either $R$ or $U$ can be stored in some additional space. Then, a small additional space is required for $\Psi U$ ($S$ in the algorithm). We stress that, in a distributed environment where the sketches are available to all processors after synchronization, then all processors have access to the $m \times m$ upper-block of $W,R,U$ by design of $\Psi$. We show in~\Cref{section:finiteprecision} that this process is as stable as Householder QR in every aspects. This version of RHQR may be the simplest conceptually, it is clearly not optimal in terms of flops, as it requires to sketch $m-j+1$ vectors at step $j$. This issue is addressed by the left-looking RHQR factorization.

\begin{algorithm}
\caption{Randomized-Householder QR factorization (right-looking)}\label{algo:rhqr_rightlooking}
\KwInput{Matrix $W \in \mathbb{R}^{n\times m}$, matrix $\Omega \in \mathbb{R}^{\ell \times (n-m)}$, $m < l \ll n-m$}
\KwOutput{$U \in \R^{n \times m}, \; S \in \R^{(\ell+m)\times m}, \; R \in \R^{m \times m}$  such that $S = $ \RhSk{$U, \Omega$}, and \RhSk{$W, \Omega$} = $P(u_1) \cdots P(u_m) \cdot \left[R; \; 0_{(\ell \times m} \right]$}
\SetKwFunction{Rh}{RHVector}
\SetKwFunction{Rhqr}{RHQR\_right}
\SetKwProg{Fn}{function}{:}{}
\Fn{ \Rhqr{$W$, $\Omega$}}{
\For{$j = 1:m$}{
$w \gets w_j$ \\
$\left[y \; \; Y \right] \gets$ \RhSk{$W_{j:m}, \; \Omega$} \label{line:rhsketch}\\
$u_j, s_j, \rho, \sigma, \beta$ = \Rh{$w$, $y$, $j$} \label{line:rhvec} \\
$r_j \gets w_{1:j-1} \oslash (-\sigma \rho) \oslash 0_{m-j}$ \\
\If{$j < m$}{
$W_{j+1:m} \gets W_{j+1:m} - \beta \cdot u_j \cdot s_j^t \,Y$  \\
}
}
\KwRet $R = (r_j)_{j \leq m}$, $U = (u_j)_{j \leq m}$, $S = (s_j)_{j \leq m}$
}
\end{algorithm}

\comment{
\begin{algorithm}
\caption{Randomized-Householder QR factorization (right-looking)}\label{algo:rhqr_rightlooking}
\KwInput{Matrix $W \in \mathbb{R}^{n\times m}$, matrix $\Omega \in \mathbb{R}^{\ell \times (n-m)}$, $m < l \ll n-m$}
\KwOutput{$R, U, S$, such that $S = \Psi U$ and~\cref{eq:maintheorem} holds}
\SetKwFunction{Rh}{RHVector}
\SetKwFunction{Rhqr}{RHQR\_right}
\SetKwProg{Fn}{function}{:}{}
\Fn{ \Rhqr{$W$, $\Omega$}}{
\For{$j = 1:m$}{
$r_j \gets (w_j)_{1:j-1}$\\
$w \gets w_j$ \\
$v, v', \rho, \sigma, \beta$ = \Rh{$(w)_{j:n}$, $\Omega$} \label{line:rhvec} \\
$u_j \gets 0_{j-1} \oslash v$ \\
$s_j \gets 0_{j-1} \oslash v'$ \\
$r_j \gets r_j \oslash (-\sigma \rho) \oslash 0_{m-j}$ \\
\If{$j < m$}{
Sketch $\Omega W_{m+1:n, j+1:m}$ \label{line:sketchR}\\
$X \gets \left[W_{1:m,j+1:m}; \; \Omega W_{m+1:n, j+1:m} \right]$ \quad \# i.e $\; X \gets \Psi W$ \\
$W_{j+1:m} \gets W_{j+1:m} - \beta \cdot u_j \cdot s_j^t \,X$  \\
}
}
\KwRet $R = (r_j)_{j \leq m}$, $U = (u_j)_{j \leq m}$, $S = (s_j)_{j \leq m}$
}
\end{algorithm}
}

In the light of the compact formulas~\cref{eq:generalcompact} for the composition of multiple randomized Householder reflectors, we introduce the main algorithm for computing the RHQR factorization in~\Cref{algo:rhqr_leftlooking}. Only the $j$-th column of $W$ is modified at iteration $j$. At the beginning of this iteration, the compact form of the previously computed reflectors is used to apply them to the $j$-th column of $W$. Then the randomized Householder reflector of the resulting column is computed, and the compact factorization is updated. After refreshing the vector $w$ in~\cref{line:refreshwj}, one synchronization is made at~\cref{line:leftsketch}. If the next vector $w_{j+1}$ is already available (e.g QR factorization, $s$-step and block Krylov methods), its sketch can be computed in the same synchronization, hence avoiding a further synchronization in~\cref{line:leftoptsketch}. This algorithm can be executed almost in place, with the same recommendations as for the right-looking version, and a minor additional storage space for the matrix $T$. The sole high-dimensional operations in~\Cref{algo:rhqr_leftlooking} are the two sketches in~\Cref{line:leftsketch,line:leftoptsketch} (again, the second is optional), and the update of the $j$-th column in~\Cref{line:refreshwj}. The cost of the update is dominated by that of $U_{j-1} \cdot (T_{j-1}^t S_{j-1}^t z)$, which is $2nj$ flops, for a final cost of $2nm^2$ flops. Unlike the Householder QR factorization, updating the $T$ factor has a negligible cost. The two sketches add $2n \log(n)$ flops. The cost of the whole factorization is then essentially $2nm^2 +2n \log(n)$. This shows that the left-looking RHQR requires twice less flops than Householder QR, essentially the cost of an LU factorization of $W$ or that of RGS. We note that, similarly to RGS, mixed precision can be used. For example, one could use half or single precision for storing and computing with high-dimensional matrices , while performing the low-dimensional operations in double precision, see experiments in~\Cref{section:experiments}. 

\begin{algorithm}
\caption{Randomized-Householder QR (RHQR) (left-looking)}\label{algo:rhqr_leftlooking}
\KwInput{Matrix $W \in \mathbb{R}^{n\times m}$, matrix $\Omega \in \mathbb{R}^{\ell \times (n-m)}$, $m < l \ll n$}
\KwOutput{$U \in \R^{n \times m}, \; S \in \R^{(\ell+m)\times m}, \; T,R \in \R^{m \times m}$  such that $W = (I_n - U T S^t) \cdot \left[R; \; 0_{(n-m) \times m} \right]$ and \RhSk{$W,\Omega$} = $(I_{\ell+m} - S T S^t) \cdot \left[R; \; 0_{\ell \times m} \right]$}
\SetKwFunction{Rh}{RHVector}
\SetKwFunction{Rhqr}{RHQR\_left}
\SetKwProg{Fn}{function}{:}{}
\Fn{ \Rhqr{$W$, $\Omega$}}{
\For{$j = 1:m$}{
$w \gets w_j$ \\
$y \gets $ \RhSk{$w, \; \Omega$} \label{line:leftoptsketch} \\
\If{$j \geq 2$}{
$w_j \gets w_j - U_{j-1} T_{j-1}^t S_{j-1}^t y$  \quad \textcolor{black}{\#$U_{j-1} = (u_i)_{i \leq j-1}$, idem $S_{j-1}$, $T_{j-1} = (t_{i,k})_{1 \leq i,k \leq j-1}$} \label{line:refreshwj}\\ 
$y \gets $ \RhSk{$w, \; \Omega$} \quad \quad \quad \# if next vector $w_{j+1}$ available, sketch it now \label{line:leftsketch} \\
}
$u_j, s_j, \rho, \sigma, \beta = $ \Rh{$w, \; y, \; j$},  \\
$r_j \gets (w)_{1:j-1} \oslash (-\sigma \rho) \oslash 0_{m-j}$ \\
$t_j \gets 0_{j-1} \oslash \beta \oslash 0_{m-j}$ \\
\If{$j \geq 2$}{
$(t_j)_{1:j-1} = - \beta \cdot T_{j-1} S_{j-1}^t s_j$ \\ 
}
}
\KwRet $R = (r_j)_{j \leq m}, \; U = (u_j)_{j \leq m}, \; S = (s_j)_{j \leq m}, \; T = (t_j)_{j \leq m}$
}
\end{algorithm}
\comment{
\begin{algorithm}
\caption{Randomized-Householder QR (RHQR) (left-looking)}\label{algo:rhqr_leftlooking}
\KwInput{Matrix $W \in \mathbb{R}^{n\times m}$, matrix $\Omega \in \mathbb{R}^{\ell \times n}$, $m < l \ll n$}
\KwOutput{$R, U$, $S$, $T$ such that $S = \Psi U$, with $\Psi$ as in~\cref{eq:psi}, and such that~\cref{eq:maintheorem} holds. Q factor implicit : $Q = (I_n - U T S^t) \left[I_m; \; 0_{\ell \times m} \right]$, $W = QR$.  }
\SetKwFunction{Rh}{RHVector}
\SetKwFunction{Rhqr}{RHQR\_left}
\SetKwProg{Fn}{function}{:}{}
\Fn{ \Rhqr{$W$, $\Omega$}}{
\For{$j = 1:m$}{
$w \gets w_j$ \\
\If{$j \geq 2$}{
Sketch $\Omega (w)_{m+1:n}$ \label{line:sketchz1} \\
$z \gets (w)_{1:m} \oslash \Omega (w)_{m+1:n}$ \quad \# i.e $\; z \gets \Psi w$  \\
$w \gets w_j - U_{j-1} T_{j-1}^t S_{j-1}^t z$  \quad \textcolor{black}{\#$U_{j-1} = (u_i)_{i \leq j-1}$, idem $S_{j-1}$, $T_{j-1} = $ first $j-1$ rows of $(t_i)_{i \leq j-1}$} \label{line:refreshwj}\\ 
}
$v, v', \rho, \sigma, \beta = $ \Rh{$(w)_{j:n}$, $\Omega$},  \label{line:leftrhvec}\\
$u_j \gets 0_{j-1} \oslash v$ \\
$s_j \gets 0_{j-1} \oslash v'$ \\
$r_j \gets (w)_{1:j-1} \oslash (-\sigma \rho) \oslash 0_{m-j}$ \\
$t_j \gets 0_{j-1} \oslash \beta \oslash 0_{m-j}$ \\
\If{$j \geq 2$}{
$(t_j)_{1:j-1} = - \beta \cdot T_{1:j-1,1:j-1} S_{j-1}^t s_j$ \\ 
}
}
\KwRet $R = (r_j)_{j \leq m}, \; U = (u_j)_{j \leq m}, \; S = (s_j)_{j \leq m}, \; T = (t_j)_{j \leq m}$
}
\end{algorithm}
}

If the thin Q factor is needed, one needs to perform the following operation:
$$ Q = \begin{bmatrix} I_m \\ 0_{(n-m) \times m} \end{bmatrix} - U T U_{1:m,1:m}^t.$$
The cost is dominated by that of the matrix-matrix multiply $U_m \cdot ( T_m U_{1:m,1:m}^t)$, which is $2nm^2$.

To conclude this section, we detail in~\Cref{algo:brhqr} a block version of~\Cref{algo:rhqr_leftlooking}. The matrix $W \in \R^{n \times (mb)}$ is now considered as formed by vertical slices $W^{(1)}, \cdots W^{(b)} \in \R^{n \times b}$. \Cref{algo:rhqr_leftlooking} is then applied consecutively to each block. For each $W^{(j)}$, we denote $U_m^{(j)}, \Psi U_m^{(j)}, T_m^{(j)}$ the output of~\Cref{algo:rhqr_leftlooking} applied to $W^{(j)}$. As to the storage space and the communications, we refer to the comments made for~\Cref{algo:rhqr_leftlooking}. The update of columns $(j-1)b + 1, \hdots , jb$ of $W$ is presented as such in~\cref{line:eraseW} for simplicity, however it should be done in place.

\begin{algorithm}
\caption{Block Randomized-Householder QR (block RHQR)}\label{algo:brhqr}
\KwInput{$W = \left[W^{(1)} \cdots W^{(b)}\right] \in \mathbb{R}^{n\times mb}$, matrix $\Omega \in \mathbb{R}^{(\ell+mb) \times (n-mb)}$, $mb < l \ll n$ }
\KwOutput{$R,(U^{(j)}, S^{(j)} \, T^{(j)})_{1 \leq j \leq b}$ such that $W = \Bigl(\prod_{j=1}^b (I_n - U^{(j)} T^{(j)} (S^{(j)})^t) \Bigr) \cdot \left[R; \; 0_{ \ell \times mb} \right]$}
\SetKwFunction{leftrhqr}{RHQR_left}
\SetKwFunction{blockrhqr}{blockRHQR}
\SetKwProg{Fn}{function}{:}{}
\Fn{ \blockrhqr{$W$, $\Omega$}}{
\For{$j = 1:b$}{
$W \gets W^{(j)}$ \\
\For{$k = 1:j-1$}{
$Y \gets $ \RhSk{$W, \Omega$} \\
$W \gets W - U^{(j)} (T^{(j)})^t (S^{(j)})^t Y$  \label{line:eraseW}\\
}
$R_{1:(j-1)b, \; (j-1)b+1:jb} \gets W_{1:(j-1)b, \; (j-1)b+1:jb}$ \\
$R_{(j-1)b+1:jb, \; (j-1)b+1:jb}, U^{(j)}, S^{(j)}, T^{(j)} = $ \Rhqr{$W, \Omega$} \\
}
Return $R$, $\left\{ U^{(j)}, S^{(j)}, T^{(j)} \right\}_{j \leq b}$ 
}
\end{algorithm}

\comment{
\begin{algorithm}
\caption{Block Randomized-Householder QR (block RHQR)}\label{algo:brhqr}
\KwInput{$W = \left[W^{(1)} \cdots W^{(b)}\right] \in \mathbb{R}^{n\times mb}$, matrix $\Omega \in \mathbb{R}^{(\ell+mb) \times (n-mb)}$, $mb < l \ll n$ }
\KwOutput{$R,(U^{(j)}, S^{(j)} \, T^{(j)})_{1 \leq j \leq b}$ such that $W = \Bigl(\prod_{j=1}^b (I_n - U^{(j)} T^{(j)} (S^{(j)})^t) \Bigr) \cdot \left[R; \; 0_{ \ell \times mb} \right]$}
\SetKwFunction{leftrhqr}{RHQR_left}
\SetKwFunction{blockrhqr}{blockRHQR}
\SetKwProg{Fn}{function}{:}{}
\Fn{ \blockrhqr{$W$, $\Omega$}}{
\For{$j = 1:b$}{
$M \gets W^{(j)}$ \\
\For{$k = 1:j-1$}{
Sketch $\Omega M_{mb+1:n,1:m}$ \\
$Z \gets \left[ M_{1:mb,1:m}; \; \Omega M_{mb+1:n,1:m} \right]$  \quad \# i.e $Z \gets \Psi M = \Psi W^{(j)}$\\
$M \gets M - U^{(j)} (T^{(j)})^t (S^{(j)})^t Z$  \label{line:eraseW}\\
Sketch $\Omega M_{mb+1:n,1:m}$ \\
$Z \gets \left[M_{1:mb,1:m}, \Omega M_{mb+1:n,1:m} \right]$ \\
}
$R_{1:(j-1)b, \; (j-1)b+1:jb} \gets M_{1:(j-1)b, \; (j-1)b+1:jb}$ \\
$R^{(j)}, U^{(j)}, S^{(j)}, T^{(j)} = $ \Rhqr{$W^{(j)}, \Omega$} \\
$R_{(j-1)b+1:jb, \;(j-1)b+1,jb} \gets R^{(j)}$ 
}
Return $R$, $\left\{ U^{(j)}, S^{(j)}, T^{(j)} \right\}_{j \leq b}$ 
}
\end{algorithm}
}

\subsection{RHQR reconstructed from HQR}

If $W$ is entirely available at the begining of the procedure, we remark that in exact arithmetics, all elements of RHQR can be deduced from the output of the Householder QR factorization of $\Psi W$. Indeed, in the light of~\Cref{thm:mainthm}, the latter writes
$$\begin{bmatrix} W_1 \\ \Omega W_2 \end{bmatrix} = \Biggl( I_{\ell+m} - \begin{bmatrix} U_1 \\ \Omega U_2 \end{bmatrix} T \left[ U_1^t \; \; (\Omega U_2)^t \right] \Biggr) \begin{bmatrix} R \\ 0_{\ell \times m} \end{bmatrix} \quad \text{(HQR factorization)},$$
where $U_1$, $U_2$ denote respectively the first $m$ rows and last $n-m$ rows of $U_m$ output by RHQR of $W$, $T$ is their shared T factor, and $R$ is their shared R factor. To retrieve $U_2$, we may follow the construction of the randomized Householder vectors $u_1, \hdots, u_m$, focusing on their last $n-m$ entries $u_1^{(2)}, \hdots, u_m^{(2)}$. With the same notations as in~\Cref{algo:rhousevector,algo:rhqr_leftlooking}, $u_1^{(2)}$ is simply $w_1^{(2)}$ divided by $\alpha_1$, where $\alpha_1$ denotes the scaling coefficient chosen in~\Cref{algo:rhousevector} ($\sqrt{\beta}$ or $\gamma$). Then for $u_j^{(2)}, j \geq 2$,
\begin{center}
\begin{enumerate}
    \item $w^{(2)} \gets w_j^{(2)} - U_2 T^t (\Psi U)^t \Psi w_j x$
    \item $u_j^{(2)} \gets \alpha_j^{-1} \cdot w^{(2)}$
\end{enumerate}
\end{center}
which corresponds overall to the formula $W_2 = U_2 T^t (\Psi U)^t \Psi W.$ This formula should be consistent with that obtained by writing the last $n-m$ lines of the RHQR factorization of $W$, namely $W_2 = -U_2 T U_1^t R$. Indeed,
$$ \Psi W = \begin{bmatrix} R \\ 0_{\ell \times m} \end{bmatrix} - \Psi U_m \cdot T_m \cdot U_1^t R$$
Then
$$(\Psi U)^t \Psi W = U_1^t R - (\Psi U)^t \Psi U \cdot T U_1^t R  = U_1^t R - (T^{-1} + T^{-t}) \cdot T U_1^t R = -T^{-t} T U_1^t \cdot R $$
Finally multiplying on the left by $T^t$, we get $T^t (\Psi U)^t \Psi W = -T U_1^t R$ (which shows that it is upper-triangular), hence $U_2 T^t (\Psi U)^t \Psi W = U_2 T U_1^t R$. Also, according to~\Cref{thm:mainthm}, the coefficient $\alpha_j$ scaling the $j$-th randomized Householder is equal to that scaling the $j$-th Householder vector in Householder QR of $\Psi W$. These computations yield~\Cref{algo:reconstructrhqr} (recRHQR). This algorithm makes a single synchronization in~\cref{line:sketchW}. There are multiple ways of carrying out the computations retrieving $U_2$ in~\cref{line:backwardsolve}. We choose the explicit recursion using the coefficients $(\alpha_j)_{1\leq j\leq m}$ retrieved from the Householder factorization of $\Psi W$. Experiments show that on some very difficult examples, this algorithm is more stable than Randomized Cholesky QR~\cite{brgs}, see experiments in~\Cref{section:experiments}.

\begin{algorithm}
\caption{RHQR reconstructing randomized Householder vectors (recRHQR) }\label{algo:reconstructrhqr}
\KwInput{Matrix $W \in \mathbb{R}^{n\times m}$, matrix $\Omega \in \mathbb{R}^{\ell \times n}$, $m < l \ll n$}
\KwOutput{$R, U, S, T$ such that $S = \Psi U$, with $\Psi$ defined as in~\cref{eq:psi}}
\SetKwFunction{Rh}{RHVector}
\SetKwFunction{Hqr}{HouseholderQR}
\SetKwFunction{recRHQR}{recRHQR}
\SetKwProg{Fn}{function}{:}{}
\Fn{ \recRHQR{$W$, $\Omega$}}{
Sketch $\Omega W_{m+1:n,1:m}$ \label{line:sketchW} \\
$Z \gets \left[W_{1:m,1:m}; \; \Omega W_{m+1:n, 1:m} \right]$ \quad \# i.e $ \; Z \gets \Psi W$ \\
$S, T, R \gets$ \Hqr{$Z$} \textcolor{black}{\#Householder vectors $S = \Psi U \in \R^{(\ell+m)\times m}$, T factor, R factor} \\
$Y \gets S_{1:m,1:m}$ \\
Backward solve for $X$ in $W_{m+1:n,1:m} = X \cdot \mathrm{ut}(T^t Y^t Z)$ \label{line:backwardsolve} \\
\KwRet $R, \; \; U = \left[ Y; \; X \right], \; \; S, \; \; T$
}
\end{algorithm}

\begin{remark} As for the deterministic Householder QR, provided that $\left[ I_m; \; 0_{(n-m)\times m} \right] - Q$ and $\left[ I_m; \; 0_{\ell \times m} \right] - \Psi Q$ are non-singular, then the RHQR process yields their unique LU factorizations with unit diagonal. Indeed, writing the thin Q factor yielded by RHQR,
\begin{align} \label{eq:lufact} Q = \begin{bmatrix} I_m \\ 0_{(n-m) \times m} \end{bmatrix} - U_m T_m U_1^t \iff \underbrace{\begin{bmatrix}I_m \\ 0_{(n-m) \times m} \end{bmatrix} - Q}_{\text{non-singular}} =  \underbrace{U}_{\substack{\text{lower-triangular,} \\ \text{diagonal of ones}}} \; \cdot \; \underbrace{T U_1^t}_{\text{upper-triangular}}
\end{align}
which is the unique LU factorization of $\left[ I_m; \; 0_{(n-m)\times m} \right] - Q$. Now sketching this equation, and by design of $\Psi$ which preserves the upper block of the sketched matrix,
\begin{align} \label{eq:sketchedlufact} \underbrace{\begin{bmatrix}I_m \\ 0_{\ell \times m} \end{bmatrix} - \Psi Q}_{\text{non-singular}} =  \underbrace{\begin{bmatrix}U_1 \\ \Omega U_2 \end{bmatrix}}_{\substack{\text{lower-triangular,} \\ \text{diagonal of ones}}} \; \cdot \; \underbrace{T U_1^t}_{\text{upper-triangular}}. 
\end{align}
which is the unique LU factorization of $\left[ I_m; \; 0_{\ell \times m} \right] - \Psi Q$. Hence both the randomized Householder vectors and their sketches can be obtained from $Q$ and $\Psi Q$. In turn, if another orthonormalization process of $W$ yields a factor $Q$ such that $\mathrm{Range}(Q) = \mathrm{Range}(W)$, $(\Psi Q)^t \Psi Q = I_m$, the randomized Householder vectors can be obtained from the LU factorization of $\left[ I_m; \; 0_{\ell \times m} \right] - \Psi Q$. The L factor of the latter is $\Psi U$, from which we can deduce $U_1$. The factor $T$ can be deduced from $\Psi U_m$ thanks to~\cref{eq:generalwoodburry}, but can also be deduced from the upper-triangular factor of the LU factorization, which is $U_1^t T$. If $R$ is not given, it can be retrieved as $(\Psi Q)^t \Psi W$. The factor $U_2$ can be retrieved in many ways, for example as in recRHQR. \end{remark}
\vspace{1cm}

\section{Finite precision analysis of RHQR}\label{section:finiteprecision}
In this section, we analyze in finite precision the accuracy of the RHQR factorization by using the worst-case rounding error analysis. The use of this model is made possible by the dimension reduction offered by the sketching process, coupled with the use of accurate sketching techniques. Given an $\epsilon$-embedding $\Omega$ of a finite set of vectors and whose application $x \mapsto \Omega x$ is stable, we show in~\Cref{thm:backwardfactorization} that the sketch of the computed RHQR can be written as $\Psi W + \Delta W = S \widehat{R}$, where $\widehat{R}$ is the computed R factor, $\Delta W$ is a small (columnwise) backward error, and $S$ is an orthogonal matrix. From this we infer, as in the deterministic analysis, that the sketch of the computed $\widehat{Q}$ factor is close to an orthogonal matrix in \Cref{thm:conditionnumberQ}, and that the sketched factorization is accurate in~\Cref{thm:factoerror}, all independently of the condition number of the input $W$. In other words, the computed RHQR factorization inherits the stability properties of the Householder QR factorization of the sketch of the input matrix $W$. Finally, if $\Psi$ is an $\epsilon$-embedding of the range of the high-dimensional computed Q factor, we show in~\Cref{thm:conditionnumberQ} that the latter is well-conditioned, and in~\Cref{thm:factoerror} that the high-dimensional factorization is accurate, all independently of the condition number of the input $W$. This presentation follows very closely that made in~\cite[Chapter 19]{higham}, while exploiting $\epsilon$-embedding properties. We finally show that the subsampled randomized Hadamard transform (SRHT) is a highly accurate sketching process, with a backward error of magnitude $\gamma_{\log_2(n)+5}$ entirely made by deterministic operations $x \mapsto \sqrt{n/\ell} \cdot x$ and $x \mapsto Hx$. While we restrict our attention to SRHT in this paper, the analysis can be extended to other sketching techniques by considering their forward/backward error bound with the appropriate value of $k$ in \eqref{eq:stabilitysketch} and using potentially the probabilistic rounding error analysis as done in \cite{rgs} for Gaussian sketching. 

We outline that, for all sketching distributions that we mentioned in the preliminaries~\Cref{section:preliminaries:randomization}, the requirement on the sampling size in order to sketch a finite set of $p(m)$ vectors, where $p$ is a polynomial of modest degree, is much weaker than that in order to sketch a full $m$-dimensional vector subspace. The former is only a modest multiple of $\log(m)$, while the latter is a modest multiple of $m$. For this reason, we will assume in the following analysis that if $\Omega$ is drawn from some $(\epsilon, \delta, m)$-OSE distribution, then it is also an $\epsilon$-embedding of a modest number of selected vectors, also uncorrelated to $\Omega$, with probability at least $1 - \mathcal{O}(\delta)$. 

Let $\Omega \in \R^{\ell \times n}$ be an $\epsilon$-embedding for some set of vectors $\mathbf{E} \subset \R^n$. To simplify the following derivations, let us also suppose that $\ell$ is even. We shall make use of notations $\gamma$ and $\chi$ introduced in~\eqref{eq:notationgamma} and~\eqref{eq:notationchi}. Let $k \in \mathbb{N}$ such that the following relation is satisfied (with certainty or at least high probability):
\begin{align} \label{eq:stabilitysketch} \forall x \in \mathbf{E}, \quad \widehat{\Omega x} = \Omega x + \Delta z, \quad \|\Delta z \| \leq \gamma_k \cdot \| \Omega x \| \end{align}
We recall the backward error result of the inner product between two vectors $x,y \in \R^n$, given in~\cite[Chapter 3, eq. (3.4)]{higham}: 
$$\mathrm{fl}\Bigl( \langle x, y \rangle \Bigr) = \langle x + \Delta x, y \rangle, \quad |\Delta x| \leq \gamma_n |x| \implies \| \Delta x \| \leq \gamma_n \| x \|$$
Accounting for the sketching error, as well as the error made when computing the inner product of two vectors in $\R^\ell$, we obtain: 
$$\forall x, y \in \mathbf{E}, \quad \fl \Bigl[ \langle \Omega x, \Omega y \rangle \Bigr] = \langle \Omega x + \Delta_1 z, \Omega y + \Delta_2 z + \Delta_3 z \rangle, \quad \begin{dcases} \| \Delta_1 z \| \leq \gamma_k \| \Omega x \| \\ \| \Delta_2 z \| \leq \gamma_k \| \Omega y \| \\ \| \Delta_3 z \| \leq \gamma_\ell \| \Omega y + \Delta_2 z \| \leq \gamma_\ell(1+\gamma_k) \| \Omega y \| \end{dcases}$$
yielding
$$ \forall x\in \mathbf{E}, \quad \fl \left( \| \Omega x \|^2\right) = \| \Omega x \|^2 + \Delta \omega, \quad \| \Delta \omega \| \leq \gamma_{2k+\ell} \| \Omega x \|^2 $$
hence
$$ \fl \left( \| \Omega x \| \right) = \| \Omega x \| + \Delta'\omega, \quad \| \Delta' \omega \| \leq \left( \e + \frac{1}{2} \| \Delta \omega \| + \e \frac{1}{2} \| \Delta \omega \| \right) \leq \gamma_{k + \ell/2 + 1} \| \Omega x \|$$
Furthermore, it is clear that if the computation of $\Psi x$ is made by the straightforward concatenation of the sketch of the last $n-m$ entries of $x$ and the first $m$ entries of $x$, then sketching a vector by using $\Psi$ benefits from the same accuracy as described above. The dependency of the analysis on the large dimension $n$ is captured in the constant $k$. As we will see at the end of the section, when using SRHT sketching, $k$ is at most a modest multiple of $\sqrt{n}$, and might even be only a modest multiple of $\log_2(n)$. 

\begin{lemma} \label{lemma:accuracyrhvector}
    Let $\Psi$ be an $\epsilon$-embedding only for $w \in \R^n$. Let $u \in \R^n$ be the corresponding randomized Householder vector. Let $s = \Psi u$, and $\beta = 2/\|s\|^2$. Let $\widehat{u}, \widehat{s}$, and $\widehat{\beta}$ be their finite precision counterpart computed by~\Cref{algo:rhousevector}.  Suppose that there exists $k \in \mathbb{N}$ such that the forward accuracy of the sketching process verifiesi~\Cref{eq:stabilitysketch}. Suppose that $\gamma_{2k + \ell + 7} \leq \f$. Then the accuracy expected from $\widehat{u}, \widehat{s}, \widehat{\beta}$ is summarized in~\Cref{tab:precisionu}.
    \begin{figure}[htbp]
    \centering
    \renewcommand{\arraystretch}{1.5}
        \centering
        \begin{tabular}{|c||c|c|}
            \hline
            \textbf{Scaling} & $\| s \|^2 = 2$ & $\langle u, e_1 \rangle = 1$ \\
            \hline\hline
            $| \widehat{u} - u | $ & $\leq \gamma_{2k + \ell + 7} |u| $ & $\leq \gamma_{2k + \ell + 5} |u|$ \\
            \hline
            $\|\widehat{s} - s \|$ & $\leq \gamma_{3k + \ell + 7} \| s \|$ & $ \leq \gamma_{3k + \ell + 5} \| s \|$ \\
            \hline
            $|\widehat{\beta} - \beta|$ & $0$ (exact) & $\leq \gamma_{2k + \ell + 4} \beta$ \\
            \hline
        \end{tabular}
        \caption{Unique precision $\e$}
        \label{tab:precisionu}
    \end{figure}

    In the mixed precision settings, only the first line of the table is modified, namely 
    $$| \widehat{u} - u | \leq \chi_2 |u | \lesssim 2 \f |u|$$
    in both scalings. If $s = \Psi u$ is scaled such that $\|s\| = \sqrt{2}$, we get
    $$
    \|\widehat{s}\| = (1 \pm \gamma_{3k + \ell + 7}) \sqrt{2} \quad \text{(both precision settings)}, \quad \begin{dcases} 
        \| \widehat{u} \| \leq \frac{1+\gamma_{2k + \ell + 7}}{1-\epsilon} \sqrt{2} \quad (\text{unique precision } \e) \\
        \| \widehat{u} \| \leq \frac{1+\chi_2}{1-\epsilon} \sqrt{2} \quad (\text{mixed precisions } \e, \f)
    \end{dcases}$$
\end{lemma}
\begin{proof}
    We give the detail of the proof, which follows closely that of~\cite[Lemma 19.1]{higham}, in appendix. Note that the bound for $u$ is componentwise while that for $s$ is normwise. The crucial quantity is the computed sketched norm of the input $w$. After the first entry of both $w$ and $\widehat{\Psi w}$ is modified with the computed norm, we carefully scale both vectors by taking advantage of some simple formulas. The output $\widehat{s}$ is not the less accurate sketch of the scaled $\widehat{u}$, but the more accurate direct scaling of $\widehat{\Psi w}$. 
\end{proof}

We now continue with the application of one reflector. For simplicity, we consider in this section that $s = \Psi u$ is scaled such that $\|s\|=\sqrt{2}$. If $\Psi$ was an $\epsilon$-embedding for the vector $w$ from which $u$ is built, then  by design of $\Psi$, it is an $\epsilon$-embedding for the unscaled and finite precision output of $\fl \left(w - \widehat{\rho} e_1 \right)$. 

\begin{lemma}\label{lemma:accuracyapplicationreflector}
    Let $u, \widehat{u}$ and $s, \widehat{s}$ verify the conclusions of the previous lemma. Let $y = P(u, \Psi)\cdot x$, and let us denote $\widehat{y} = y + \Delta y$. Suppose that $\Psi$ is an $\epsilon$-embedding for the set $\{x, y, \widehat{y}, \,\Delta y \}$. Suppose that there exists $k \in \mathbb{N}$ such that the forward accuracy of the sketching process verifies~\Cref{eq:stabilitysketch}. Suppose that $\gamma_{4k+2\ell+7} \leq \f$. Then we obtain: 
    $$\Psi \widehat{y} = \Bigl[P(\Psi u) + \Delta P \Bigr] \cdot x, \quad \| \Delta P \|_2 \leq \begin{dcases} \frac{1+\epsilon}{1-\epsilon} \cdot \gamma_{14k + 7\ell + 37} \quad (\text{unique precision } \e) \\ \\  \frac{1+\epsilon}{1-\epsilon} \cdot \chi_{12} \quad \quad \quad (\text{mixed precision } \f, \e) \end{dcases}$$
\end{lemma}
\begin{proof}
    We give the detail of the proof, which follows closely that of~\cite[Lemma 19.2]{higham} in the appendix. We first bound $\widehat{\alpha} = \fl \Bigl( \langle \fl(\Psi x), \widehat{s} \rangle \Bigr)$ in terms of $\| \Psi x \|$. We then do the same for $\widehat{\alpha} \cdot \widehat{u}$, depending on the chosen precision setting. After completing the bounds of $\widehat{y}$ in terms of the norm of $\Psi x$, we sketch the whole equation and benefit from $\Psi P(u, \Psi) = P(\Psi u) \Psi$ where $P(\Psi u)$ is a deterministic Householder reflector, thus with spectral norm $1$. Note that, as in the deterministic case, $\Delta y$ can be bounded component-wise by $y$, however we do not use this property. 
\end{proof}
We now analyze the composition of several randomized Householder reflectors.

\begin{lemma} \label{lemma:sequencereflectors}
    Let $V \in \R^{n \times m}$ be an input matrix. Let $(u_j, s_j, \widehat{u}_j, \widehat{s}_j)_{j \in \{1 \hdots m\}}$ verify the conclusions of the previous lemmas. Let us denote the sequence
    $$\begin{dcases}
        V^{(0)} := V, \\
        V^{(j+1)} := P(u_{j+1}, \Psi) \cdot V^{(j)} \quad j \in \{2, \hdots m\}.
    \end{dcases}$$
    Assume that $\Psi$ is an $\epsilon$-embedding of the finite set of $4m^2$ vectors containing the columns of $\widehat{V}^{(j-1)}$, $P(\widehat{u}_j, \Psi) \cdot \widehat{V}^{(j-1)}$, $\widehat{V}^{(j)}$ and $\widehat{V}^{(j)} - P(\widehat{u}_j, \Psi) \cdot \widehat{V}^{(j-1)}$ for $j \in \{1 \cdots m\}$. Then the final output $\widehat{V}^{(m)}$ verifies
    $$\Psi \widehat{V}^{(m)} = P(\Psi u_m) \cdots P(\Psi u_1) \cdot (\Psi V + \Delta V),$$
    where the backward error matrix $\Delta V \in \R^{\ell \times m}$, formed by vectors $\Delta v_1 \hdots \Delta v_m$, verifies
    $$\forall j \in \{1 \hdots m\}, \quad \| \Delta v_j \| \lesssim  \begin{dcases} \frac{1+\epsilon}{1-\epsilon} \cdot \gamma_{14k + 7\ell + 37} \cdot m \cdot \| \Psi v_j \| \quad (\text{unique precision } \e) \\ \\ \frac{1+\epsilon}{1-\epsilon} \cdot \chi_{12} \cdot m \cdot  \| \Psi v_j \| \quad (\text{mixed precision } \f, \e) \end{dcases} $$
\end{lemma}
\begin{proof}
    Using iteratively the previous lemma, considering unique precision $\e$, we obtain
    $$\Psi \widehat{V}^{(m)} = \Bigl(P(\Psi u_m) + \Delta P^{(m)}\Bigr) \cdots \Bigl(P(\Psi u_1) + \Delta P^{(1)}\Bigr) \cdot \Psi V, \quad \| \Delta P^{(j)} \|_2 \leq  \begin{dcases} \frac{1+\epsilon}{1-\epsilon} \cdot \gamma_{14k + 7 \ell + 37} \quad (\text{precision } \e) \\ \\ \frac{1+\epsilon}{1-\epsilon} \cdot \chi_{12} \quad (\text{mixed precision } \f, \e) \end{dcases}$$
    since the spectral norm is bounded by the Frobenius norm. Note that since we sketched the factorization, we could commute the sketching matrix and the randomized reflectors, and the right-hand side is a composition of perturbed strict isometries of $\R^\ell$, $P(\Psi u_1), \hdots P(\Psi u_m)$. Let us denote
    \begin{align}\label{eq:gamma} \gamma = \begin{dcases}
        \frac{1+\epsilon}{1-\epsilon} \cdot \gamma_{14k + 7\ell + 37} \quad (\text{unique precision } \e) \\ \\
        \frac{1+\epsilon}{1-\epsilon} \cdot \chi_{12} \quad (\text{mixed precisions } \e, \f)
    \end{dcases} \end{align}
    The result follows by applying~\cite[Lemma 3.7]{higham},
    \begin{align} \label{eq:gammaprime} \|\Delta v_j \| \leq \Bigl( (1 + \gamma)^m -1 \Bigr) \| \Psi v_j \| \leq \frac{m \gamma}{1-m\gamma} \|\Psi v_j \| =: m \gamma' \| \Psi v_j \|.\end{align}
\end{proof}
Let us now suppose that $\Omega$ used in $\Psi$ is drawn from ($\epsilon, \delta, m$)-OSE and, as in~\cite[Chapter 19]{higham}, apply the previous lemma to inputs $W$ and $\left[I_m; \; 0_{(n-m)\times m} \right]$, describing respectively the computed $\widehat{R}$ and the computed $\widehat{Q}$. Let us use the notations $\gamma$ and $\gamma'$ from~\eqref{eq:gamma} and~\eqref{eq:gammaprime}. In the implementation proposed in this work, the lower-triangular part is explicitly zeroed-out to yield the upper-triangular $\widehat{R} \in \R^{m \times m}$. However, if we would naively apply the reflectors to the whole matrix $W$, we would instead get a matrix $\widebar{R} \in \R^{n \times m}$, whose upper-triangular part is $\widehat{R}$, and the lower-triangular part would contain round-off errors.
$$\Psi \widebar{R} = P(\Psi u_m) \cdot P(\Psi u_1) \cdot (\Psi W + \Delta V), \quad  \| \Delta v_j \| \leq m \gamma' \| \Psi \widebar{w}_j \|,$$
However, as pointed out in~\cite[Chapter 19]{higham} in the deterministic Householder QR, since the entries at index $j+1 \hdots n$ of $\Delta y$ defined in~\eqref{eq:deltay} are explicitly zeroed out during step $j$ of the algorithm, then by design of $\Psi$ the corresponding $\Delta P^{(j)}$ also has its $j+1, \hdots n$ rows set to zero. In this case, the bound on $\Delta P$ is still valid, and we obtain
$$\Psi \begin{bmatrix} \widehat{R} \\ 0_{(n-m) \times m} \end{bmatrix} = P(\Psi u_m) \cdot P(\Psi u_1) \cdot (\Psi W + \Delta V), \quad  \| \Delta v_j \| \leq m \gamma' \| \Psi w_j \|.$$
In the following results, we will simply denote
$$\gamma' = \Bigl(1+\mathcal{O}(\epsilon)\Bigr)\gamma.$$
Now, by unstacking the reflectors on the right and re-stacking them in reverse order on the left,
$$P(\Psi u_1) \cdots P(\Psi u_m) \cdot \Psi \begin{bmatrix} \widehat{R} \\ 0_{(n-m) \times m} \end{bmatrix} = \Psi W + \Delta V.$$
which by design of $\Psi$ simply yields the following claim.
\begin{theorem} \label{thm:backwardfactorization}
    Let $\Omega$ used in $\Psi$ come from $(\epsilon, \delta, m)$-OSE. Neglect possible minor correlations between $\Omega$ and the round-off errors. Suppose that there exists $k \in \mathbb{N}$ such that the forward accuracy of the sketching process verifies~\Cref{eq:stabilitysketch}. Suppose that $\gamma_{4k+2\ell+7} \leq \f$. Let $\widehat{R}$ be the computed $R$ factor in the right-looking RHQR process on matrix $W$.  Then with probability at least $1-\delta$, there exists an orthonormal matrix $S \in \R^{(\ell+m) \times (\ell+m)}$ and a matrix $\Delta V \in \R^{n \times m}$ such that
    $$S \begin{bmatrix} \widehat{R} \\ 0_{\ell \times m} \end{bmatrix} = \Psi W + \Delta V, \quad \| \Delta v_j \| \lesssim \begin{dcases}  \Bigl(1+\mathcal{O}(\epsilon)\Bigr) \cdot m \cdot \gamma_{14k + 7\ell + 37} \cdot \| \Psi w_j \| \quad (\text{unique precision } \e ) \\ \\ \Bigl(1+\mathcal{O}(\epsilon)\Bigr) \cdot m \cdot \chi_{12} \cdot \| \Psi w_j \| \quad \quad \quad \; (\text{mixed precision } \e, \f)\end{dcases}$$
    with notations $\gamma_p \lesssim p \e$ and $\chi_p \lesssim p \f$ defined in~\eqref{eq:notationgamma} and~\eqref{eq:notationchi}, higher precision $\e$ and lower precision $\f$. 
\end{theorem}
As in the deterministic case, we note that in practical implementations the $j$-th column of the factorization undergoes only the first $j$ reflectors, hence the constant $m \gamma'$ can be lowered to $j \gamma'$. We insist that the above result requires only a finite set of vectors to be well-sketched, and speaks of the accuracy of the simultaneous implicit Householder QR factorization induced by the RHQR process. As in~\cite[Chapter 19]{higham}, the matrix $S$ is never explicitely formed and is of pure theoretical interest, allowing to derive the incoming stability results.

We discuss now the condition number of the computed thin Q factor. As in the deterministic case, we can apply the previous analysis to the input $\left[ I_m ; \; 0_{(n-m)\times m} \right]$. From~\Cref{lemma:sequencereflectors} we get
$$\Psi \cdot \widehat{Q} = P(\Psi u_1) \cdots P(\Psi u_m) \cdot \Psi \cdot \Bigl( \begin{bmatrix} I_m \\ 0_{(n-m)\times m} \end{bmatrix} + \Delta V \Bigr) = S + \Delta V', \quad \Delta V' = S \cdot \Delta V, \quad \| \Delta v_j \| \leq m \gamma' \cdot 1$$
(recall that the columns of $S$ are unit vectors). As the inequality holds column-wise, we get, for the whole matrix,
$$\|\Psi \widehat{Q} - S \|_2 \leq \| \Psi \widehat{Q} - S \|_F \leq m^{3/2} \gamma'.$$
Using Weyl's inequality for singular values, we finally get the following result.
\begin{theorem} \label{thm:conditionnumberQ}
    Let $\Omega$ used in $\Psi$ come from $(\epsilon, \delta, m)$-OSE. Neglect possible minor correlations between $\Omega$ and the round-off errors. Suppose that there exists $k \in \mathbb{N}$ such that the forward accuracy of the sketching process verifies~\Cref{eq:stabilitysketch}. Suppose that $\gamma_{4k+2\ell+7} \leq \f$ Let $\widehat{Q}$ denote the computed thin Q factor by applying the randomized Householder reflectors to $\left[I_m ; \; 0_{(n-m) \times m} \right]$. Then with probability at least $1- \mathcal{O}(\delta)$ we get
    $$\mathrm{Cond}(\Psi \widehat{Q}) \lesssim \begin{dcases} \frac{1+m^{3/2} \cdot \Bigl(1+\mathcal{O}(\epsilon)\Bigr) \cdot \gamma_{14k + 7\ell + 37}}{1-m^{3/2} \cdot \Bigl( 1 + \mathcal{O}(\epsilon) \Bigr)\cdot \gamma_{14k + 7\ell + 37}} \quad (\text{unique precision } \e) \\ \\ \frac{1+m^{3/2} \cdot \Bigl(1+\mathcal{O}(\epsilon)\Bigr) \cdot \chi_{12}}{1-m^{3/2} \cdot \Bigl(1+\mathcal{O}(\epsilon)\Bigr)\cdot \chi_{12}}  \quad \quad \quad (\text{mixed precision } \e, \f) \end{dcases}$$
    $$\mathrm{Cond}(\widehat{Q}) \leq \begin{dcases} \frac{1+\epsilon}{1-\epsilon} \cdot \frac{1+m^{3/2} \cdot \Bigl(1+\mathcal{O}(\epsilon)\Bigr) \cdot \gamma_{14k + 7\ell + 37}}{1-m^{3/2} \cdot \Bigl( 1 + \mathcal{O}(\epsilon) \Bigr)\cdot \gamma_{14k + 7\ell + 37}} \quad (\text{unique precision } \e) \\ \\ \frac{1+\epsilon}{1-\epsilon} \cdot \frac{1+m^{3/2} \cdot \Bigl(1+\mathcal{O}(\epsilon)\Bigr) \cdot \chi_{12}}{1-m^{3/2} \cdot \Bigl(1+\mathcal{O}(\epsilon)\Bigr)\cdot \chi_{12}} \quad \quad \quad (\text{mixed precision } \e, \f) \end{dcases}$$
    with notations $\gamma_p \lesssim p \e$ and $\chi_p \lesssim p \f$ defined in~\eqref{eq:notationgamma} and~\eqref{eq:notationchi}, higher precision $\e$ and lower precision $\f$. 
    Both these bounds are independent of the initial condition number of the input $W$.
\end{theorem}
We finally consider the accuracy of the factorization. As in the deterministic case, it comes from the factorization being written in its backward stable form in~\Cref{thm:backwardfactorization}. Let us denote $\widehat{W} = \widehat{Q} \cdot \widehat{R}$, $\widetilde{W} = S \left[ \widehat{R}; \; 0_{\ell \times m} \right]$, and $\Delta Q = \Psi \widehat{Q} - S \left[I_m ; \; 0_{(n-m) \times m} \right]$. We get
\begin{align*} \| \Psi w_j - \Psi \widehat{w}_j \| & = \| (\Psi w_j - \Psi \widetilde{w}_j) + (\Psi \widetilde{w}_j - \Psi \widehat{w}_j \| \\
& \leq \| \Delta v_j \| + \| \Delta Q \cdot \widehat{r}_j \| \\
& \leq m \gamma' \| \Psi w_j \| + \| \Delta Q \|_F \cdot \|\widehat{r}_j \| \\
& \leq m \gamma' \| \Psi w_j \| + m^{3/2} \gamma' \|\widehat{r}_j\| \\
& \lesssim (m^{3/2} + m) \gamma' \| \Psi w_j \|
\end{align*}
which proves the final following result.
\begin{theorem}\label{thm:factoerror}
    Let $\Omega$ used in $\Psi$ come from $(\epsilon, \delta, m)$-OSE. Neglect possible minor correlations between $\Omega$ and the round-off errors. Suppose that there exists $k \in \mathbb{N}$ such that the forward accuracy of the sketching process verifies~\Cref{eq:stabilitysketch}. Let $\widehat{Q}, \widehat{R}$ denote the computed Q and R factors of the right-looking RHQR decomposition. Let $\widehat{W}$ denote $\widehat{Q} \widehat{R}$. Then with probability at least $1 - \mathcal{O}(\delta)$ we get
    $$\| \Psi (W - \Psi \widehat{Q} \widehat{R})(:,j) \| \leq \begin{dcases} \Bigl(1+\mathcal{O}(\epsilon)\Bigr) \cdot m^{3/2} \gamma_{14k + 7\ell + 37} \| \Psi W(:,j) \| \quad (\text{unique precision } \e) \\ \Bigl(1+\mathcal{O}(\epsilon)\Bigr) \cdot m^{3/2} \chi_{12} \| \Psi W(:,j) \| \quad \quad \quad (\text{mixed precision } \e, \f)\end{dcases} $$
    $$ \|(W - \widehat{Q} \widehat{R})(:,j) \| \leq \begin{dcases} \Bigl(1+\mathcal{O}(\epsilon) \Bigr) \cdot m^{3/2} \cdot \gamma_{14k + 7\ell + 37} \cdot\|W(:,j) \| \quad (\text{unique precision } \e)  \\  \\ \Bigl(1+\mathcal{O}(\epsilon)\Bigr) \cdot m^{3/2} \cdot \chi_{12} \cdot\|W(:,j) \| \quad \quad \quad (\text{mixed precision } \e, \f) \end{dcases}$$
    with notations $\gamma_p \lesssim p \e$ and $\chi_p \lesssim p \f$ defined in~\eqref{eq:notationgamma} and~\eqref{eq:notationchi}, higher precision $\e$ and lower precision $\f$. 
    Both these bounds are independent of the initial condition number of the input $W$.
\end{theorem} 

All the previous results are given as a function of the forward accuracy that we can expect from the sketching step. One very stable sketching process is given by the SRHT technique, as detailed below.

\begin{lemma}\label{lemma:srht}
    Let $\Omega \in \R^{\ell \times n}$ be drawn from SRHT distribution. Let it be applied in finite precision $\e$ in the following manner:
$$x \xrightarrow{\text{scale}} \sqrt{\frac{n}{\ell}} \cdot x \xrightarrow{\substack{\text{randomly} \\ \text{flip signs}}} D \cdot \sqrt{\frac{n}{\ell}} \cdot x \xrightarrow{\substack{\text{Hadamard} \\ \text{transform}}} HD \cdot \sqrt{\frac{n}{\ell}} \cdot x \xrightarrow{\text{sample}} PHD \cdot \sqrt{\frac{n}{\ell}} \cdot x.$$
    Then without any probabilistic requirement, $x \mapsto \Omega x$ is backward  stable with the following accuracy:
    $$\widehat{\Omega x} = \Omega(x + \Delta x), \quad \| \Delta x \| \leq \gamma_{\log_2(n) + 5} \|x\| $$
    Furthermore, the error is entirely made by the deterministic operations, namely the scaling steps and the Hadamard transform.
\end{lemma}
\begin{proof}
    The proof is detailed in appendix. The initial scaling is backward stable with accuracy $\gamma_3$. The random flipping of signs is exact. The normalized Hadamard transform is backward stable with accuracy $\gamma_{\log_2(n)+2}$. The final sampling is exact.
\end{proof}
\begin{remark} \label{rmk:forwardstablesrht}
    The final statement of~\Cref{lemma:srht} motivates for the assumption that the backward error can be sketched as any other vector drawn independently of $\Omega$, yielding
    $$\widehat{\Omega x} = \Omega x + \Delta y, \quad \| \Delta y \| = \| \Omega \cdot \Delta x \| \leq (1+\epsilon) \| \Delta x \| \leq (1+\epsilon) \cdot \gamma_{\log_2(n) + 5} \cdot \| x \| \leq \frac{1+\epsilon}{1-\epsilon} \cdot \gamma_{\log_2(n) + 5} \cdot \| \Omega x \|$$
    and thus that the forward error is of the same magnitude as the backward error. In the proof for RHQR's stability, this would only double the finite number of vectors needed to be sketched in order to apply~\Cref{lemma:accuracyrhvector,lemma:accuracyapplicationreflector,lemma:sequencereflectors}. This would still be a negligible requirement on the sampling size, compared to the requirement for sketching the full range of $\widehat{Q}$, and thus the probabilities in~\Cref{thm:backwardfactorization,thm:conditionnumberQ,thm:factoerror} would still be $1-\mathcal{O}(\delta)$.
\end{remark}

If $k$ from equation~\eqref{eq:stabilitysketch} is $(1+\epsilon)(1-\epsilon)^{-1} \cdot (\log_2(n)+5)$, we see that the final bounds in~\Cref{thm:backwardfactorization,thm:conditionnumberQ,thm:factoerror} depend on $\log(n)$, $\ell$, and $m$. Since the dependence on the large dimension $n$ is logarithmic, the accuracy of RHQR then depends mainly on $m$ and $\ell$, the dimensions of the sketch of the input matrix $W$. If we don't make the assumption that the backward error of $x \mapsto \Omega x$ is well-sketched, we still obtain a bound that is only dependent on $\sqrt{n/\ell} \cdot (\log_2(n) + 5) \approx \sqrt{n}$. In turn, we should assume that $\gamma_{4k + 2\ell + 7} \approx 4\sqrt{n} \e \leq \f$ in~\Cref{lemma:accuracyapplicationreflector} such that the mixed precision bounds in~\Cref{thm:backwardfactorization,thm:conditionnumberQ,thm:factoerror} hold. As to the unique precision bound, they should be read by replacing instances of $\gamma_{12k+6\ell+31}$ with $12 \cdot \sqrt{n} \cdot \e$. \comment{ These bounds still cover most simulations settings, for instance $m = 10^3, n = 10^{11}$. }

\section{Application to Arnoldi and GMRES} \label{section:arnoldigmres}
As in~\cite[Algorithm 6.3, p163]{saad}, the left-looking RHQR factorization is straightforwardly embedded into the Arnoldi iteration, yielding~\Cref{algo:rhArnoldi}. Before detailing the algorithm, let us justify right away that it computes a basis of the Krylov subspace:

\begin{proposition}
    \Cref{algo:rhArnoldi} applied to $A \in \R^{n \times n}$, $b, x_0 \in \R^n$, and $\Omega \in \R^{\ell \times (n-m)}$ produces $Q_{m+1} = \left[ Q_m \; | \; q_{m+1} \right] \in \R^{n \times (m+1)}$ and $H_{m+1,m} \in \R^{(m+1)\times  m}$ such that 
    \begin{align} \label{eq:arnoldirelation} A Q_m = Q_{m+1} H_{m+1,m}.\end{align}
    and $q_1$ is a multiple of $b -A x_0$.
\end{proposition}
\begin{proof}
    The proof is almost identical to that found in~\cite{saad}. Note that the inverse of the composition of randomized reflectors is not its transpose, which has no consequence in the proof. By the definition of the vector $z$ at~\cref{line:zjp1}, once the randomized Householder reflector is applied at~\cref{line:nextreflector}, we obtain:
    \begin{align} h_j = P_{j+1} P_j \cdots P_1 \cdot A q_j. \end{align}
    Since the coordinates $j+2, j+3 \cdots n$ of $h_j$ are zero, it is invariant under the remaining reflectors:
    \begin{align} \label{eq:hj} h_j = P_m \cdots P_{j+2} h_j = P_m \cdots P_1 A q_j. \end{align}
    This relation being true for all $j$, we obtain the factorization:
    $$P_m \cdots P_1 \left[r_0 \; \; A q_1 \; \; \hdots \; \; A q_m \right] = \left[h_0 \; \; h_1 \; \;  \cdots \; \; h_m \right].$$
    Multiplying on the left by the reflectors in reverse order $P_1 \cdots P_m$, we obtain:
    $$ \left[r_0 \; \; A q_1 \; \; \hdots \; \; A q_m \right] = P_1 \cdots P_m \left[h_0 \; \; h_1 \; \;  \cdots \; \; h_m \right] = \mathcal{P}_m^{-1} \left[h_0 \; \; h_1 \; \;  \cdots \; \; h_m \right]$$
    which concludes the proof.
\end{proof}

As Householder-Arnoldi costs more than MGS-Arnoldi, we see RHQR-Arnoldi costs more than RGS Arnoldi. The reason is similar : the vector $z$ used as input for the $j+1$-th RHQR iteration ($z$ in~\Cref{algo:rhArnoldi}) is not $A q_j$, but $P_j \cdots P_1 A q_j$. This induces an additional generalized matrix vector product in~\cref{line:addmatvec}, and in our context, an additional sketch and thus an additional synchronization in~\cref{line:addsync}. This additional cost is put into perspective by the greater stability of RHQR-Arnoldi when compared to RGS-Arnoldi on difficult examples, see~\Cref{section:experiments}. Since this work concerns large-scale contexts, we choose the current presentation in which the Arnoldi basis is its implicit form. This form is sufficient to compute the coordinates of an element from the Krylov subspace in the computed basis. However, one can also store the vectors $q$ from~\Cref{algo:rhArnoldi} along the iteration if the storage space is available.

\begin{algorithm}
\caption{Randomized Householder Arnoldi }\label{algo:rhArnoldi}
\KwInput{Matrix $A \in \mathbb{R}^{n\times n}$, $x_0, b \in \R^n$, $m \in \mathbb{N}^*$, matrix $\Omega \in \mathbb{R}^{\ell \times n}$, $m < l \ll n$,}
\KwOutput{Matrices $Q_m, H_{m+1,m}$ such that~\cref{eq:arnoldirelation} holds ($Q_m$ in compact form, but explicit is a byproduct and may be stored)}
\SetKwFunction{Rh}{RHVector}
\SetKwFunction{rhqrarnoldi}{RHQR\_Arnoldi}
\SetKwProg{Fn}{function}{:}{}
\Fn{ \rhqrarnoldi{$A, b, x_0, \Omega$}}{
$w \gets b - A x_0$ \\
$z \gets $ \RhSk{$w,\Omega$} \\
\For{$j = 1:m+1$}{
$u_j, s_j, \rho, \sigma, \beta \gets $ \Rh{$w$, $z$, $j$}  \\
$h_{j-1} \gets (w)_{1:j-1} \oslash (-\sigma \rho) \oslash 0_{m+1-j}$ \label{line:nextreflector}\\
$t_j \gets 0_{j-1} \oslash \beta \oslash 0_{m-j}$ \\
\If{$j \geq 2$}{
$(t_j)_{1:j-1} = - \beta \cdot T_{j-1} S_{j-1}^t s_j$ \\ 
}
\If{$j \leq m$}{
$q_j \gets e_j^{n} - U_j T_j S_j^t e_j^{\ell+m}$ \quad \textcolor{black}{\# Optional : store $q_j$}\\
$w\gets Aq_j$ \\
$z \gets $\RhSk{$z,\Omega$} \\
$w \gets w - U_j T_j^t S_j^t z$ \quad \label{line:zjp1} \label{line:addmatvec}\\
$z \gets $\RhSk{$w, \Omega$} \quad  \label{line:addsync}\\
}
}
Set $H$ as the first $m+1$ rows of $\left[h_1 \; h_2 \hdots h_m\right]$, discard $h_0$. \\
\KwRet $(u_j)_{j \leq m+1}, (s_j)_{j \leq m+1}, (t_j)_{j \leq m+1}, H$ \quad \textcolor{black}{\#$S = \Psi U$. Optional : return $(q_j)_{1 \leq j \leq m}$} 
}
\end{algorithm}

We also adapt the GMRES process to the RHQR-Arnoldi from~\Cref{algo:rhArnoldi} as in~\cite{saad}[Algorithm 6.10, p174]. We present this algorithm with the compact formulas of left-looking RHQR.

\begin{algorithm}
\caption{Randomized Householder GMRES}\label{algo:rhGMRES}
\KwInput{Matrix $A \in \mathbb{R}^{n\times n}$, $x_0, b \in \R^n$, $m \in \mathbb{N}^*$, matrix $\Omega \in \mathbb{R}^{\ell \times (n-m)}$, $m < l \ll n-m$,}
\KwOutput{$x_m \in \mathcal{K}_j(A,r_0)$ such that~\cref{eq:rgmresopti} holds}
\SetKwFunction{Rh}{RHVector}
\SetKwFunction{rhqrarnoldi}{RHQR\_Arnoldi}
\SetKwFunction{rhqrgmres}{RHQR\_GMRES}
\SetKwProg{Fn}{function}{:}{}
\Fn{ \rhqrgmres{$A, b, x_0, \Omega$}}{
$c \gets $ \RhSk{$b, \Omega$} \\
$\beta = \mathrm{sign}\langle b, e_1 \rangle \cdot \| c \|$ \\
$\left[U \; | \; u \right], \left[S \; | \; s \right], \left[T \; | \; t \right], H \gets $\rhqrarnoldi{$A,b,x_0,\Omega$}  \\
Solve Hessenberg system $H y = \beta e_1^{m+1}$ \\
$x \gets (I_n - U T S^t) (y \oslash 0_\ell)$ \\
\KwRet $x$
}
\end{algorithm}

\begin{proposition}
    Let $H_{m+1,m}, Q_m$ and $Q_{m+1}$ be the Arnoldi relation computed by RHQR-Arnoldi in~\Cref{algo:rhArnoldi}. Let us pick
    $$x_m = Q_m y, \quad y := \argmin_{y \in \R^m} \| \beta e_1 - H_{m+1,m} y \|, \quad \beta = \| \Omega b \|.$$
    Then
    \begin{align} \label{eq:rgmresopti} x_m = \argmin_{x \in \mathcal{K}_m(A,x_0)} \| \Psi(b - A x) \|.\end{align}
    In particular, if $W$ denotes any basis of $\mathcal{K}_m(A,x_0)$ and if $\Omega$ is an $\epsilon$-embedding for $\mathrm{Range}(W_2)$, then
    $$\|b - A x_m \| \leq \frac{1+\epsilon}{1-\epsilon} \argmin_{x \in \R^n} \| b - A x \|$$
\end{proposition}
\begin{proof}
    Just as in RGS, it follows from the orthogonality of the sketch $\Psi Q_{m+1}$ of the Arnoldi basis built by~\Cref{algo:rhArnoldi}:
    $$\argmin_{\rho \in \R^m} \|\beta e_1 - H_{m+1,m} \rho \| = \argmin_{\rho \in \R^m} \|\Psi Q_{j+1} \left[ \beta e_1 - H_{m+1,m} \rho \right] \| = \argmin_{\rho \in \R^m} \|\Psi \left[ b - A Q_j \rho \right] \|$$
\end{proof}

\section{The RHQR based BLAS2-RGS} \label{section:rmgs}
In this section, we derive a BLAS2-RGS process based on the RHQR algorithm applied to matrix $W \in \R^{n \times m}$ topped by a $m \times m$ bloc of zeros. This derivation is similar to that of~\cite{barlow} for the deterministic case, which was motivated by a remark made by Charles Sheffield to Gene Golub.

Let us take $\Omega \in \R^{\ell \times n}$, and suppose that $W$ and $\Omega W$ are full-rank. Let us form $\Psi \in \R^{(\ell + m) \times (n+m)}$, and perform the RHQR factorization of the also full rank matrix
\begin{align} \label{eq:Wring}
    \widebar{W} = \left[ \begin{matrix} 0_{m \times m} \\ W \end{matrix} \right] \in \R^{(n+m) \times m}.
\end{align}
We stress that in such context, the sign in the construction of the randomized Householder vectors is never flipped. Let us denote $\widebar{u}_1 \hdots \widebar{u}_m \in \R^{n+m}$ the randomized Householder vectors, forming the matrix $\widebar{U}$. Let us also scale them such that the diagonal of $\widebar{U} \in \R^{n \times m}$ is the identity matrix $I_m$. Let us denote $T$ the associated T factor. By construction, $\widebar{U}$ and its sketch $\Psi \widebar{U} $ are of the form
$$ \widebar{U} = \left[ \begin{matrix} I_m \\ U \end{matrix} \right] \in \R^{n \times m}, \quad \Psi \widebar{U} = \left[ \begin{matrix} I_m \\ \Omega U \end{matrix} \right] \in \R^{(\ell+m) \times m}, \quad U \in \R^{n \times m}.$$
Seeing as the composition of the randomized Householder reflectors is non-singular, $R$ is also non-singular. Hence, writing the factorization in its compact form, we get (in exact arithmetics)
$$\left[ \begin{matrix} R - T R \\ - U T R \end{matrix} \right] = \left[ \begin{matrix} 0_{m \times m} \\ W \end{matrix} \right]  \implies  \begin{dcases} \; T = I_m \\ \; -U R = W \end{dcases}$$
Let us write $\widebar{Q}$ the thin Q factor output by RHQR. In this context,
$$\widebar{Q} = \begin{bmatrix} I_m \\ 0_{(n+m)\times m} \end{bmatrix} - \left[ \begin{matrix} T \\ U \end{matrix} \right] T \left[ I_m \quad (\Omega U)^t \right] \cdot \begin{bmatrix} I_m \\ 0_{\ell \times m} \end{bmatrix} = \begin{bmatrix} I_m \\ 0_{(n+m)\times m} \end{bmatrix} - \left[ \begin{matrix} T \\ U \end{matrix} \right] = - \left[ \begin{matrix} 0_{m \times m} \\ U \end{matrix} \right].$$
$\Psi \widebar{Q}$ is orthogonal by construction, hence, denoting $Q = -U$, we see that $\Omega Q$ is orthogonal. Hence $W = QR$ is a randomized QR factorization.

Let us then study in detail the arithmetics of one iteration. Remembering that we choose to scale $u_1$ such that its first coordinate is $+1$, it is clear that it is given by
$$z_1 \gets \left[ \begin{matrix} 0_m \\ w_1 \end{matrix} \right] - \| \Omega w_1 \| e_1, \quad u_1 \gets \frac{1}{-\| \Omega w_1 \|} \cdot z_1 = \left[ \begin{matrix} e_1 \\ -w_1/\| \Omega w_1 \| \end{matrix} \right] := \left[ \begin{matrix} e_1 \\ -q_1 \end{matrix} \right], \quad \| \Psi u_1 \|^2 = 2, \quad \| \Omega q_1 \| = 1$$
(remark that both scaling of $u_1$ are equivalent in this context). Hence the first vector $q_1$ is the same as that built by RGS on $W$ with sketching matrix $\Omega$. At iteration $j \geq 2$, denoting $T_{j-1}$ the T factor at iteration $j-1$ denoting $U_{j-1}$ the matrix formed by $u_1, \hdots u_{j-1}$, denoting $Q_{j-1}$ the matrix formed by $q_1, \hdots q_{j-1}$, we first write
$$z_j \gets P_{j-1} \cdots P_1 \cdot \left[ \begin{matrix} 0_{j-1} \\ 0_{m-j+1} \\ w_j \end{matrix} \right] =  \left[ \begin{matrix} 0_{j-1} \\ 0_{m-j+1} \\ w_j \end{matrix} \right] - \left[ \begin{matrix} I_{j-1} \\ 0_{1 \times (m-j+1)} \\ - Q_{j-1} \end{matrix} \right] T_{j-1}^t \left[ I_{j-1} \; \;  0_{(m-j+1) \times 1} \; \; (-\Omega Q_{j-1})^t \right] \left[ \begin{matrix} 0_{j-1} \\ 0_{m-j+1} \\ w_j \end{matrix} \right]$$
yielding
$$z_j \gets \left[ \begin{matrix} T_{j-1}^t (\Omega Q_{j-1})^t \Omega w_j \\ 0_{m-j+1} \\ w_j - Q_{j-1} T_{j-1}^t (\Omega Q_{j-1})^t \Omega w_j \end{matrix} \right] =: \left[ \begin{matrix} x \\ 0_{m-j+1} \\ y \end{matrix} \right].$$
By design of the RHQR algorithm, $x$ (i.e the first $j-1$ entries of $z_j$) gives the first $j-1$ entries of the $j$-th column of the R factor. Now computing the vector $u_j$,
$$\widebar{u}_j \gets \left[ \begin{matrix} 0_{m} \\ y \end{matrix} \right] - \| \Omega y \| e_j, \quad u_j \gets \frac{1}{-\|\Omega y \|} \cdot \widebar{u}_j = \left[ \begin{matrix} 0_{j-1} \\ 1 \\ 0_{m-j} \\ -y/\| \Omega y \| \end{matrix} \right] = \left[\begin{matrix} e_j \\ -q_j \end{matrix} \right]$$
By design of the RHQR algorithm, the $j$-th entry of the $j$-th column of $R$ is given by $\| \Omega y \|$. Seeing as $T_m = I_m$ in exact arithmetics, we are supposed to get
$$\tilde{q}_j \gets w_j - Q_{j-1} (\Omega Q_{j-1})^t \Omega w_j, \quad q_j \gets \frac{1}{\| \Omega \tilde{q}_j \|} \cdot \tilde{q}_j$$
which proves the following claim:
\begin{proposition}
    The BLAS2-RGS in~\Cref{algo:blas2rgs} is equivalent to RGS in exact arithmetics.
\end{proposition}
\begin{algorithm}
\caption{BLAS2-RGS}\label{algo:blas2rgs}
\KwInput{Matrix $W \in \mathbb{R}^{n\times m}$, matrix $\Omega \in \R^{\ell \times n}$}
\KwOutput{Matrices $R$, $Q$ such that $W = QR$}
\For{$j = 1:m$}{
$w \gets w_j$ \\
\If{$j \geq 2$}{
$z \gets \Omega w$\\
$r_j \gets  T_{j-1}^t S_{j-1}^t z$ \quad \textcolor{black}{\#$S_{j-1} = (s_i)_{1 \leq i \leq j-1}$, $T_{j-1} =$ first $j$ rows and $j$ columns of $(t_i)_{1 \leq i \leq j}$}\\
$w \gets w_j - Q_{j-1} r_j$ \quad \textcolor{black}{\#$U_{j-1} = (u_i)_{1 \leq i \leq j-1}$}\\
}
$z \gets \Omega w$ \\
$\rho = \| \Omega z \|$ \\
$q_j \gets \rho^{-1} \cdot w, \quad s_j \gets \rho^{-1} \cdot s$ \\
$r_j \gets r_j \oslash \rho \oslash 0_{m-j}$ \\
$t_j \gets 0_{j-1} \oslash 1 \oslash 0_{m-j}$ \\
$(t_j)_{1:j-1} \gets - T_{1:j-1, 1:j-1} \cdot S_{j-1}^t s_j$\\
}
\textit{Optional : } $Q_j \gets Q_j \cdot T_m$\\
\KwRet $(r_j)_{ j \leq m}, (q_j)_{ j \leq m}$ 
\end{algorithm}

We note that BLAS2-RGS may be adapted in a BLAS3 block version.

\section{Randomized reflectors with partial sketching}\label{section:rhqrbis}

We showcase here another randomization of the Householder reflector, allowing for another RHQR process (trimRHQR) which produces accurate factorizations and well conditioned bases in all cases of our test set. This modified algorithm does not compute a basis whose sketch is orthogonal, but experimentally we observed that the sampling size is smaller than the one required by RHQR, and can even be smaller than the column dimension of the input matrix.
  
Let us pick a matrix $\Omega \in \R^{\ell \times m}$ with unit columns. Let us rule out the possibility that the $m$ first columns of $\Omega$ are orthogonal to the last $n-m$, nor that they are orthogonal to each other. One can verify as in~\Cref{prop:elimination} that with any such matrix $\Omega$, it is possible to eliminate all coordinates of an input vector below its first entry, setting $u = w \pm \|\Omega u\| e_1$, and using the fact that $\| \Omega e_1 \| = 1$:
$$P(u, \Omega) \cdot w = \pm \| \Omega w \| e_1$$
If we denote $w = c \oslash d$, $d \in \R^{n-1}$, if we denote $w' = 0 \oslash d$ and set $u = w' - \| \Omega w' \| e_2$, one will remark that the operator $P(u, \Omega)$ should modify the first entry of some input vector. This comes from the geometric properties that we ruled out for the first $m$ columns of $\Omega$ (this observation motivated for the design of $\Psi$ from~\Cref{section:rhqr} in the first place). We can still enforce the behavior we seek. Let us denote $\Omega' \in \R^{\ell \times (n-1)}$ the matrix formed by the last $n-1$ columns of $\Omega$. For the same reasons as before, denoting $u' = d - \| \Omega' d \| e_1 \in \R^{n-1}$, we can build a reflector $P(\Omega', u') \in \R^{(n-1) \times (n-1)}$ that is sketch isometric with respect to $\Omega'$, and that verifies
$$P(\Omega', u') \cdot w' = \| \Omega' w' \| e_1 \in \R^{n-1}$$
Denoting now
$$H' = \left[ \begin{array}{c|ccc} 1 & & & \\ \hline & & & \\ & & P(\Omega',u') & \\ & & & \end{array}\right] \in \R^{n \times n}$$
we get by design a reflector that was forced to leave the first line of any input untouched, while cancelling the entries of a given input below its second one. However, we have lost the sketch-isometric property of the final reflector of $\R^n$. It is by design sketch-isometric for the following matrix
$$\Psi' = \left[ \begin{array}{c|ccc} 1 & & & \\ \hline & & & \\ & & \Omega' & \\ & & & \end{array}\right] \in \R^{(\ell+1) \times n}$$
Let us generalize this construction. First, we define the sequence of sketching matrices
\begin{align} \label{eq:psij} \Psi_1 := \Omega; \quad \quad \forall j \in \{2 \hdots m\}, \quad \Psi_j := \left[ \begin{array}{c|ccc} I_{j-1} & & & \\ \hline & & & \\ & & \Omega_{j:n} & \\ & & & \end{array}\right] \in \R^{(\ell+j-1) \times n} \end{align}
and the modified randomized Householder reflectors
\begin{align} \label{eq:rhousebis} \forall z \in \R^n \setminus \Ker{\Psi_j}, \quad H(z, \Omega, j) := P(z, \Psi_j)) = I_n - \frac{2}{\| \Psi_j z \|^2} \cdot z \cdot (\Psi_j z)^t \Psi_j \quad \in \R^{n \times n}.\end{align}
We remark that if $u_j$ is a vector which first $j-1$ entries are nil, then~\cref{eq:rhousebis} can be equivalently formulated
\begin{align} \label{eq:rhousebisbis} P(u_j, \Psi_j) = H(u_j, \Omega, j) = I_n - \frac{2}{\| \Omega u_j \|^2} \cdot u_j (\Omega u_j)^t \left[0_{\ell \times (j-1)} \; \; \Omega_{j:n} \right] \end{align}
We summarize their properties in the following proposition:
\begin{proposition}
    Let $j \in \{1, \hdots , m\}$, let $w \in \R^n$, where $w = c \oslash d$, with $c \in \R^{j-1}$ and $d \in \R^{n-(j-1)}$, where $\Omega \in \R^{\ell \times n}$ has unit first $m$ columns. Let us denote $w' = 0_{j-1} \oslash d$. Suppose that $\Omega d \neq 0$ and that $w'$ is not a multiple of $e_j$. Define $u_j = w' - \| \Omega w' \| e_j \in \R^n$. Then $H(u_j, \Omega, j) \cdot w = c \oslash ( \| \Omega w' \| e_1^{n-j+1})$.
\end{proposition}
\begin{proof}
    The proof is identical to that of~\Cref{prop:elimination}, only with the important hypothesis that the vectors $\Omega e_1 \hdots \Omega e_m$ are unit vectors. 
\end{proof}

Using successive modified randomized Householder reflectors $H(u_1, \Omega, 1), \hdots, H(u_m, \Omega, m)$, it is then possible to get the following factorization
\begin{align}\label{eq:trimrhqrfact}
H(u_m, \Omega, m) \cdots H(u_1, \Omega, 1) \cdot W = \begin{bmatrix}R \\ 0_{(n-m) \times m} \end{bmatrix} \iff W = H(u_1, \Omega, 1) \cdots H(u_m, \Omega, m) \cdot \begin{bmatrix} R \\ 0_{(n-m) \times m} \end{bmatrix}  \end{align}
where $R \in \R^{m \times m}$ is upper-triangular by construction. We show this procedure in~\Cref{algo:rhousevectorbis,algo:rhqrbis_rightlooking}. In~\cref{line:rhousevecbissketch} of~\Cref{algo:rhousevectorbis}, one communication is made to sketch $w'$. If the sketches of the first $m$ canonical vectors are available, then no further synchronization is required. In~\Cref{algo:rhqrbis_rightlooking}, we make the same memory management indications as in~\Cref{algo:rhqr_rightlooking}. We note that the sketching made in~\cref{line:rhqrbisbigsketch} can be done in the communication made before the computation of the modified randomized Householder vector.

\begin{algorithm}
\caption{Computation of modified randomized Householder vector (trimRHVector)}\label{algo:rhousevectorbis}    
\KwInput{vector $\tilde{w} \in \R^{n-j+1}$, $j \in \{1, \hdots, m\}$, $\tilde{\Omega} \in \R^{\ell \times (n-j+1)}$ with unit first $m-j+1$ columns}
\KwOutput{$v, v', \rho, \sigma, \beta$ such that $\tilde{w} - \beta \langle v', \tilde{\Omega} \tilde{w} \rangle v = -\sigma \rho e_1^{n-j+1}$}
\SetKwFunction{Rhb}{trimRHVector}
\SetKwProg{Fn}{function}{:}{}
\Fn{\Rhb{$\tilde{w}$, $\tilde{\Omega}$}}{
Sketch $\tilde{\Omega} \tilde{w}$ \label{line:rhousevecbissketch} \\
$\rho \gets \| \tilde{\Omega} \tilde{w} \|$\\
$\sigma = \mathrm{sign}\langle \tilde{w}, e_1\rangle$\\
$v \gets \tilde{w} + \sigma \rho e_1$\\
Sketch $\tilde{\Omega} v$ \\
$v' \gets \tilde{\Omega} v$ \\
Choose $\alpha$ as the first entry of $v'$ or as $\sqrt{2}/\|v'\|$ \\
$v \gets \alpha v$ \\
$v' \gets \alpha v'$ \\
$\beta = 2/\|s\|^2$ \quad \# for second choice of $\alpha$, just set $\beta = 1$\\
\KwRet $v,v', \rho$, $\sigma$, $\beta$
}
\end{algorithm}

\begin{algorithm}
\caption{trimRHQR (right-looking)}\label{algo:rhqrbis_rightlooking}
\KwInput{Matrix $W \in \mathbb{R}^{n\times m}$, matrix $\Omega \in \mathbb{R}^{\ell \times n}$, $m < l \ll n$}
\KwOutput{$R, U, S$ such that $S = \Omega U$ and~\cref{eq:trimrhqrfact} holds}
\SetKwFunction{Rhb}{trimRHVector}
\SetKwFunction{Rhqrbr}{trimRHQR\_right}
\SetKwProg{Fn}{function}{:}{}
\Fn{ \Rhqrbr{$W$, $\Omega$}}{
\For{$j = 1:m$}{
$r_j \gets (w_j)_{1:j-1}$ \\
$w \gets w_j$ \\
$u_j, \rho_j, \sigma, \beta = $ \Rhb{$w_j$, $\Omega$, $j$}\\
$r_j \gets r_j \oslash (-\sigma \rho) \oslash 0_{m-j}$ \\
\If{$j \geq 2$}{
$X \gets \left[0_{\ell \times (j-1)} \; \; \Omega_{j:n}\right] W_{j+1:n}$ \label{line:rhqrbisbigsketch} \\
$W_{j+1:m} \gets W_{j+1:m} - \beta u_j s_j^t X$ \\
}
}
\KwRet $(r_j)_{1 \leq j \leq m}, (u_j)_{1 \leq j \leq m}, (s_j)_{1 \leq j \leq m}$
}
\end{algorithm}

We focus now on the formulas that allow to handle multiple modified randomized Householder reflectors efficiently, allowing for the left-looking version.

\begin{proposition} \label{thm:Qbis} Let $u_1 \hdots u_m$ be the Householder vectors generated by~\Cref{algo:rhqrbis_rightlooking}. For all $j \in \{1, \hdots , m\}$, define $\beta_j = 2 / \| \Omega u_j \|^2$. Define by induction the matrix $T_j$ for all $1 \leq j \leq m$ as
    \begin{align}\label{eq:Tqdefbis}
    \begin{split}
        & T_1 \in \R^{1 \times 1}, \; \; T_1 = \left[ \beta_1 \right] \\
        & \forall 1 \leq j \leq m-1, \; \; T_{j+1} = \left[ \begin{matrix} T_j & - \beta_j \cdot T_j(\Omega U_j)^t \Omega u_{j+1} \\ 0_{1 \times j} & \beta_j \end{matrix} \right]
        \end{split}
    \end{align}
    Then for all $1 \leq j \leq m$, we have the factored form:
    \begin{align}\label{eq:compactrightbis}
\mathcal{H}_j^{-1} := H(u_1, \Omega, 1) \cdots H(u_j, \Omega, j) = P(u_1, \Psi_1) \cdots P(u_j, \Psi_j)  = I - U_{j} T_j \mathrm{ut}\left( (\Omega U_{j})^t \Omega \right).
\end{align}
\end{proposition}
\begin{proof} See appendix. 

\end{proof}
\begin{proposition}\label{prop:woodburrybis}
With the previous notations, and regardless of the scaling of the modified randomized Householder vectors, we get
\begin{align}\label{eq:woodburryTbis}
    \forall j \in \{1, \hdots , m \}, \quad T_j^{-1} + T_j^{-t} = (\Omega U_j)^t \Omega U_j, \quad  T_j = \left[ \mathrm{ut}\left( (\Omega U_j)^t \Omega U_j\right) - \frac{1}{2} \mathrm{Diag}\left( (\Omega U_j)^t \Omega U_j \right) \right]^{-1}
\end{align}
In particular, when the randomized Householder vectors are scaled such that their sketched norm is $\sqrt{2}$, we get
\begin{align}\label{eq:Tcaledbis}
\forall j \in \{1, \hdots , m\}, \quad T_j = \left[\mathrm{ut} \left( (\Omega U_j)^t \Omega U_j \right) - I_j\right]^{-1}\end{align}

\end{proposition}
\begin{proof}
The T factor from~\Cref{prop:woodburrybis} has the same formula as that of the composition of randomized Householder reflectors from~\Cref{section:rhqr}.
\end{proof}
If the right-wise compositions admit the same T factor as for the randomized Householder reflectors from~\cref{section:rhqr}, the left-wise compositions $P(u_m, \Psi_m) \cdots P(u_1, \Psi_1) = H(u_m, \Omega, m) \cdots H(u_1, \Omega, 1)$ feature a T factor that slightly differs from $T^t$. Straightforward computations yield the following result:
\begin{proposition}\label{thm:Qmbis}
    Define by induction the matrix $\tilde{T}_j$ for all $1 \leq j \leq m$ as
    \begin{align}\label{eq:Sqdefbis}
    \begin{split}
        & \tilde{T}_1 \in \R^{1 \times 1}, \; \; \tilde{T}_1 = \left[ \beta_1 \right] \\
        & \forall 1 \leq j \leq m-1, \quad \tilde{t}_j = - \beta \cdot \tilde{T}_{1:j-1,1:j-1} \Bigl( \left[0_{\ell \times (j-1)} \; \; \Omega_{j:n} \right] U_{j-1} \Bigr)^t \Omega u_j \in \R^{(j-1) \times 1} \\
        & \forall 1 \leq j \leq m-1, \; \; \tilde{T}_{j+1} = \left[ \begin{matrix} \tilde{T}_j & \tilde{t}_j \\ 0_{1 \times (j-1)} & \beta_{j+1} \end{matrix} \right]
        \end{split}.
    \end{align}
    Then for all $1 \leq j \leq m$, we have the factored form
    \begin{align} \label{eq:trimrhqrcompactfact}
\mathcal{H}_j = H(u_j, \Omega, j) \cdots H(u_1, \Omega, 1)  = P(u_j, \Psi_j) \cdots P(u_1, \Psi_1) = I - U_{j} \tilde{T}_j^t \mathrm{ut}\left( (\Omega U_{j})^t \Omega \right).
\end{align}
\end{proposition}
We also get an equivalent formulation for $\tilde{T}_j$ from the Woodburry-Morrison formula:
\begin{proposition}\label{prop:deltawoodburry}
    With the previous notations, and regardless of the scaling of the modified randomized Householder vectors,
    \begin{align} \tilde{T}_j^t = -\left[ \mathrm{ut} \left( (\Omega U_j)^t \Omega U_j \right) - \frac{1}{2} \mathrm{Diag}\left( (\Omega U_j)^t \Omega U_j \right) - \mathrm{ut}\left( (\Omega U_j)^t \Omega \right) U_j  \right]^{-1} = -\left[ T_j^{-1} - \mathrm{ut}\left( (\Omega U_j)^t \Omega \right) U_j \right]^{-1}\end{align}
\end{proposition}
\begin{proof} See appendix.
\end{proof}

All these formulas allow to identify the compact factorizations
\begin{align} \label{eq:trimrhqr_finalform}
W = \Bigl( I_n - U_m T_m \Triu \Bigr) \cdot \begin{bmatrix} R \\ 0_{(n-m) \times m} \end{bmatrix} \iff \Bigl( I_n - U_m \tilde{T}_m^t \Triu \Bigr) \cdot W = \begin{bmatrix} R \\ 0_{(n-m) \times m} \end{bmatrix}
\end{align}
and to derive the left-looking version of trimRHQR, given in~\Cref{algo:rhqrbis_leftlooking}. We make the same memory management indications as in~\Cref{algo:rhqr_leftlooking}. If needed, the output can be used to produce the thin factor $Q_m = \mathcal{H}_m^{-1} I_{1:m}$. As to~\Cref{line:applytriubis}, the main challenge is the application of $\mathrm{ut} \left( (\Omega U_j)^t \Omega \right)$, which can be done in several ways. For example, one can first compute $(\Omega U_j)^t \Omega x$, then substract the vector
$$ \left[ \begin{matrix} 0 \\ \langle \Omega u_2, x_1 \Omega_{1:1} \rangle \\ \langle \Omega u_3, \Omega_{1:2} x_{1:2} \rangle \\ \vdots \\ \langle \Omega u_j, \Omega_{1:j} x_{1:j} \rangle \end{matrix} \right].$$
If $\Omega$ is not stored as a dense matrix but can only be applied to a vector (e.g SRHT OSE), one can also compute explicitly along the iterations
\begin{align}\label{eq:Lmformulabis}
\mathrm{slt} \left( (\Omega U_j)^t \Omega \right) = \left[ \begin{matrix} 
0 & 0 & 0 & \hdots & & \hdots & 0 \\
\langle \Omega u_2, \Omega e_1 \rangle & 0 & 0 & \hdots & & \hdots & 0 \\
\langle \Omega u_3, \Omega e_1 \rangle & \langle \Omega u_3, \Omega e_2 \rangle & 0 & \hdots & & & \\
\vdots & & \ddots & & & & & \\
\langle \Omega u_j, \Omega e_1 \rangle & \langle \Omega u_j, \Omega e_2 \rangle & \hdots & \langle \Omega u_j, \Omega e_{j-1} \rangle & 0 & \hdots & 0  \end{matrix} \right].
\end{align}
Then, the computation of $\mathrm{ut} \left( (\Omega U_j)^t \Omega \right) x$ is replaced by that of $(\Omega U_j)^t \Omega x$ followed by the substraction of $\mathrm{slt} \left( ( \Omega U_j)^t \Omega \right) \, x$. The same approach can be used to compute the update of $\tilde{T}_j$ when using~\Cref{eq:Sqdefbis}.

\begin{algorithm}
\caption{trimRHQR (left-looking)}\label{algo:rhqrbis_leftlooking}
\KwInput{Matrix $W \in \mathbb{R}^{n\times m}$, matrix $\Omega \in \mathbb{R}^{\ell \times n}$, $m < l \ll n$}
\KwOutput{$R, U_m, \Omega U_m, T_m, \tilde{T}_m$ such that~\cref{eq:trimrhqr_finalform} holds}
\SetKwFunction{Rh}{RHVector}
\SetKwFunction{Rhqrbl}{trimRHQR\_left}
\SetKwProg{Fn}{function}{:}{}
\Fn{ \Rhqrbl{$W$, $\Omega$}}{
\For{$j = 1:m$}{
$w \gets w_j$ \\
\If{$j \geq 2$}{
Sketch $\Omega w$, compute $\mathrm{ut}\left( (\Omega U_{j-1})^t \Omega \right) \cdot w$ \quad \# e.g using~\cref{eq:Lmformulabis}\\
$w \gets w - U_{j-1} \tilde{T}_{j-1}^t \mathrm{ut}\left( (\Omega U_{j-1})^t \Omega  \right) \cdot w$\label{line:applytriubis} \\
}
$v, s_j, \rho, \sigma, \beta = $ \Rhb{$(w)_{j:n}$, $\Omega$}\\
$u_j \gets 0_{j-1} \oslash v$ \\
$r_j \gets (w)_{1:j-1} \oslash (-\sigma \rho) \oslash 0_{m-j}$ \\
$t_j \gets 0_{j-1} \oslash \beta \oslash 0_{m-j}$ \\
$\tilde{t}_j \gets 0_{j-1} \oslash \beta \oslash 0_{m-j}$ \\
\If{$j \geq 2$}{
$(t_j)_{1:j-1} \gets - \beta  \cdot T_{1:j-1,1:j-1} S_j^t s_j$ \\
$(\tilde{t}_j)_{1:j-1} \gets -\beta  \cdot \tilde{T}_{1:j-1,1:j-1} \Bigl( \left[0_{\ell \times (j-1)} \; \; \Omega_{j:n} \right] U_{j-1} \Bigr)^t s_j$ \\
}
}
\KwRet $R = (r_j)_{1 \leq j \leq m}, \; U = (u_j)_{1 \leq j \leq m}, \; S = (s_j)_{1 \leq j \leq m}, \; T = (t_j)_{1 \leq j \leq m}$
}
\end{algorithm}

We note that a block version of~\Cref{algo:rhqrbis_leftlooking} can be straightforwardly derived as for~\Cref{algo:rhqr_leftlooking,algo:brhqr}.

\section{Numerical experiments}\label{section:experiments}
In this section, we test numerically the algorithms derived in this work. We first detail the matrices of our test set. We then compare left-looking RHQR to RGS on a difficult example where RHQR is the most stable, both in unique (simple and double) precision and mixed (half and double) precision settings. We then compare recRHQR and Randomized Cholesky QR (both algorithms perform a single synchronization) on the same difficult example, where reconstrcutRHQR is the most stable. Next, we compare RGS-GMRES and RHQR-GMRES on medium and high difficulty problems, and the results are consistent with those of the QR factorization. We then compare BLAS2-RGS with RGS on the same difficult example, and observe a slight advantage for BLAS2-RGS. We finally showcase the performance of trimRHQR on the same difficult examples, and observe that trimRHQR is as stable as RHQR. Furthermore, with sampling size inferior to the column dimension of $W$, it is stabler than RGS with sampling size superior to the column dimension of $W$.

The first example is formed by uniformly discretized parametric functions, also used in~\cite{rgs}, that we denote $C_m \in \R^{50000 \times m}$. For all floating (and thus rational) numbers $0 \leq x, \mu \leq 1$, the function is defined as
$$f(x,\mu) = \frac{ \sin \left( 10 (\mu + x) \right)}{ \cos{ \left( 100 (\mu -x) \right) } + 1.1 } $$
and the associated matrix is
\begin{align}
\label{syntheticMatrix}
C_m \in \R^{n \times m}, \; \; C_{i,j} = f\left( \frac{i-1}{n-1}, \frac{j-1}{m-1} \right), \quad i\in \{1 \hdots n\}, \quad j \in \{1, \hdots , m\}.\end{align}
The condition number of $C_{1500}$ in double precision is displayed in~\Cref{fig:conds64W}, and that of $C_{600}$ in single precision in~\Cref{fig:conds32W}.

For all experiments, the $\epsilon$-embedding property of drawn matrix $\Omega$ is tested by first computing the LAPACK's pivoted QR factorization of $W$. If the factorization is accurate to machine precision, we check that the condition number of the sketch of the Q factor is indeed $1 + \mathcal{O}(\epsilon)$. The least-squares problem in RGS's inner loop is always computed using LAPACK's pivoted QR, at each iteration.

\begin{table}[h]
\centering
\begin{tabular}{ |p{3.3cm}||p{2.5cm}|p{1.3cm}|p{8.1cm}|}
\hline
 Name & Size & Cond. & Origin \\
 \hline
 $C_m$   & $50000 \times m$  & $\infty$ & synthetic functions  \\
 \textit{SiO2} & $155331$ & $\approx 2300$ & quantum chemistry problem \\
 \textit{El3D} & $32663$ & $\approx 10^{28}$ & Near incompressible regime elasticity problem \\ 
 \hline
\end{tabular}
 \caption{Set of matrices used in experiments}\label{table}
 \end{table}
 
\Cref{fig:perf64} displays the accuracy obtained in double precision for both RHQR (\Cref{algo:rhqr_rightlooking,algo:rhqr_leftlooking}) and RGS (\cite[Algorithm 2]{rgs}) , in terms of condition number of the computed basis and the accuracy of the obtained factorization.  \Cref{fig:conds64W} shows the condition number of $W$ as the number of its vectors increases from $1$ to $1500$. The condition number grows exponentially and the matrix becomes numerically singular when approximately the $500$-th vector is reached. \Cref{fig:conds64} displays in red (and circle points) the condition number of the basis and the sketched basis obtained through RHQR from \Cref{algo:rhqr_leftlooking}, and in blue (and triangle points) those obtained by RGS for reference. We observe that for similar sampling size, RGS and RHQR have similar performance. However, around the $500$-th vector (which coincides with the moment the matrix becomes numerically singular), we see that RGS experiences some instabilities, while the behavior of RHQR does not change. We observe that the sketched basis of RHQR (a posteriori sketch of the output thin Q factor) remains numerically orthogonal. \Cref{fig:errors64} shows the relative Frobenius error of the QR factorization obtained by RHQR (in red) and RGS (in blue). Both algorithms produce accurate factorizations, with a slight advantage for RGS, but that becomes negligible when $m$ grows. 

\begin{figure}
  \begin{subfigure}[t]{.33\textwidth}
    \centering
    \includegraphics[width=\linewidth]{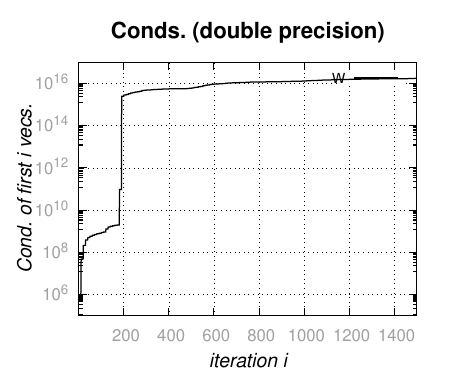}
    \caption{Condition number of test matrix $W$}
    \label{fig:conds64W}
  \end{subfigure}
  \begin{subfigure}[t]{.33\textwidth}
    \centering
    \includegraphics[width=\linewidth]{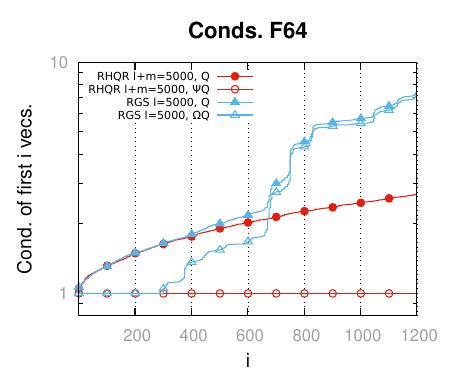}
    \caption{Condition number of bases output by RHQR and RGS}
    \label{fig:conds64}
  \end{subfigure}
  \begin{subfigure}[t]{.33\textwidth}
    \centering
    \includegraphics[width=\linewidth]{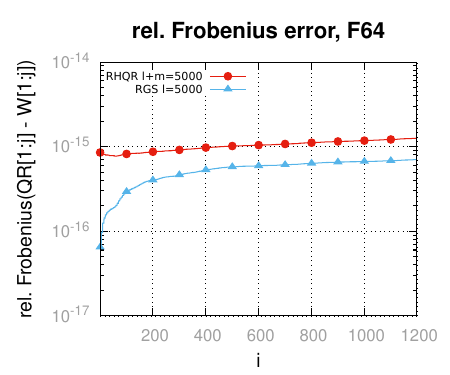}
    \caption{Factorization relative Frobenius error for RHQR and RGS}
    \label{fig:errors64}
  \end{subfigure}
  \caption{Randomized Householder QR in double precision (RHVector in the graphs) and comparison with RGS.}
  \label{fig:perf64}
\end{figure}

\Cref{fig:perf32} showcases the same experiments but in single precision. As to the comparison between RHQR and RGS, we make the same observations as in~\Cref{fig:perf64}, and we stress that the sketch of the RHQR basis remains numerically orthogonal. We also display the results obtained by Classical Gram-Schmidt (CGS), Modified Gram-Schmidt (MGS) and Householder QR (HQR). We can see that the basis computed by CGS is ill-conditioned very early in the iterations, yet the factorization remains accurate. The performance of RGS is similar to that of MGS, as pointed out in~\cite{rgs}. Finally, the basis produced by Householder QR is numerically orthogonal, just as the sketched basis output by RHQR, with an accurate factorization.

\begin{figure}
  \begin{subfigure}[t]{.33\textwidth}
    \centering
    \includegraphics[width=\linewidth]{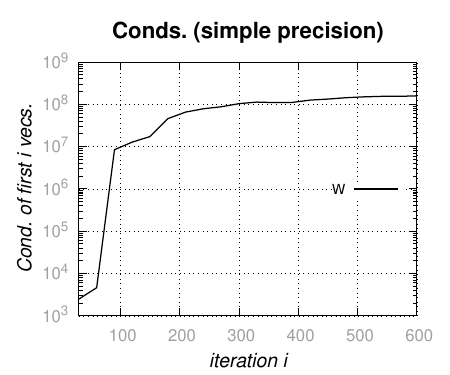}
    \caption{Condition number of test matrix $W$}
    \label{fig:conds32W}
  \end{subfigure}
  \begin{subfigure}[t]{.33\textwidth}
    \centering
    \includegraphics[width=\linewidth]{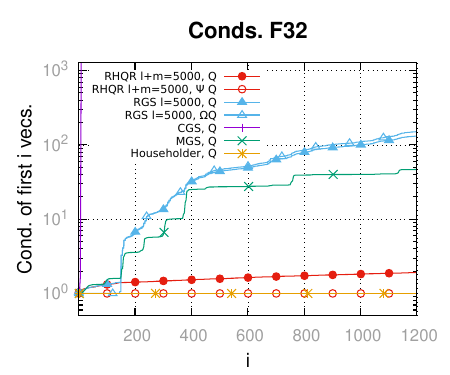}
    \caption{Condition number of bases output by RHQR and RGS}
    \label{fig:conds32}
  \end{subfigure}
  \begin{subfigure}[t]{.33\textwidth}
    \centering
    \includegraphics[width=\linewidth]{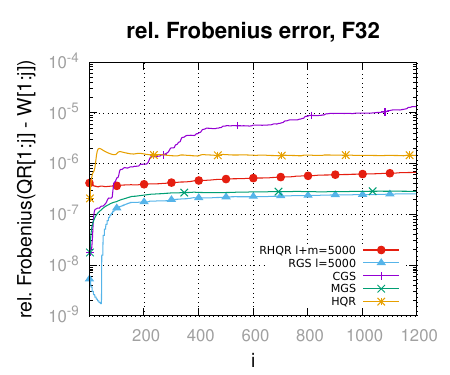}
    \caption{Factorization relative Frobenius error for RHQR and RGS}
    \label{fig:errors32}
  \end{subfigure}
  \caption{Randomized Householder QR in single precision (RHouse in the graphs) and comparison with RGS.}
  \label{fig:perf32}
\end{figure}

\Cref{fig:perfmixed} showcases the same experiment but in mixed precision. The input matrix $W = C_{1200}$ is given in half precision. All sketching operations are computed in half precision. The computation of $\rho_j$ in~\Cref{algo:rhousevector} (RHVector) is then done in double precision. It is integrated to the sketch $s_j$ in double precision, and to the randomized Householder vector $u_j$ in half precision. Then in~\Cref{algo:rhqr_leftlooking}, the matrix $T_j$ is computed in double precision. In turn, the factor $T_{j-1}^t S_{j-1}^t z$ in~\Cref{line:refreshwj} is computed in double precision, then converted in half precision to perform the update. We apply the same method to RGS: the input matrix is given in half precision, the sketching is made in half precision, the sketches are then converted in double precision. The least squares problem is computed in double precision. Its output is then converted back to half precision to perform the vector update. The result of the update is sketched again and converted to double precision. The norm of the sketch is computed in double precision, and the scaling of the obtained basis vector is performed in half precision. In~\cite{rgs}, authors mixed single and double precision, whereas we mix half and double precision on a more difficult matrix, and we sketch in half precision. We make the same observations as in~\Cref{fig:perf64,fig:perf32}, and stress again that the sketch of the basis output by RHQR remains numerically orthogonal.

\begin{figure}
\centering
  \begin{subfigure}[t]{.33\textwidth}
    \centering
    \includegraphics[width=\linewidth]{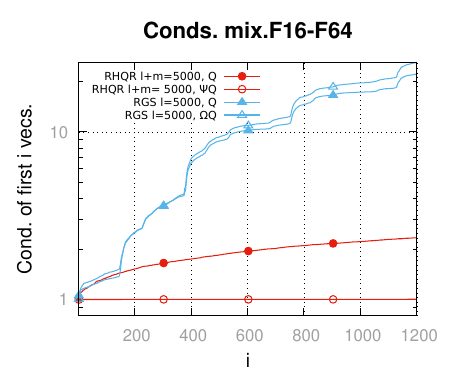}
    \caption{Condition number of bases output by RHQR and RGS}
    \label{fig:condsmixed}
  \end{subfigure}
  \begin{subfigure}[t]{.33\textwidth}
    \centering
    \includegraphics[width=\linewidth]{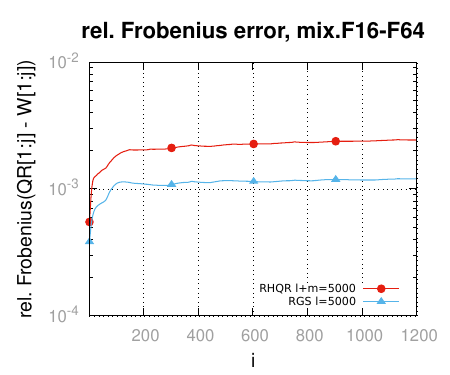}
    \caption{Factorization relative Frobenius error for RHQR and RGS}
    \label{fig:errorsmixed}
  \end{subfigure}
  \caption{Randomized Householder QR compared to RGS in mixed half/double precision.}
  \label{fig:perfmixed}
\end{figure}

\Cref{fig:reconstruct32} displays the performance of recRHQR in single precision on the matrix $C_{1200}$, and is compared to Randomized Cholesky QR since both algorithms perform a single synchronization. As detailed before, the recRHQR is solved with a backward solve using the scaling coefficients of the Householder vectors of $\Psi W$. The Randomized Cholesky QR is performed with a backward solve of the R factor retrieved from the QR factorization of the sketch. In~\Cref{fig:reconstruct32conds}, we see that both bases output by recRHQR and Randomized Cholesky QR lose their sketch orthogonality, but the former has a much more favorable trajectory than the latter. In~\Cref{fig:reconstruct32errs}, we see that both algorithms compute an accurate factorization, with a slight advantage for Randomized Cholesky QR.   

\begin{figure}
\centering
  \begin{subfigure}[t]{.33\textwidth}
    \centering
    \includegraphics[width=\linewidth]{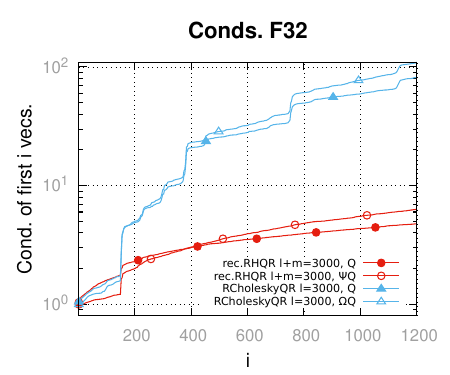}
    \caption{Condition number}
    \label{fig:reconstruct32conds}
  \end{subfigure}
  \begin{subfigure}[t]{.33\textwidth}
    \centering
    \includegraphics[width=\linewidth]{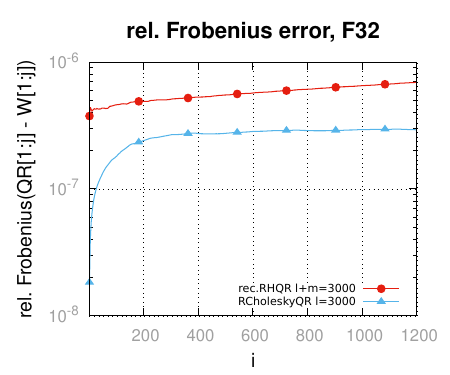}
    \caption{Accuracy of the factorization}
    \label{fig:reconstruct32errs}
  \end{subfigure}
  \caption{Performance of recRHQR compared to Randomized Cholesky QR}
  \label{fig:reconstruct32}
\end{figure}

\Cref{fig:gmres32} displays the performance of RHQR-GMRES (\Cref{algo:rhGMRES}) in single precision on the matrix \textit{SiO2} from~\cite{suitesparse}, compared to that of RGS-GMRES (based on \cite[Algorithm 3]{rgs}). The behavior of the two solvers is very similar. In~\Cref{fig:gmres32conds}, we display the condition number of $Q, \Psi Q$ (built by RHQR) and of $Q, \Omega Q$ (built by RGS). We see that both algorithms are able to produce a numerically orthogonal sketched Arnoldi basis. The condition number of the basis built by RHQR is slightly better than that built by RGS in the later iterations. In~\Cref{fig:gmres32errs}, we show the accuracy of the Arnoldi factorization $A Q_j = Q_{j+1} H_{j+1,j}$ computed by RHQR and RGS. As in the QR factorization of $W$, we see a slight advantage for RGS. Finally, in~\Cref{fig:gmres32res}, we showcase the convergence of GMRES solvers with RHQR and RGS variations. In the early iterations, both solvers are indistinguishable. Before the residual stagnation, we see RHQR performing slightly better than RGS. When residual stagnation is reached, we see that RGS has a slight advantage over RHQR. On the more difficult \textit{El3D} matrix, we observe the same results as in~\Cref{fig:perf64,fig:perf32,fig:perfmixed}, namely the numerical orthogonality of the sketched Arnoldi basis output by RHQR, the loss of orthogonality of that output by RGS, the accuracy of both factorizations with a slight advantage for RGS.

\begin{figure}
  \begin{subfigure}[t]{.33\textwidth}
    \centering
    \includegraphics[width=\linewidth]{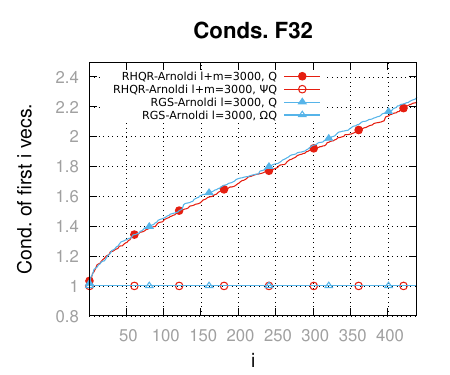}
    \caption{Condition number Arnoldi basis}
    \label{fig:gmres32conds}
  \end{subfigure}
  \begin{subfigure}[t]{.33\textwidth}
    \centering
    \includegraphics[width=\linewidth]{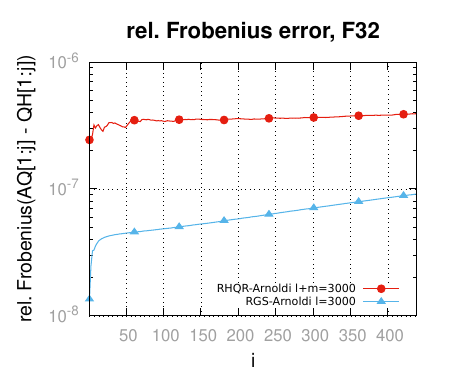}
    \caption{Accuracy of the Arnoldi relation}
    \label{fig:gmres32errs}
  \end{subfigure}
  \begin{subfigure}[t]{.33\textwidth}
    \centering
    \includegraphics[width=\linewidth]{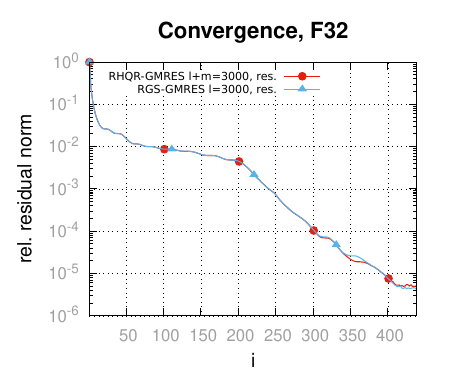}
    \caption{Convergence of GMRES}
    \label{fig:gmres32res}
  \end{subfigure}
  \caption{Performance of RHQR-GMRES compared to that of RGS-GMRES}
  \label{fig:gmres32}
\end{figure}

\Cref{fig:rmgs32} displays the performance of BLAS2-RGS (\Cref{algo:blas2rgs}) in single precision, on the matrix $C_{1200}$. As pointed out in~\Cref{section:rmgs}, the truly orthogonal matrix is $\left[I - T, \; \Omega Q T \right]$, which in exact arithmetics is $\left[0_{m \times m}; \; \Omega Q \right]$, i.e $\Omega Q$ would be orthogonal in exact arithmetics. In~\Cref{fig:rmgs32conds} we showcase (in the order given by the plot's legend) the condition numbers of $Q$, $\Omega Q$ output by BLAS2-RGS (i.e without any correction from $T$); the condition numbers of $Q T$ and $\Omega Q T$ output by BLAS2-RGS (a partial correction from $T$), the condition numbers of $\left[I - T; \; Q T \right]$ and $\left[I-T; \; \Omega Q T \right]$ output by BLAS2-RGS (full correction from $T$), and finally the comparison with the bases $Q$ and $\Omega Q$ built by RGS. We see that the sketch of the basis built by BLAS2-RGS with its full correction is numerically orthogonal. The basis built by BLAS2-RGS, both with partial correction and without correction, suffers some instabilities when the matrix $W$ becomes numerically singular, just like RGS. However, both bases built by BLAS2-RGS are better conditioned than that built by RGS. In~\Cref{fig:rmgs32errs}, we display the relative factorization error for the BLAS2-RGS basis without correction, with partial correction, and we compare it to that of RGS. We see an interesting feature: while the basis with partial correction is the better conditioned of the three, the factorization is slightly les accurate than the one without correction. The factorization without correction is as accurate as that computed by RGS.

\begin{figure}
\centering
  \begin{subfigure}[t]{.33\textwidth}
    \centering
    \includegraphics[width=\linewidth]{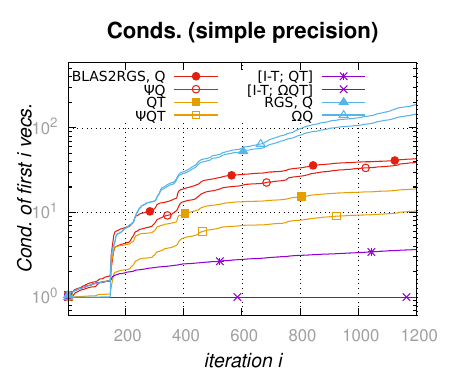}
    \caption{Condition number}
    \label{fig:rmgs32conds}
  \end{subfigure}
  \begin{subfigure}[t]{.33\textwidth}
    \centering
    \includegraphics[width=\linewidth]{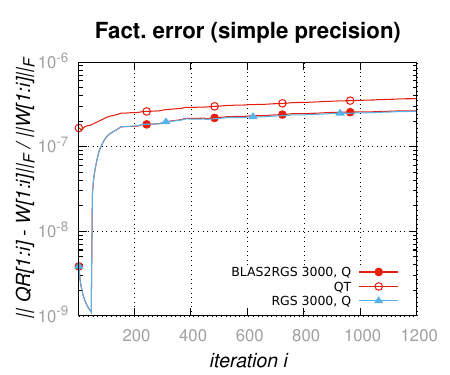}
    \caption{Accuracy of the factorization}
    \label{fig:rmgs32errs}
  \end{subfigure}
  \caption{Performance of BLAS2-RGS compared to RGS}
  \label{fig:rmgs32}
\end{figure}

\Cref{fig:perf64bis} displays the accuracy obtained in double precision for both trimRHQR (\Cref{algo:rhqrbis_rightlooking,algo:rhqrbis_leftlooking}) and RGS, in terms of condition number of the computed basis and accuracy of the factorization, on the set of synthetic functions. \Cref{fig:conds64bis} displays in red (and circle points) the condition number of the basis obtained through RHQR from \Cref{algo:rhqrbis_leftlooking}, and in blue (and triangle points) that obtained by RGS for reference. In this experiment, the linear least squares problem of RGS is solved as before with LAPACK pivoted QR. We make the same global observations as in~\Cref{fig:perf64}. We note that the condition number of the basis output by trimRHQR is, in its stable regime, higher than that of both RHQR and RGS. Most importantly, we see that trimRHQR with a sampling size $\ell = 1000 < 1500 = m$ is stable, and stabler than RGS with a sampling size of $1550 > m$. We make the same observation in finite precision in~\Cref{fig:perf32bis}. 
\begin{figure}
\centering
  \begin{subfigure}[t]{.33\textwidth}
    \centering
    \includegraphics[width=\linewidth]{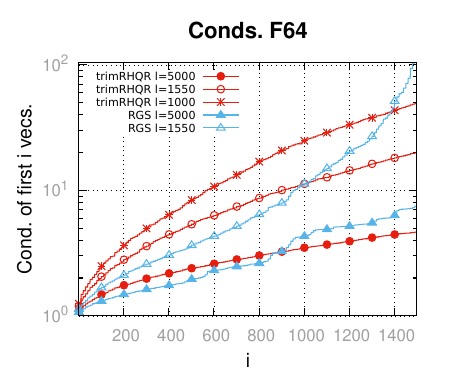}
    \caption{Condition number of the bases output by trimRHQR and RGS}
    \label{fig:conds64bis}
  \end{subfigure}
  \begin{subfigure}[t]{.33\textwidth}
    \centering
    \includegraphics[width=\linewidth]{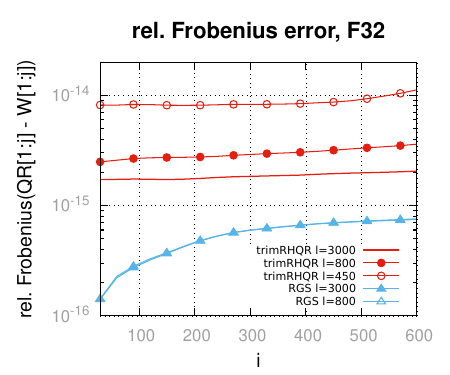}
    \caption{Factorization error of trimRHQR and RGS}
    \label{fig:errors64bis}
  \end{subfigure}
  \caption{Performance of trimRHQR and RGS in double precision}
  \label{fig:perf64bis}
\end{figure}

\Cref{fig:perf32bis} showcases the performance of~\Cref{algo:rhqrbis_leftlooking} in single precision. We make the same observations as in double precision (remark that trimRHQR with sampling size $\ell = 300 < 600$ outperforms RGS with $\ell = 800$).

\begin{figure}
\centering
  \begin{subfigure}[t]{.33\textwidth}
    \centering
    \includegraphics[width=\linewidth]{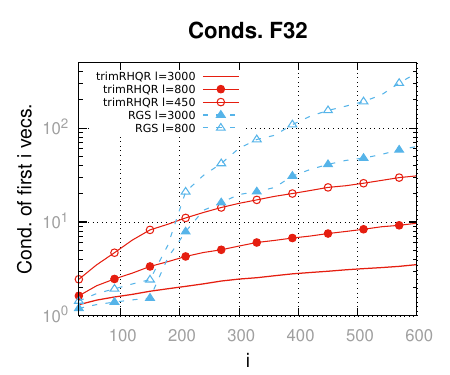}
    \caption{Condition number of the bases computed by trimRHQR and RGS}
    \label{fig:conds32bis}
  \end{subfigure}
  \begin{subfigure}[t]{.33\textwidth}
    \centering
    \includegraphics[width=\linewidth]{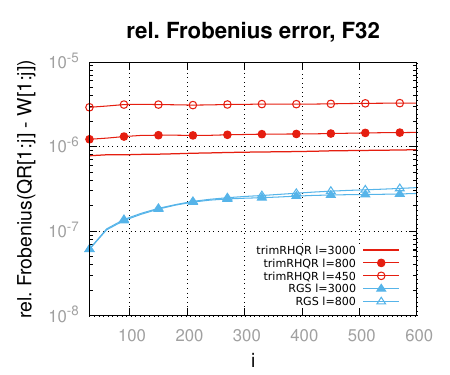}
    \caption{Factorization error of trimRHQR and RGS}
    \label{fig:errors32bis}
  \end{subfigure}
  \caption{Performance of trimRHQR and RGS in single precision}
  \label{fig:perf32bis}
\end{figure}

\section{Acknowledgements}
This project has received funding from the European Research
Council (ERC) under the European Union’s Horizon 2020 research and innovation program (grant agreement No 810367).

\printbibliography

@InProceedings{bsrht,
  author = 	 {Balabanov, Oleg and Beaupere, Matthias and Grigori, Laura and Lederer, Victor},
  title = 	 {Block subsampled randomized Hadamard transform for low-rank approximation on distributed architectures},
  booktitle = {Proceedings of the 40th International Conference on Machine Learning},
  year = 	 2023,
  organization = {PMLR 202},
  url = 	 {https://doi.org/10.48550/arxiv.2210.11295}
}

@article{rgs,
  title={Randomized Gram--Schmidt Process with Application to GMRES},
  author={Balabanov, Oleg and Grigori, Laura},
  journal={SIAM Journal on Scientific Computing},
  volume={44},
  number={3},
  pages={A1450--A1474},
  year={2022},
  publisher={SIAM}
}

@article{puglisi,
  title={Modification of the Householder method based on the compact WY representation},
  author={Puglisi, Chiara},
  journal={SIAM Journal on Scientific and Statistical Computing},
  volume={13},
  number={3},
  pages={723--726},
  year={1992},
  publisher={SIAM}
}

@article{woodruff,
  author    = {David P. Woodruff},
  title     = {Sketching as a Tool for Numerical Linear Algebra},
  journal   = {CoRR},
  volume    = {abs/1411.4357},
  year      = {2014},
  url       = {http://arxiv.org/abs/1411.4357},
  eprinttype = {arXiv},
  eprint    = {1411.4357},
  bibsource = {dblp computer science bibliography, https://dblp.org}
}

@book{higham,
  title={Accuracy and stability of numerical algorithms},
  author={Higham, Nicholas J},
  year={2002},
  publisher={SIAM}
}

@book{saad,
  title={Iterative methods for sparse linear systems},
  author={Saad, Yousef},
  year={2003},
  publisher={SIAM}
}

@article{suitesparse,
  title={The University of Florida sparse matrix collection},
  author={Davis, Timothy A and Hu, Yifan},
  journal={ACM Transactions on Mathematical Software (TOMS)},
  volume={38},
  number={1},
  pages={1--25},
  year={2011},
  publisher={ACM New York, NY, USA}
}

@incollection{wilkinson,
  title={The algebraic eigenvalue problem},
  author={Wilkinson, JH},
  booktitle={Handbook for Automatic Computation, Volume II, Linear Algebra},
  year={1971},
  publisher={Springer-Verlag New York}
}

@article{paige,
  title={A useful form of unitary matrix obtained from any sequence of unit 2-norm n-vectors},
  author={Paige, Christopher C},
  journal={SIAM Journal on Matrix Analysis and Applications},
  volume={31},
  number={2},
  pages={565--583},
  year={2009},
  publisher={SIAM}
}

@article{barlow,
  title={Block Modified Gram--Schmidt Algorithms and Their Analysis},
  author={Barlow, Jesse L},
  journal={SIAM Journal on Matrix Analysis and Applications},
  volume={40},
  number={4},
  pages={1257--1290},
  year={2019},
  publisher={SIAM}
}

@article{carson,
  title={Block Gram-Schmidt algorithms and their stability properties},
  author={Carson, Erin and Lund, Kathryn and Rozlo{\v{z}}n{\'{i}}k, Miroslav and Thomas, Stephen},
  journal={Linear Algebra and its Applications},
  volume={638},
  pages={150--195},
  year={2022},
  publisher={Elsevier}
}

@article{schriebervanloan,
  title={A storage-efficient WY representation for products of Householder transformations},
  author={Schreiber, Robert and Van Loan, Charles},
  journal={SIAM Journal on Scientific and Statistical Computing},
  volume={10},
  number={1},
  pages={53--57},
  year={1989},
  publisher={SIAM}
}

@article{rokhlintygert,
  title={A fast randomized algorithm for overdetermined linear least-squares regression},
  author={Rokhlin, Vladimir and Tygert, Mark},
  journal={Proceedings of the National Academy of Sciences},
  volume={105},
  number={36},
  pages={13212--13217},
  year={2008},
  publisher={National Acad Sciences}
}

@article{ailon2009,
  title={The fast Johnson--Lindenstrauss transform and approximate nearest neighbors},
  author={Ailon, Nir and Chazelle, Bernard},
  journal={SIAM Journal on computing},
  volume={39},
  number={1},
  pages={302--322},
  year={2009},
  publisher={SIAM}
}

@inproceedings{sparsejlt,
  title={A sparse johnson: Lindenstrauss transform},
  author={Dasgupta, Anirban and Kumar, Ravi and Sarl{\'o}s, Tam{\'a}s},
  booktitle={Proceedings of the forty-second ACM symposium on Theory of computing},
  pages={341--350},
  year={2010}
}

@article{troppsrht,
  title={Improved analysis of the subsampled randomized Hadamard transform},
  author={Tropp, Joel A},
  journal={Advances in Adaptive Data Analysis},
  volume={3},
  number={01n02},
  pages={115--126},
  year={2011},
  publisher={World Scientific}
}

@article{brgs,
  title={Randomized block Gram-Schmidt process for solution of linear systems and eigenvalue problems},
  author={Balabanov, Oleg and Grigori, Laura},
  journal={arXiv preprint arXiv:2111.14641},
  year={2021}
}

@article{tsqr,
  title={Communication-optimal parallel and sequential QR and LU factorizations},
  author={Demmel, James and Grigori, Laura and Hoemmen, Mark and Langou, Julien},
  journal={SIAM Journal on Scientific Computing},
  volume={34},
  number={1},
  pages={A206--A239},
  year={2012},
  publisher={SIAM}
}

@article{melnichenko,
  title={CholeskyQR with Randomization and Pivoting for Tall Matrices (CQRRPT)},
  author={Melnichenko, Maksim and Balabanov, Oleg and Murray, Riley and Demmel, James and Mahoney, Michael W and Luszczek, Piotr},
  journal={arXiv preprint arXiv:2311.08316},
  year={2023}
}

@article{martinssonqrcp,
  title={Householder QR factorization with randomization for column pivoting (HQRRP)},
  author={Martinsson, Per-Gunnar and Quintana OrtÍ, Gregorio and Heavner, Nathan and Van De Geijn, Robert},
  journal={SIAM Journal on Scientific Computing},
  volume={39},
  number={2},
  pages={C96--C115},
  year={2017},
  publisher={SIAM}
}

@article{guqrcp,
  title={Randomized QR with column pivoting},
  author={Duersch, Jed A and Gu, Ming},
  journal={SIAM Journal on Scientific Computing},
  volume={39},
  number={4},
  pages={C263--C291},
  year={2017},
  publisher={SIAM}
}

@article{multisketching,
  title={Analysis of Randomized Householder-Cholesky QR Factorization with Multisketching},
  author={Higgins, Andrew J and Szyld, Daniel B and Boman, Erik G and Yamazaki, Ichitaro},
  journal={arXiv preprint arXiv:2309.05868},
  year={2023}
}

@article{balabanovqr,
  title={Randomized Cholesky QR factorizations},
  author={Balabanov, Oleg},
  journal={arXiv preprint arXiv:2210.09953},
  year={2022}
}

@article{yujijoel,
  title={Fast \& accurate randomized algorithms for linear systems and eigenvalue problems},
  author={Nakatsukasa, Yuji and Tropp, Joel A},
  journal={arXiv preprint arXiv:2111.00113},
  year={2021}
}

@article{randlapack,
  title={Randomized numerical linear algebra: A perspective on the field with an eye to software},
  author={Murray, Riley and Demmel, James and Mahoney, Michael W and Erichson, N Benjamin and Melnichenko, Maksim and Malik, Osman Asif and Grigori, Laura and Luszczek, Piotr and Derezi{\'n}ski, Micha{\l} and Lopes, Miles E and others},
  journal={arXiv preprint arXiv:2302.11474},
  year={2023}
}

@article{foundations,
  title={Randomized numerical linear algebra: Foundations and algorithms},
  author={Martinsson, Per-Gunnar and Tropp, Joel A},
  journal={Acta Numerica},
  volume={29},
  pages={403--572},
  year={2020},
  publisher={Cambridge University Press}
}

@article{findingstructure,
  title={Finding structure with randomness: Probabilistic algorithms for constructing approximate matrix decompositions},
  author={Halko, Nathan and Martinsson, Per-Gunnar and Tropp, Joel A},
  journal={SIAM review},
  volume={53},
  number={2},
  pages={217--288},
  year={2011},
  publisher={SIAM}
}

@article{ilse,
  title={The effect of coherence on sampling from matrices with orthonormal columns, and preconditioned least squares problems},
  author={Ipsen, Ilse CF and Wentworth, Thomas},
  journal={SIAM Journal on Matrix Analysis and Applications},
  volume={35},
  number={4},
  pages={1490--1520},
  year={2014},
  publisher={SIAM}
}

\appendix

\section*{Appendix : Detail of proofs }\label{section:appendix}
\textbf{Proof of~\Cref{lemma:accuracyrhvector}} :
Denote $k_1 = k + \ell/2 + 1$
Let us denote $a$ the first entry of $w$, and $\rho = \| \Psi w \|$. According to previous derivations, we get
$$\widehat{\rho} = (1+\theta_{k_1}) \rho, \quad \fl (a - \rho) = (1+\theta_{k_1+1})(a - \rho)$$ 
$$\fl \sqrt{\rho (a - \rho)} = (1+\delta) \sqrt{ 1 + \theta_{2k_1+2}} \sqrt{ \rho (a - \rho)} = (1+\delta)(1+\theta_{k_1+1}) \sqrt{\rho(a-\rho)} = (1+\theta_{k_1+2}) \sqrt{ \rho (a-\rho)}$$
Let us denote $z = w - \| \Psi w \| e_1$ and $\beta_z = 2/\| \Psi w \|^2 = 2/\rho^2$.
$$\| \lambda \Psi z \|^2 = 2 \iff \lambda^2 = \frac{2}{\| \Psi z \|^2} = \beta_z \iff \lambda = \pm \beta_z,$$
hence if we want to output $u = \lambda z$ such that $\| \Psi u \|^2 = 2$, we should multiply $z$ by $\sqrt{\beta_z}$. Straightforward computations show that
$$\beta_z = \frac{1}{\rho(\rho-a)}, \quad \sqrt{\beta}_z = \frac{1}{\sqrt{\rho(\rho-a)}}.$$
We can compute the inverse of $\rho(a-\rho)$ and then take the square root, yielding $\widehat{\sqrt{\beta_z}} = (1+\theta_{2k_1+2}) \sqrt{\beta_z}$. However it is slightly more accurate to divide the entries of $z$ by $\sqrt{\rho(\rho-a)}$. Let us now denote $b \in \R^{m-1}$ the entries $2$ to $m$ of $w$, and $c \in \R^{n-m}$ the last $n-m$ entries of $w$. In memory, we have these two vectors:
$$\widehat{z} = \begin{bmatrix} (1+\theta_{k_1+1}) (a - \rho) \\ b \\ c \end{bmatrix}, \quad \fl \Psi z = \begin{bmatrix} (1+\theta_{k_1+1}) (a-\rho) \\ b \\ \Omega c + \Delta d \end{bmatrix}, \quad \| \Delta d \| \leq \gamma_k \| \Omega c \|$$
Let us then scale them by dividing their entries by $\sqrt{\rho(\rho-a)}$. As to their first entry, the most accurate computation we can do is the following:
$$\sqrt{\beta_z}(a-\rho) = \sqrt{\frac{a}{\rho} - 1}.$$
Accounting for the computations of $a/\rho$, the substraction of $1$, and the square root, we get:
$$ \fl \sqrt{\frac{a}{\rho} - 1} = (1+\theta_{k_1+2}) \sqrt{\frac{a}{\rho}-1}$$
(we just avoided a final error on the first entry of the order of $3k_1$). Then we can straightforwardly divide all the other entries by $\sqrt{\rho(\rho-a)}$, yielding:
$$\widehat{u} = \begin{bmatrix} (1+\theta_{k_1+2}) \sqrt{\beta_z} (a-\rho) \\ \mathrm{Diag}(1+\theta_{2k_1+5}) \sqrt{\beta_z} b \\ \mathrm{Diag}(1+\theta_{2k_1+5}) \sqrt{\beta_z} c \end{bmatrix},$$
hence
$$| u - \widehat{u} | \leq \gamma_{2k_1+5} |u| = \gamma_{2k + \ell + 7} |u| \quad (\text{scaling } \|s\|^2 = 2).$$
As for the sketch, only the last $n-m$ coordinates differ (by design of $\Psi$), and we obtain:
$$ \begin{bmatrix} (1+\theta_{k_1+2}) \sqrt{\beta_z} (a-\rho) \\ \mathrm{Diag}(1+\theta_{2k_1+5}) \sqrt{\beta_z} b \\ \mathrm{Diag}(1+\theta_{2k_1+5})\sqrt{\beta_z} (\Omega c + \Delta d) \end{bmatrix}.$$
The final error on the last $\ell$ entries of $ \fl \Psi u$ is:
$$\sqrt{\beta_z} \Omega c + \sqrt{\beta_z} \cdot \Delta d + \sqrt{\beta_z} \cdot \mathrm{Diag}(\theta_{2k_1+5})(\Omega c + \Delta d),$$
whose norm is bounded by
$$\sqrt{\beta_z} \| \Omega c \| \cdot \left[ \gamma_k + \gamma_{2k_1+5}(1+\gamma_k) \right] \leq \gamma_{2k_1 + k + 5} \sqrt{\beta_z} \| \Omega c \|,$$
hence
$$ \|s - \widehat{s} \| = \| \Psi u - \fl \Psi u \| \leq \gamma_{2k_1 + k + 5} \| \Psi u \|.$$
Let us now detail the scaling in precision $\f$. The computations on $\Psi u$ and on the first $m$ coordinates of $u$ are unchanged. As proposed above, we may compute $\rho(a-\rho)$ and take its square root, then cast $\sqrt{\beta_z}$ to precision $\f$. Supposing that $(2k_1+2) \e \leq \f$, we get $\widehat{\sqrt{\beta_z}} = (1+\zeta) \sqrt{\beta_z}$ (see~\eqref{eq:notationchi}), and accounting for the scaling error we get
$$\widehat{u} = \begin{bmatrix} (1+\theta_{k_1+2}) \sqrt{\beta_z} (a - \rho) \\ \mathrm{Diag}(1+\theta_{2k_1+5}) \sqrt{\beta_z} b \\ \mathrm{Diag}(1+\eta_2) \sqrt{\beta_z} c \end{bmatrix}$$
yielding
$$| \widehat{u} - u | \leq \chi_2 \|u\|.$$
Let us move on to the second scaling, where $\langle u, e_1 \rangle = 1$. We might as well write $1$ directly to the first entry of the output. As to the other coordinates of $z$, we simply divide them all by $\fl(a - \rho)$, yielding
$$\widehat{u} = \begin{bmatrix} 1 \\ \mathrm{Diag}(1+\theta_{2k_1+3}) b / (a-\rho) \\ \mathrm{Diag}(1+\theta_{2k_1+3}) c / (a-\rho) \end{bmatrix}, \quad \fl \Psi u = \begin{bmatrix} 1 \\ \mathrm{Diag}(1+\theta_{2k_1+3}) b / (a-\rho) \\ \mathrm{Diag}(1 + \theta_{2k_1+3}) (\Omega c + \Delta d)/(a-\rho) \end{bmatrix},$$
thus $|\widehat{u}-u| \leq \gamma_{2k_1+3} |u|$. The final error in the last $n-m$ coordinates of $\fl \Psi u$ writes:
$$ \frac{1}{a - \rho} \cdot \Delta d + \mathrm{Diag}(\theta_{2k_1+3}) \frac{1}{a - \rho} \cdot (\Omega c + \Delta d),$$
whose norm is bounded by
$$\Bigl\| \Omega \Bigl( \frac{c}{a-\rho} \Bigr) \Bigr\| \cdot  \left[\gamma_k + \gamma_{2k_1+3} (1+\gamma_k) \right] \leq \gamma_{2k_1+k+3} \cdot \Bigl\| \Omega \Bigl( \frac{c}{a-\rho} \Bigr) \Bigr\|,$$
hence $ \| \Psi u - \fl \Psi u \| \leq \gamma_{2k_1+k+3} \| \Psi u \|$, showing that both scalings are equally accurate in practice. Again, in mixed precision, only the last $n-m$ coordinates of $\widehat{u}$ are changed. We cast $(a - \rho)^{-1}$ to simple precision, and supposing that $(2k_1 + 3) \e \leq \f$, we get similar $\| \widehat{u} - u \| \leq \chi_2 \| u \|$. In this scaling, we have to return a $\beta_u := 2/\| \Psi u \|^2$, slightly different from $1$. Since $u = (a - \rho)^{-1} z$ and $\| \Psi z \|^2 = 2 \rho(\rho-a)$, we get, in exact arithmetics,
$$\| \Psi u \|^2 = 2 \frac{\rho}{\rho-a} \implies \beta_u := \frac{2}{\| \Psi u \|^2} = \frac{\rho - a}{\rho} = 1 - \frac{a}{\rho},$$
that we might compute directly from $\widehat{\rho}$ and $a$:
$$\fl \frac{a}{\rho} = \frac{1+\delta}{1+\theta_{k_1}} \frac{a}{\rho} = (1+\theta_{2k_1+1}) \frac{a}{\rho},$$
and finally
$$\widehat{\beta}_u = (1+\theta_{2k_1+2}) \beta_u.$$

\textbf{Proof of~\Cref{lemma:accuracyapplicationreflector} :}
    We first denote 
    $$ \alpha = \langle \Psi x, \Psi u \rangle, \quad \widehat{\alpha} = \fl \Bigl( \langle \widehat{\Psi x}, \widehat{s} \rangle \Bigr) =: \langle \Psi x, \Psi u \rangle + \Delta \alpha$$ 
    Accounting also for the error made by the computation of the inner product of $\R^{\ell}$, we get
    $$\widehat{\alpha} = \langle \Psi u + \Delta_1 z, \Psi x + \Delta_2 z + \Delta_3 z \rangle, \quad \begin{dcases} \| \Delta_1 z \| \leq \sqrt{2} \cdot \gamma_{3k + \ell + 7} \\ \| \Delta_2 z \| \leq \gamma_k \| \Psi x \| \\ \| \Delta_3 z \| \leq \gamma_{\ell} (1 + \gamma_k) \| \Psi x \| \end{dcases}$$
    hence 
    $$| \Delta \alpha | \leq \gamma_{3k + \ell + 7 + \ell + k} \sqrt{2} \| \Psi x \| \leq \gamma_{4k + 2\ell + 7} \sqrt{2} \| \Psi x \|.$$
    Let us, for the remaining of the proof, denote $k_2 = 4k + 2\ell+7$. We now compute, in the unique precision setting,
    $$\fl ( \widehat{\alpha} \widehat{u}) = \mathrm{Diag}(1+\delta)(\alpha + \Delta \alpha)(u + \Delta u) := \alpha u + \Delta s$$
    We begin by addressing $\widehat{\alpha} \widehat{u} - \alpha u$:
    \begin{align*} | \widehat{\alpha} \widehat{u} - \alpha u | & \leq \underbrace{\sqrt{2}\|\Psi x \| \cdot \|u\| \gamma_{2k + \ell + 7}}_{\alpha \cdot \Delta u} + \underbrace{\sqrt{2}\|\Psi x\| \cdot \|u\| \gamma_{k_2}}_{\Delta \alpha \cdot u} + \underbrace{\sqrt{2} \|\Psi x \| \cdot \|u\| \gamma_{k_2} \gamma_{2k + \ell+7}}_{\Delta \alpha \cdot \Delta u}  \\ & \leq \sqrt{2} \|\Psi x \| \cdot \|u\| \cdot \gamma_{2k + \ell + 7 + k_2}  = \sqrt{2} \|\Psi x \| \cdot \|u\| \cdot \gamma_{6k + 3\ell + 14}.\end{align*}
    From this we also infer that $| \widehat{\alpha} \widehat{u} | \leq \sqrt{2} \|\Psi x \| \cdot \| u \| \cdot (1+\gamma_{6k + 3\ell + 14})$. Since $\Delta s = (\widehat{\alpha}\widehat{u} - \alpha u) + \mathrm{Diag}(\delta)(\widehat{\alpha}\widehat{u})$,
    $$|\Delta s| \leq \sqrt{2}\|\Psi x \| \cdot \|u\| \cdot \Bigl( \gamma_{6k + 3\ell + 14} + \e(1 + \gamma_{6k + 3\ell + 14}) \Bigr) = \sqrt{2} \|\Psi x \| \cdot \|u\| \cdot \gamma_{6k + 3\ell + 15} \leq 2 \cdot \frac{1}{1-\epsilon} \cdot \|\Psi x \| \cdot \gamma_{6k + 3\ell + 15}(1+\gamma_{2k + \ell + 7})$$
    Let us address the mixed precision derivations as well : consider that $\gamma_{4k + 2\ell + 7} \leq \f$ and that $\alpha$ is casted to precision $\f$, yielding $\widehat{\alpha} = (1+\zeta)\alpha$. Considering that the scaling in lower precision is $\mathrm{Diag}(1+\zeta)$ instead of $\mathrm{Diag}(1+\delta)$, we simply get $|\Delta s| \leq \sqrt{2} \|\Psi x \| \cdot \|u\| \cdot \chi_5 \leq 2(1-\epsilon)^{-1} \|\Psi x \| \cdot \chi_5(1+\chi_2)$.
    Moving on to $x - \alpha u$. Taking into account the error made by difference of the two vectors, we get
    $$\fl(x - \alpha u) = \mathrm{Diag}(1+\delta)\Bigl(x - \fl(\widehat{\alpha}\widehat{u}) \Bigr) = x - \alpha u - \Delta s + \mathrm{Diag}(\delta) (x - \alpha u - \Delta s) = y + \Delta s + \mathrm{Diag}(\delta)(y+\Delta s)$$
    (the sign of the entries of $\Delta s$ is not important), more simply put as
    \begin{align} \label{eq:deltay} \widehat{y} = y + \Delta y, \quad \Delta y := \Delta s + \mathrm{Diag}(\delta)(y + \Delta s)\end{align}
    (for mixed precision, just replace $\delta$ by $\zeta$). We want to bound $\Delta y$ in terms of $\| \Psi x \|$. Meanwhile $|\Delta s|$ is already expressed in term of $\|\Psi x \|$. We mention here that, as in the deterministic case, we have the opportunity of expressing $|\Delta y|<|y|$, component-wise, even if we don't use it in this proof, as $\Delta s$ is bounded component-wise in terms of $\|\Psi x\|$ and $\|y\| = (1+\mathcal{O}(\epsilon)) \| \Psi y \| = (1+\mathcal{O}(\epsilon)) \| \Psi P x \| = (1+\mathcal{O}(\epsilon)) \| \Psi x \|$. Nevertheless, let us proceed and simply bound $\| \Delta y \|$ in terms of $\| \Psi x \|$. Using the $\epsilon$-embedding property on $y$, we get
    $$\|\Delta y \| \leq \begin{dcases} \| \Psi x \| \cdot  \frac{2}{1-\epsilon} \gamma_{6k + 3\ell +15}(1 +\gamma_{2k + \ell + 7}) \\ \\ \| \Psi x \| \cdot \frac{2}{1-\epsilon} \cdot \chi_5(1+\chi_2)  \end{dcases} $$ 
    Hence, using the $\epsilon$-embedding property on $\Delta y$,
    $$\Psi \widehat{y} = y + \Delta'y, \quad \| \Delta' y \| = \| \Psi \cdot \Delta y \| \leq \begin{dcases} \| \Psi x \| \cdot \frac{1+\epsilon}{1-\epsilon} \cdot  \cdot \gamma_{12k + 6\ell + 30}(1+\gamma_{2k + \ell + 7}) \\ \\ \| \Psi x \| \cdot \frac{1+\epsilon}{1-\epsilon} \cdot  \cdot \chi_{10}(1+\chi_2)\end{dcases} $$
    According to rounding error analysis,
    $$\begin{dcases} \gamma_{12k + 6\ell + 30} + \gamma_{12k + 6\ell + 30} \cdot \gamma_{2k + \ell + 7} = \gamma_{12k + 6\ell + 30} + \gamma_{2k + \ell + 7} \leq \gamma_{14k + 7 \ell + 37} \\
    \chi_{10} + \chi_{10} \chi_2 = \chi_{10} + \chi_2 \leq \chi_{12} \end{dcases}$$
    yielding the final
    $$\Psi \widehat{y} = y + \Delta'y, \quad \| \Delta' y \| \leq \begin{dcases} \| \Psi x \| \cdot \frac{1+\epsilon}{1-\epsilon} \cdot \gamma_{14k + 7\ell + 37}\\ \\ \| \Psi x \| \cdot \frac{1+\epsilon}{1-\epsilon} \chi_{12}\end{dcases} $$
    In both single precision and mixed precision, the result follows from setting
    $$ \Delta P := \frac{1}{\| \Psi x \|^2} \cdot (\Psi \cdot \Delta y) \cdot (\Psi x)^t$$
    and observing that $\Delta P \cdot \Psi x = \Psi \Delta y$, thus
    $$\Psi \widehat{y} = \Psi y + \Psi \cdot \Delta y = \Psi \cdot P(u, \Psi) \cdot x + \Delta P \cdot \Psi x = P(\Psi u) \cdot \Psi x + \Delta P \cdot \Psi x = \Bigl(P(\Psi u) + \Delta P \Bigr) \cdot \Psi x.$$

\textbf{Proof of~\Cref{lemma:srht}:} We first derive the accuracy of the normalized Hadamard transform, in natural order, with its normalization performed in one step after the end of the recursive calls. We proceed by induction on $p$, where the dimension $n$ of the input is $n = 2^p$. All the computations are done in higher precision $\e$ and using worst-case rounding error analysis. For all $p$, we denote $H_p$ the non-normalized Hadamard transform of $\R^{2^p}$, that is
$$H_1 := \begin{bmatrix} 1 & 1 \\ 1 & -1 \end{bmatrix} \in \R^{2 \times 2}, \quad H_{p+1} = \begin{bmatrix} H_p & H_p \\ H_p & -H_p \end{bmatrix} \in \R^{2^{p+1} \times 2^{p+1}} \quad p\geq 1.$$
For $p = 1$ and $x \in \R^2$, $x = x_1 \oslash x_2$ we must account for the errors
$$\mathbf{fl}(x_1 + x_2) = (1+\delta)(x_1 + x_2), \quad \mathbf{fl}(x_1 - x_2) = (1+\delta)(x_1 - x_2).$$
Now, let us do the computations of the $2$-dimensional non-normalized Hadamard transform,
$$\widehat{ H_1 x} = \begin{bmatrix} (1+\delta)(x_1 + x_2) \\ (1+\delta) (x_1 - x_2)\end{bmatrix} = H_1 x + \Delta y, \quad | \Delta y | \leq \e |H_1 x|.$$
By setting $\Delta x := H^{-1} \Delta y$ we get 
$$\widehat{H_1 x} = H_1(x + \Delta x), \quad \| \Delta x \| \leq \e \cdot \|x\|.$$
Let us assume now that for a given $p$, and a given $x \in \R^{2^p}$, we get
$$\widehat{H_p x} = H_p (x + \Delta x), \quad \| \Delta x \| \leq \gamma_p \cdot\|x\| \quad \text{(normwise)}.$$
Let then $x, y \in \R^{2^p}$ and $\Delta x, \Delta y$ their associated error verifying the induction hypothesis. We get
$$\fl \left( H_{p+1} \begin{bmatrix} x \\ y \end{bmatrix} \right) = \mathrm{Diag}(1+\delta) \begin{bmatrix}  H_p x + H_p y \; \; + \; \; H_p \cdot \Delta x + H_p \cdot \Delta y \\ H_p x - H_p y \; \; + \; \;  H_p \cdot \Delta x + H_p \cdot \Delta y \end{bmatrix},$$ 
yielding
$$ \fl \Bigl( H_{p+1} \begin{bmatrix} x \\ y \end{bmatrix} \Bigr) = H_{p+1} \begin{bmatrix} x \\ y \end{bmatrix} + H_{p+1} \begin{bmatrix} \Delta x \\ \Delta y \end{bmatrix} + \mathrm{Diag}(\delta) H_{p+1} \Bigl( \begin{bmatrix} x \\ y \end{bmatrix} + \begin{bmatrix} \Delta x \\ \Delta y \end{bmatrix} \Bigr). $$
Let us denote $z = x \oslash y$ and $\Delta z = \Delta x \oslash \Delta y$. Seeing as
$$\begin{dcases} \| H_{p+1} \cdot \Delta z \| = \sqrt{2} \| \Delta z  \| \leq \sqrt{2} \cdot \gamma_p \|z\| \\
\| \mathrm{Diag}(\delta) \cdot H_{p+1} (z +\Delta z) \| \leq \sqrt{2} \cdot \e(1+\gamma_p) \|z\| \end{dcases}$$
we get
$$\widehat{H_{p+1} z } = H_{p+1} z + \Delta t, \quad \| \Delta t \| \leq \sqrt{2} \gamma_{p+1} \|z\|.$$
Setting $\Delta' z = H_{p+1}^{-1} \cdot \Delta t = \sqrt{2} H_{p+1}^t \cdot \Delta t$, we get
$$\widehat{H_{p+1} z} = H_{p+1}(z + \Delta'z), \quad \| \Delta' z \| \leq \gamma_{p+1} \|z\|,$$
which concludes the induction. The final scaling by $\sqrt{n}$ can make two errors, whether $n$ is an even or odd power of $2$. In the worst case, the two errors come from the computation of $\sqrt{n} = 2^{-p/2}$, and the second comes from the scaling. Denoting simply $H$ the normalized Hadamard transform of $\R^{2^p} = \R^n$, we get
$$\forall x \in \R^n, \quad \widehat{Hx} = H(x+\Delta x) = Hx + \Delta y, \quad \| \Delta x \| \leq \gamma_{\log_2(n)+2} \|x\|, \quad \|\Delta y \| \leq \gamma_{\log_2(n)+2} \| Hx\| = \gamma_{\log_2(n)+2} \|x\|.$$
Let us now discuss the SRHT matrix $\Omega$ as described in the lemma. Let us denote $S : x \mapsto \sqrt{n/\ell} \cdot x$. The first scaling makes a backward error of magnitude $\gamma_3$, as three errors are made: the computation of $n/\ell$, the computation of its square root, and the scaling by the result. The random flip of signs is exact. So far we have
$$\fl \left( D S \cdot x \right) = DS(x + \Delta_1 x), \quad \| \Delta_1 x \| \leq \gamma_3 \|x\|$$
Let us denote $x' = DS(x + \Delta_1 x)$. Applying the normalized Hadamard transform, we get
$$ \fl \left( HDSx \right) = H(x' + \Delta'x), \quad \| \Delta'x \| \leq \gamma_{\log_2(n)+2} \|x'\| \leq \gamma_{\log_2(n)+2} \cdot \sqrt{\frac{n}{\ell}} \cdot (1+\gamma_3) \cdot \|x\|$$
hence denoting $\Delta_2 x := S^{-1} D \cdot \Delta'x$, we get
$$\fl \left( HDSx \right) = HDS(x + \Delta_1 x + \Delta_2 x), \quad \| \Delta_2 x \| \leq \gamma_{\log(n)+2}(1+\gamma_3) \|x\|$$
Summing the two errors,
$$\| \Delta_1 x  \Delta_2 x \| \leq \Bigl(\gamma_3 + \gamma_{\log_2(n)+2}(1+\gamma_{3}) \Bigr) \cdot \|x\| \leq \gamma_{\log_2(n)+5} \|x\|$$
Finally, the uniform sampling is exact, which concludes the proof.

\vspace{0.6cm}

\textbf{Proof of~\Cref{thm:Qbis}} This is also proven by straightforward induction on $j$, however this derivation is somewhat more complicated and we detail it here. For simplicity, we show the computation for the scaling where for all $j \in \{1, \cdots m\}$, $\| \Omega u_j \| = \sqrt{2}$. Also for simplicity, we write $H_j := H(u_j, \Omega, j) = P(u_j, \Psi_j)$ for all $j \in \{1, \hdots , m\}$. For the case $m = 2$, multiplying $H_1$ and $H_2$, we observe that
$$H_1 H_2 = I_n - U_2 T_2 \mathrm{ut}\left( (\Omega U_2)^t \Omega \right),$$
where $\mathrm{ut}$ denotes the upper triangular part and
$$ T_2 = \left[ \begin{matrix} 1 & - \langle \Omega u_1, \Omega u_2 \rangle \\ 0 & 1 \end{matrix} \right].$$
Let us then suppose that 
$$H_1 \cdots H_j = I - U_j T_j \mathrm{ut}\left( (\Omega U_j)^t \Omega \right),$$
for some matrix $T_j \in \R^{j \times j}$. Multiplying by $H_{j+1}$ on the right-side and using associativity, we obtain
\begin{align*} & \; \; H_1 \cdots H_j H_{j+1} = \\ & I_n - u_{j+1} (\Omega u_{j+1})^t \left[ 0_{\ell \times j} \; \; \Omega_{j+1} \right] - U_j T_j \mathrm{ut} \left( (\Omega U_j)^t \Omega \right) + U_j T_j \mathrm{ut}\left( (\Omega U_j )^t \Omega \right) u_{j+1} (\Omega u_{j+1})^t \left[ 0_{\ell \times j} \; \; \Omega_{j+1} \right].\end{align*}
We first observe that
$$ (\Omega u_{j+1})^t \left[ 0_j \; \; \Omega_{j+1} \right] = \left[ 0_j \; \; (\Omega u_{j+1})^t \Omega_{j+1} \right].$$
Since the first $j$ coordinates of $u_{j+1}$ are zero, we also have
$$ \mathrm{ut}\left( ( \Omega U_j)^t \Omega \right) u_{j+1} = (\Omega U_j)^t \Omega u_{j+1},$$
so that the total composition writes
\begin{align}
\resizebox{\textwidth}{!}{$
I_n - u_{j+1} \left[0_{\ell \times j} \; \; (\Omega u_{j+1})^t \Omega_{j+1} \right] - U_j T_j \left[ \begin{matrix} (\Omega u_1)^t \Omega & &  \\ 0_{1 \times 1} & (\Omega u_2)^t \Omega_2 & \\  \vdots & & \\ 0_{1 \times (j-1)} & & (\Omega u_j)^t \Omega_j \end{matrix} \right] + U_j \left( T_j (\Omega U_j)^t \Omega u_{j+1} \right) \left[ 0_{1 \times j} \; \; (\Omega u_{j+1})^t \Omega_{j+1} \right]$
},
\end{align}
which can be factored as
$$ H_1 \cdots H_j H_{j+1} = I - U_{j+1} \left[ \begin{matrix} T_j & - T_j (\Omega U_j)^t \Omega u_{j+1} \\ 0_{1 \times j} & 1 \end{matrix} \right] \mathrm{ut}\left( (\Omega U_{j+1})^t \Omega \right).$$
The proposition follows by induction.

\vspace{0.6cm}

\textbf{Proof of~\Cref{prop:deltawoodburry}} :
Apply the Woodburry-Morrison formula to $\mathcal{H}_j$ and $\mathcal{H}_j^{-1}$ and derive
    \begin{align}\label{eq:woodburryS_notstablebis}
        \forall j \in \{1, \hdots , m \}, \quad \tilde{T}_j^t = - \left[ I_j - T_j \mathrm{ut} \left( (\Omega U_j)^t \Omega \right) U_j \right]^{-1} T_j.
    \end{align}
Set
$$\Delta_j =  (\Omega U_j)^t \Omega U_j - \mathrm{ut}\left( (\Omega U_j)^t \Omega \right) U_j$$
then inject $\Delta_j$ in~\cref{eq:woodburryS_notstablebis} to find
\begin{align} \label{eq:woodburrySbis} \tilde{T}_j^t = (T_j^{-t} - \Delta_j)^{-1}.\end{align}
Finally, use~\cref{eq:woodburryTbis}.

\end{document}